[1994/12/01]
\documentclass[10pt, reqno]{amsart}
\usepackage{a4wide}
\usepackage[english, activeacute]{babel}
\usepackage{amsmath,amsthm,amsxtra}
\usepackage{epsfig}
\usepackage{amssymb}
\usepackage{latexsym}
\usepackage{amsfonts}
\usepackage{hyperref}
\usepackage{pdftricks}

\begin{psinputs}
\usepackage{pst-plot}
\end{psinputs}

\title{Dynamics of soliton-like solutions for slowly varying, generalized gKdV equations: refraction vs. reflection}
\author{Claudio Mu\~noz C.}
\address{Universit\'e de Versailles Saint-Quentin-en-Yvelines \\ LMV-UMR 8100, 45 av. des Etats-Unis, 78035 Versailles cedex, France}
\email{Claudio.Munoz@math.uvsq.fr}
\date{September, 2010}
\subjclass[2000]{Primary 35Q51, 35Q53; Secondary 37K10, 37K40}
\keywords{gKdV equations, integrability theory, soliton dynamics, slowly varying medium, reflection, refraction}
\thanks{This research was supported in part by a CONICYT-Chile and an \emph{Allocation de Recherche} grants}

\chardef\bslash=`\\ 

\newtheorem{thm}{Theorem}[section]

\newtheorem{lem}[thm]{Lemma}
\newtheorem{prop}[thm]{Proposition}

\theoremstyle{definition}
\newtheorem{defn}{Definition}[section]

\theoremstyle{remark}
\newtheorem{rem}{Remark}[section]

\newtheorem{Cl}{Claim}

\numberwithin{equation}{section}

\newcommand{\R}{\mathbb{R}}
\newcommand{\N}{\mathbb{N}}

\newcommand{\la}{\lambda}

\newcommand{\al}{\alpha}
\newcommand{\ga}{\gamma}

\newcommand{\supp}{\operatorname{supp}}

\def\bm{\left( \begin{array}{cc}}
\def\endm{\end{array}\right)}

 \providecommand{\abs}[1]{\lvert#1 \rvert}
 \providecommand{\norm}[1]{\lVert#1 \rVert}

\newcommand{\ve}{\varepsilon}

\newcommand{\be}{\begin{equation}}
\newcommand{\ee}{\end{equation}}
\newcommand{\ba}{\begin{equation*}}
\newcommand{\ea}{\begin{equation*}}
\newcommand{\bea}{\begin{eqnarray}}
\newcommand{\eea}{\end{eqnarray}}
\newcommand{\bee}{\begin{eqnarray*}}
\newcommand{\eee}{\end{eqnarray*}}
\newcommand{\ben}{\begin{enumerate}}
\newcommand{\een}{\end{enumerate}}
\newcommand{\nonu}{\nonumber}

\newcommand{\eval}[2][\right]{\relax
  \ifx#1\right\relax \left.\fi#2#1\rvert}

\let\abs=\envert

\let\norm=\enVert

\begin{document}
\begin{abstract}
In this work we continue our study of the description of the soliton-like solutions of the variable coefficients, subcritical gKdV equation   
$$
u_t + (u_{xx} -\la u  + a(\ve x) u^m )_x =0,\quad \hbox{ in } \quad \R_t\times \R_x,  \quad m=2,3 \hbox{ and } 4,
$$
with $0\leq \la<1$, $1<a(\cdot )<2$ a strictly increasing, positive and asymptotically flat potential, and $\ve$ small enough. In \cite{Mu2} we proved the existence (and uniqueness in most of the cases) of a \emph{pure} soliton-like solution $u(t)$ satisfying
$$
\lim_{t\to -\infty}\|u(t) - Q(\cdot -(1-\la)t) \|_{H^1(\R)} =0, \quad 0\leq \la<1,
$$
provided $\ve$ small enough. Here $R(t,x) := Q_c(x-(c-\la)t)$ is the standard $H^1$-soliton solution of $R_t + (R_{xx} -\la R + R^m)_x =0$. In addition, this solution is global in time and satisfies, for all $0<\la\leq\frac{5-m}{m+3}$,
\be\label{MU2}
 \sup_{t\gg \frac 1\ve }\|u(t) - 2^{-1/(m-1)}Q_{c_\infty}(\cdot -\rho(t)) \|_{H^1(\R)} \leq K\ve^{1/2},
\ee
for suitable scaling and translation parameters $c_\infty(\la)\geq 1$ and $\rho'(t) \sim (c_\infty -\la) $, and $K>0$. In the cubic case, $m=3$, this result also holds for $\la=0$.  

\medskip

\noindent
The purpose of this paper is the following. We give an \emph{almost complete} description of the remaining case $\frac{5-m}{m+3}<\la<1$. Surprisingly, there exists a fixed, positive number $\tilde \la \in (\frac{5-m}{m+3} ,1)$, independent of $\ve$, such that the following alternative holds:

\smallskip

\begin{enumerate}
\item \emph{Refraction}. For all $\frac{5-m}{m+3}<\la<\tilde \la$, the soliton solution behaves as in \cite{Mu2}, and satisfies (\ref{MU2}), but now $\la <c_\infty<1 $, and $\rho'(t) \sim c_\infty -\la >0$.
\item \emph{Reflection}. If $\tilde \la <\la<1$, then the soliton-like solution is reflected by the potential and it satisfies
$$
\sup_{t\gg \frac 1\ve }\|u(t) - Q_{c_\infty}(\cdot -\rho(t)) \|_{H^1(\R)} \leq K\ve^{1/2}.
$$
with $0<c_\infty <\la$, and $\rho'(t) \sim c_\infty-\la <0$. This last is a completely new type of soliton-like solution for gKdV equations, also present in the NLS case \cite{Mu1}.

\end{enumerate}

\medskip

Moreover, for any $0<\la<1$, with $\la\neq \la$, the solution is not pure as $t\to +\infty$, in the sense that
$$
\limsup_{t\to +\infty}\|u(t) - \kappa(\la)Q_{c_\infty}(\cdot -\rho(t)) \|_{H^1(\R)}>0,
$$
with $\kappa(\la) $ depending on $\la$.

\end{abstract}
\maketitle \markboth{Dynamics of soliton solutions for perturbed gKdV equations} {Claudio Mu\~noz}
\renewcommand{\sectionmark}[1]{}

\section{Introduction and Main Results}

\medskip

In this work we continue our study of the dynamics of a soliton for some generalized Korteweg-de Vries equations (gKdV), started in \cite{Mu2}. In that paper the objective was the study of the global behavior of a \emph{generalized soliton solution} for the following subcritical, variable coefficients gKdV equation:
\be\label{aKdV0}
u_t + (u_{xx} -\la u  + a(\ve x) u^m )_x =0,\quad \hbox{ in } \quad \R_t\times \R_x,  \quad m=2,3 \hbox{ and } 4.
\ee
Here $u=u(t,x)$ is a real-valued function, $\ve>0$ is a small number, $\la\geq 0$ a fixed parameter, and the \emph{potential} $a(\cdot )$ a smooth, positive function satisfying some specific properties, see (\ref{ahyp}) below. 
 
 \smallskip
 
The above equation represents in some sense a simplified model of \emph{long dispersive waves in a channel with variable depth}, which takes in account \emph{large} variations in the shape of the solitary wave. The primary physical model, and the dynamics of a generalized soliton-solution, was formally described by Karpman-Maslov, Kaup-Newell, Asano, and Ko-Kuehl \cite{KM1,KN1,Asano,KK}, with further results by Grimshaw \cite{Gr1}, and Lochak \cite{Lo}. From a mathematical point of view, an additional objective was the study of perturbations of integrable systems, in this case the KdV equation ($m=2$). See \cite{Mu2, New} and references therein for a detailed physical introduction to this model.

\smallskip

The main novelty in the works above cited was the discovery of a \emph{dispersive tail} behind the soliton, with small height but large width, as a consequence of the lack of conserved quantities such as mass or energy. However, no mathematically rigorous proof of this phenomenon was given. 

\smallskip

In addition, from the mathematical point of view, equation (\ref{aKdV0}) is a variable coefficients version of the gKdV equation
\be\label{gKdV}
u_t + (u_{xx}- \la u +u^m)_x =0, \quad \hbox{ in } \quad \R_t\times \R_x;\quad m\geq 2 \hbox{ integer}.
\ee

This last equation is important due to the existence of localized, exponentially decaying and smooth solutions called \emph{solitons}. Given real numbers $x_0$ (=the translation parameter), and $c>0$ (=the scaling), solitons are solutions of (\ref{gKdV}) of the form
\be\label{(3)}
u(t,x):= Q_c(x-x_0-(c-\la)t), \quad  \hbox{ with } \quad Q_c(s):=c^{\frac 1{m-1}} Q(c^{1/2} s),
\ee  
and where $Q$ is the unique --up to translations-- function satisfying the second order nonlinear ordinary differential equation
\be\label{soliton}
Q'' -Q + Q^m =0, \quad Q>0, \quad Q\in H^1(\R).
\ee
In this case, this solution belongs to the Schwartz class and it is explicitly given by the formula
$$
Q(x) = \Big[ \frac{m+1}{2\cosh^2(\frac {(m-1)}2 x)}\Big]^{\frac 1{m-1}}.
$$
In particular, if $c>\la$ the solution (\ref{(3)}) represents a \emph{solitary wave}\footnote{In this paper we will not make any distinction between soliton and solitary wave, unlike in the mathematical-physics literature.}, of scaling $c$ and \emph{velocity} $(c-\la)$, defined for all time moving to the right without \emph{any change} in shape, velocity, etc. In other words, a soliton represents a \emph{pure}, traveling wave solution with \emph{invariant profile}. In addition, this equation allows soliton solutions with negative velocities, moving to the left direction, provided $c<\la$. Finally, for the case $c=\la$, one has a stationary soliton solution, $Q_\la(x)$. These two last solutions do not exist in the standard model of gKdV (namely when $\la=0$.)  In this sense, the dynamics of (\ref{gKdV}) is richer than the usual inviscid gKdV equation.

\medskip

Coming back to (\ref{aKdV0}), the corresponding Cauchy problem has been considered in \cite{Mu2}; in particular, we showed global well-posedness for $H^1(\R)$ initial data, even in the absence of some standard conserved quantities. The proof of this result is an adaptation of the fundamental work of Kenig, Ponce and Vega \cite{KPV}, in addition to the introduction of some new monotone quantities. See Proposition \ref{Cauchy} below for more details.

\smallskip

One fundamental question related to (\ref{gKdV}) is how to {\bf generalize} a soliton-like solution to more complicated models. In \cite{BL}, the existence of soliton solutions for generalized KdV equations with suitable autonomous nonlinearities was established. However, very little is known in the case of an inhomogeneous nonlinearity, as in the case of (\ref{aKdV0}). In a general situation, no elliptic, time-independent ODE can be associated to the soliton solution, unlike the standard autonomous case studied in \cite{BL}. Other methods are needed.

\smallskip

Concerning some time dependent, generalized KdV and mKdV equations ($m=2$ and $m=3$), Dejak-Jonsson, and Dejak-Sigal \cite{SJ,DS} studied the dynamics of a soliton for not too large times, of $O(\ve^{-1})$. Recently, Holmer  \cite{H} has improved some of the Dejak-Sigal results in the KdV case, up to the Ehrenfest time $O(|\log \ve |\ve^{-1})$. In their model, the perturbation is of linear type, which do not allow large variations on the soliton shape, different to the scaling itself. 

\smallskip

Finally, in \cite{Mu2} we described the soliton dynamics for \emph{all time} in the case of the time independent, perturbed gKdV equation (\ref{aKdV0}). In order to state this last result, and our present main results, let us first describe the framework that we have considered for the potential $a(\cdot)$ in (\ref{aKdV0}). 

\bigskip

\noindent
{\bf Setting and hypotheses on $a(\cdot)$}. Concerning the function $a$ in (\ref{aKdV0}), we assume that $a\in C^3(\R)$ and there exist constants $K, \ga>0$ such that
\be\label{ahyp} 
\begin{cases}
1< a(r) < 2, \quad a'(r)>0,  \quad | a^{(k)}(r)| \leq K e^{-\ga|r|}, \quad \hbox{ for all } r\in \R, \\
0<a(r) -1 \leq  Ke^{\ga r}, \; \hbox{ for all } r\leq 0, \, \hbox{ and} \\
0<2-a(r)\leq K e^{-\ga r} \; \hbox{ for all } r\geq 0.
\end{cases}
\ee
In particular, $\lim_{r\to -\infty}a(r) = 1$ and $\lim_{r\to +\infty} a(r) = 2$. The choice (1 and 2) here do not imply a loss of generality, it just simplifies the computations. In addition, we assume the following hypothesis: there exists $K>0$ such that for $m=2,3$ and $4$,
\be\label{3d1d}
| (a^{1/m})^{(3)}(s) | \leq K (a^{1/m})'(s), \quad \hbox{ for all } \quad s\in \R.
\ee
This condition is generally satisfied, however $a'(\cdot)$ must not be a compactly supported function. In addition, note that (\ref{aKdV0}) formally behaves as a gKdV equation (\ref{gKdV}), with constant coefficients $1$ and $2$, as $x\to \pm \infty$.

\medskip

Let us remark some important facts about (\ref{aKdV0}) (see \cite{Mu2} for more details.) First, this equation is not  invariant anymore under scaling and spatial translations. Moreover, a nonzero solution of (\ref{aKdV0}) \emph{might lose or gain some mass}, depending on the sign of $u$, in the sense that, at least formally, the quantity
\be\label{Ma}
M[u](t):= \frac 12\int_\R u^2(t,x)\,dx    \qquad \hbox{ (= {\bf mass})}
\ee
satisfies the identity
\be\label{dMa}
 \partial_t M[u](t) = -\frac{\ve}{m+1} \int_\R a'(\ve x) u^{m+1}(t,x)dx.
\ee
On the other hand, the {\bf energy}
\be\label{Ea}
E_a [u](t) :=  \frac 12 \int_\R u_x^2(t,x)\,dx + \frac \la 2 \int_\R u^2(t,x)\, dx - \frac 1{m+1}\int_\R a(\ve x)  u^{m+1}(t,x)\,dx
\ee
remains formally constant for all time. Recall that these quantities are conserved for $H^1$-solutions of (\ref{gKdV}), inside the corresponding interval of existence. 

\medskip

In addition, there exists another conservation law, valid only for solutions with enough decay at infinity:
\be\label{L1}
 \int_\R u(t,x)dx = \hbox{constant}.
\ee

\smallskip

Now let us describe what we mean by a soliton-like solution of (\ref{aKdV0}). Indeed, in \cite{Mu2} we introduced the concept of \emph{pure generalized soliton-solution} for (\ref{aKdV0}), of size $c=1$ and velocity $1-\la>0$.

\medskip

\begin{defn}[Pure generalized soliton-solution for (\ref{aKdV0}), \cite{Mu2}]\label{PSS}~

Let $0\leq \la <1$ be a fixed number. We will say that (\ref{aKdV0}) admits a \emph{pure} generalized soliton-like solution (of scaling equals $1$ and initial velocity equals $1-\la>0$) if there exist a $C^1$ real valued function $\rho=\rho(t)$, defined for all large time, and a global in time $H^1(\R)$ solution $u(t)$ of (\ref{aKdV}) such that 
\bea
\label{menos}\lim_{t\to - \infty}\|u(t) -Q(\cdot -(1-\la)t)\|_{H^1(\R)} & =& 0,\\
\label{mas} \lim_{t\to +\infty} \big\|u(t) - 2^{-1/(m-1)}Q_{c_\infty} (\cdot - \rho(t)) \big\|_{H^1(\R)} & =& 0,
\eea
with $\lim_{t\to +\infty} \rho(t) =+\infty$, and where $c_\infty= c_\infty(\la)>0$ is the scaling \emph{suggested by the energy conservation law}.
\end{defn}

\begin{rem}
The above definition describes a soliton-like solution being \emph{completely pure} at both $t\to \pm \infty$. Note e.g. that the standard soliton $Q(x -(1-\la)t) $ is a pure soliton solution of (\ref{gKdV}), with invariant profile and no dispersive behavior.  The coefficient $2^{-1/(m-1)}$ in front the soliton solution in (\ref{mas}) comes from the fact that (\ref{aKdV0}) behaves like the standard gKdV equation
$$
u_t + (u_{xx} -\la u +2u^m)_x =0, 
$$
as $x\to +\infty$.

However, in this definition we do not consider a possible case of a \emph{reflected} soliton,
$$
\lim_{t\to + \infty}\|u(t) -Q_{c_\infty}(\cdot - \rho(t))\|_{H^1(\R)}  = 0, \quad \lim_{t\to + \infty}\rho(t)= -\infty.
$$
\end{rem}

\begin{rem}[On the scaling $c_\infty$]
Let us explain in more detail the main argument --based in the energy conservation law--, to determine the scaling $c_\infty(\la)$. Let $u(t)$ be a pure soliton solution, as in Definition \ref{PSS}. Then one has 
$$
E_a[u](-\infty)  = (\la -\la_0) M[Q],
$$
with $\la_0$ given by
\be\label{l0}
\la_0:= \frac{5-m}{m+3} \in(0,1),
\ee
(cf. Appendix \ref{IdQ} for the details.) On the other hand, one has
$$
E_a[u](+\infty) = \frac{c_\infty^{\frac 2{m-1} -\frac 12}(\la)}{2^{\frac 2{m-1}}}( \la -c_\infty(\la) \la_0 ) M[Q]; 
$$
for $c_\infty =c_\infty(\la)$. From the energy conservation law one obtains
$$
\frac{c_\infty^{\frac 2{m-1} -\frac 12}(\la)}{2^{\frac 2{m-1}}}( \la -c_\infty(\la) \la_0 ) =  \la -\la_0,
$$
that is,
\be\label{cinfty}
c_\infty^{\la_0}(c_\infty -\frac{\la}{\la_0} )^{1-\la_0} = 2^{\frac 4{m+3}}(1-\frac{\la}{\la_0} )^{1-\la_0}.
\ee
In \cite{Mu2} we proved the existence of a unique solution $c_\infty(\la)\geq 1$ of this last  algebraic equation, for all $0\leq \la \leq \la_0$ (see Lemma \ref{ODE} for more details.) Moreover, the application $\la \mapsto c_\infty(\la) $ is a smooth \emph{decreasing} map with $c_\infty(0)=2^{4/(m+3)}$ and $c_\infty(\la_0) =1$. However, note that (\ref{cinfty}) is valid only under the assumptions (\ref{menos})-(\ref{mas}). In particular, if there exists a reflected soliton, it should obey a different scaling law. 
\end{rem}

\begin{rem}[Balance of mass]
Note that a pure generalized soliton-like solution may loss almost one half of its mass during the interaction. Indeed, a simple computation shows that the mass at infinity is given by 
$$
M[u](-\infty) = M[Q], \quad M[u](+\infty) = 2^{-2/(m-1)}c_\infty^{ 2/(m-1) -1/2} M[Q].
$$
Since $c_\infty(\la)$ is a decreasing map in $\la$ (see preceding remark), one has e.g.
\bee
M[u](+\infty)|_{\la=0} & = & 2^{\frac 4{m+3} -\frac 2{m-1}} M[Q] , \\
 M[u](+\infty)|_{\la =\la_0}&  = &  2^{-\frac 2{m-1}} M[Q].
\eee
\end{rem}

\bigskip

\subsection*{Description of the dynamics} Let us be more precise. By assuming the validity of (\ref{ahyp}) and (\ref{3d1d}), we proved, among other things, the following result.

\begin{thm}[Dynamics of solitons for gKdV under slowly varying medium, see \cite{Mu2}]\label{MT}~

Suppose $m=2,3$ and $4$, and let $0\leq \la <1$ be a fixed number. Consider $\la_0$ as in (\ref{l0}), and $c_\infty(\la)$ satisfying (\ref{cinfty}).
There exists a small constant $\ve_0>0$ such that for all $0<\ve<\ve_0$ the following holds. 

\medskip

\begin{enumerate}

\item \emph{Existence of a soliton-like solution}. 

\noindent
There exists a solution $u\in C(\R, H^1(\R))$ of (\ref{aKdV0}), global in time, such that 
\be\label{Minfty}
\lim_{t\to -\infty} \|u(t) - Q(\cdot -(1-\la)t) \|_{H^1(\R)} =0, 
\ee
with conserved energy $E_a[u](t) = (\la-\la_0)M[Q].$ This solution is \emph{unique} in the following cases: $(i)$ $m=3$; and $(ii)$ $m=2,4$, provided $\la>0$.

\medskip

\item \emph{Interaction soliton-potential and refraction}.  

\noindent
Suppose now in addition that $0< \la\leq\la_0$ for the cases $m=2,4$, and $0\leq \la \leq \la_0$ if $m=3$.\footnote{This is the case of nonpositive energy.} There exist constants $K,T, c^+>0$ and a $C^1$-function $\rho(t)$, defined in $[T,+\infty)$, such that 
$$
w^+(t,\cdot) := u(t,\cdot) - 2^{-1/(m-1)} Q_{c^+} (\cdot -\rho(t)) 
$$   
satisfies
\begin{enumerate}
\item \emph{Stability and asymptotic stability}. For any $t\geq T$,
\be\label{MT2}
\|w^+(t)\|_{H^1(\R)} +|\rho'(t) -(c_\infty(\la) -\la)| + |c^+ -c_\infty(\la)| \leq K\ve^{1/2};
\ee
and for some fixed $0<\beta <\frac 12 (c_\infty-\la)$,\footnote{Recall that $c_\infty(\la)\geq 1$ for all $0\leq \la\leq \la_0$.} depending on $\ve$, 
\be\label{MT3}
\lim_{t\to +\infty } \|w^+(t)\|_{H^1(x>\beta t)} =0.
\ee
\item \emph{Bounds on the scaling parameter}. Define $\theta := \frac 1{m-1} -\frac 14>0.$ One has, for all $\la>0$,
\be\label{Pc2}
 \frac 1K \limsup_{t\to +\infty }\|w^+(t)\|_{H^1(\R)}^2 \leq  \big( \frac{c^+}{c_\infty} \big)^{2\theta} -1 \leq K \liminf_{t\to +\infty} \|w^+(t)\|_{H^1(\R)}^2.
\ee
\end{enumerate}
\end{enumerate}
\end{thm}

\medskip

The proof of this result, and in particular of (\ref{MT2}), requires the introduction of an approximate solution, up to first order in $\ve$. Roughly speaking, the solution $u(t)$ behaves like a well modulated soliton-solution, plus a small order term, namely
\be\label{tildeuu}
u(t,x) \sim \mu(t) Q_{c(t)} (x-\rho(t)) + \ve \nu(t) A_{c(t)} (x-\rho(t)),
\ee
where $c(t), \rho(t)$ are the scaling and position parameters, and $\mu(t), \nu(t)$ $A_c$ are  unknown functions, to be found.
In \cite{Mu2} we proved that this description is a good approximation of the dynamics, provided $(c,\rho)$ follow a well defined dynamical system, of the form (cf. Lemma \ref{ODE0} for more details):
\be\label{firstorder0}
\begin{cases}
c'(t) \sim \ve f_1(t), \quad c(-T_\ve) \sim 1,\\
\rho'(t) \sim c(t) -\la, \quad \rho(-T_\ve) \sim -(1-\la)T_\ve,
\end{cases}
\ee   
for a given function $f_1(t)>0$ and some well defined time $T_\ve \gg \frac 1\ve$ (see (\ref{Te}) for a precise definition.) Therefore, the infinite dimensional dynamics reduces to a simple finite dimensional problem, which describes the main properties of the soliton solution. Once this system is well understood, the main problem reduces to an advanced form of  stability argument, in the spirit of Weinstein, and Martel-Merle \cite{We,MMcol1}.

\medskip

\begin{rem}[On the order of the error term in (\ref{MT2})]~

\noindent
Note that $u(t)$ behaves like an \emph{almost} pure soliton solution, in the sense of Definition \ref{PSS}, up to an error of order $\ve^{1/2}$ in $H^1(\R)$. A first sight, the order of magnitude of this term may appear somehow strange. However, it can explained by the existence of a {\bf dispersive tail} behind the soliton, formally found by physicists in \cite{KN1}. This tail is  mathematically described by the function $A_c$ in (\ref{tildeuu}). Indeed, one can see (cf. Proposition \ref{prop:decomp}), that $A_c$ is an almost flat function, with support of size $O(\ve^{-1})$. From this fact, it is clear that 
$$
\| \ve A_c (\cdot - \rho(t)) \|_{H^1(\R)} \leq K\ve^{1/2}, \quad \| \ve A_c (\cdot - \rho(t))\|_{L^1(\R)} =O(1).
$$
Note that this bound holds even for the cubic case, $m=3$, which makes a big difference with the model studied in \cite{SJ}. In that paper the authors found an upper bound of order $\ve$. We believe that our upper bound is bigger due to the shape variation experienced by the soliton, which is not present in the theory developed by \cite{SJ}.  
\end{rem}

\begin{rem}
\emph{Stability} (\ref{MT2}) and \emph{asymptotic stability} (\ref{MT3}) of solitary waves for gKdV equations as stated in the above Theorem have been widely studied since the '80s. The main ideas of our proof are classical in the literature. For more details, see e.g. \cite{Benj,BSS,MMT,PW}.
\end{rem}

\medskip

In addition, by using a contradiction argument and the $L^1$-conservation law (\ref{L1}), it was proved that no soliton-like solution exist in this regime:

\medskip

\begin{thm}[Non-existence of pure soliton-like solution for \ref{aKdV0}, \cite{Mu2}]\label{MTcor}~

Under the context of Theorems \ref{MT}, suppose $m=2,3,4$ with $0<\la\leq \la_0$. There exists $\ve_0>0$ such that for all $0<\ve<\ve_0$, 
\be\label{MTcor1}
\limsup_{t\to +\infty} \|w^+(t)\|_{H^1(\R)} >0. 
\ee
\end{thm}

\begin{rem}
Let us explain in some words the proof of this last theorem. The proof it is mainly based in an argument introduced in \cite{Me}, in a completely different context. We suppose that $\lim_{t\to +\infty} \|w^+(t)\|_{H^1(\R)} =0$. Using a monotonicity argument, one can show that, for any $\la>0$, the convergence is indeed exponentially in time:
$$
\sup_{t\gg \ve^{-1}}\|w^+(t)\|_{H^1(\R)} \leq K e^{-\ve \ga t},
$$  
up to a small modulation parameter in the space variable. This time decay can be traduced in space decay via a new monotonicity formula, which allows to define the integral of $w^+(t)$ as $t\to +\infty$, and proves that it is small. Using the $L^1$-conservation law, and comparing the result obtained at both $t\sim \pm \infty$, we obtain the desired contradiction. 

\medskip

However, this argument \emph{does not give} a quantitative lower bound on the size of the defect $w^+(t)$, as $t\to +\infty.$
\end{rem}

\begin{rem}\label{NE}
From the proof of this result in \cite{Mu2}, we emphasize that the same conclusion in Theorem \ref{MTcor} holds {\bf for any} $0<\la<1$  if we assume the validity of (\ref{MT2}) -(\ref{MT3}) for all $t$ large enough, after some minor modifications (cf. Section \ref{5}.)
\end{rem}

Summarizing, Theorems \ref{MT} and \ref{MTcor} can be represented in the following figure:

\medskip
 
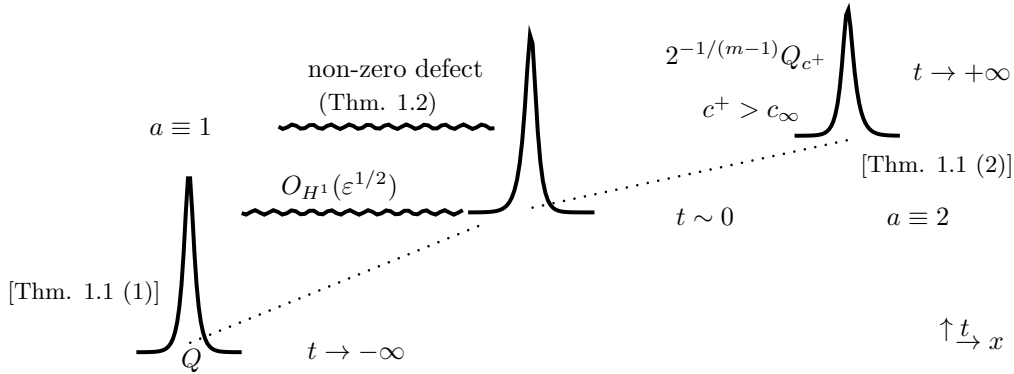
\begin{pdfpic} 
\begin{pspicture}(0,-1)(16,12)
\psset{xunit=0.7cm,yunit=0.6cm}
\psplot[linecolor=black, linewidth=1.5pt]{0}{2}{1.357 4.48 x 1    sub 10 mul exp 1 4.48 x 1 sub 10 mul exp div add 0.6667 exp div 4.652 mul 0.0 add  }

\psplot[linecolor=black, linewidth=1.5pt]{6.3}{8.7}{1.357 4.48 x 7.5   sub 10  mul exp 1 4.48 x 7.5 sub  8 mul exp div add 0.6667 exp div 4.652 mul 0.1 add 3 add }

\psplot[linecolor=black, linewidth=1.5pt]{12.5}{14.5}{0.957 4.48 x 13.5    sub 8 mul exp 1 4.48 x 13.5 sub 10 mul exp div add 0.6667 exp div 4.652 mul 0.1 add 4.7 add  }

\psplot[linecolor=black, linewidth=1.5pt]{2.7}{6.8}{ x 1000 mul sin  20 div  5.0 add}

\psplot[linecolor=black, linewidth=1.5pt]{2}{6.2}{ x 1000 mul sin  20 div  3.1 add}

\uput[0](15,0.5){$\uparrow t$}
\uput[0](15.3,0.2){$\rightarrow x$}
\uput[0](0.6,-0.2){$Q$}
\uput[0](9.8,6.6){$ 2^{-1/(m-1)}Q_{c^+}$}
\uput[0](10.5,5.4){$c^+ > c_\infty$}
\uput[0](3.2,5.5){\small (Thm. 1.2)}
\uput[0](3,6.3){non-zero defect}
\uput[0](14,3){$a\equiv 2$}
\uput[0](0,5){$a\equiv 1$}
\uput[0](2.5,3.7){$O_{H^1}(\varepsilon^{1/2})$}
\uput[0](10,3){$t\sim 0$}
\uput[0](14.5,6.2){$t\rightarrow +\infty$}
\uput[30](13.5,3.8){\small [Thm. 1.1 (2)] }
\uput[30](-2.7,1){\small [Thm. 1.1 (1)]}
\uput[0](3,0){$t\rightarrow -\infty$}

\psline[linewidth=1pt,linecolor=black,linestyle=dotted](1,0.2)(6.5,2.8) 
\psline[linewidth=1pt,linecolor=black,linestyle=dotted](7.5,3.2)(13.5,4.70) 
\end{pspicture}
\end{pdfpic}

\medskip

A first important question left open in \cite{Mu2} was the behavior of the solution $u(t)$ from Definition \ref{PSS} in the case of {\bf positive energy}, namely $\la_0<\la<1$. The analysis in this case requires more attention due the fact that the scaling of the soliton solution \emph{decreases} as long as the interaction soliton-potential takes place. This behavior is in part a consequence of the competition between the strength of the potential and the initial kinetic energy. In this paper our first objective is to describe in detail that case. Indeed, in the next paragraphs we will state the following surprising result: given a fixed $\la$ close to $1$, for any small $\ve>0$ the soliton is {\bf reflected} by the potential $a(\ve \cdot)$.  This result is basically a consequence of the fact that, given $0<\la<1$ and $c>0$ fixed, with $c<\la$, the small soliton $Q_c(\cdot -(c-\la)t)$, solution of 
$$
u_t + (u_{xx} -\la u + u^m)_x =0, \quad \hbox{ in } \; \R_t \times \R_x, 
$$
moves towards the {\bf left}.

\subsection*{Main Results}
 
Let us recall the setting of our problem. Let $0<\la< 1$ be a fixed parameter, consider the equation
\be\label{aKdV}
\begin{cases}
u_t + (u_{xx} -\la u + a (\ve x) u^m)_x =0 \quad \hbox{ in \ } \R_t \times \R_x,   \\
m=2,3 \hbox{ and } 4;\quad  0< \ve\leq\ve_0;  \quad a(\ve \cdot) \hbox{ satisfying } (\ref{ahyp}) \hbox{-}(\ref{3d1d}). 
\end{cases}
\ee

\medskip 
Here $\ve_0>0$ is a small parameter. Under these hypotheses, our main results are as follows. First, we describe the dynamics of interaction soliton-potential. Let $\tilde \la  =\tilde \la(m)$ be the unique solution of the algebraic equation
\be\label{tlan}
\tilde \la (\frac {1-\la_0}{\tilde \la -\la_0})^{1-\la_0} =2^{\frac 4{m+3}}, \quad \la_0<\tilde \la<1, \quad \la_0 \hbox{ given by } (\ref{l0}). 
\ee
(See Lemma \ref{ODE} for more details.) We claim that this number represents a sort of \emph{equilibrium} between the energy of the solitary wave and the strength of the potential. Indeed, first we prove that the dynamics in the case $\la_0<\la<\tilde \la$ is similar to that of \cite{Mu2}.

\begin{thm}[Interaction soliton-potential and refraction, case $\la_0<\la<\tilde \la$]\label{MTL1}~

Suppose $\la_0< \la <\tilde \la$. There exists $\ve_0>0$ such that for all $0<\ve<\ve_0$ the following holds. There exist constants $K, \tilde T, c^+, c_\infty(\la)>0$, with $\la <c_\infty(\la) <1$; and a smooth function $ \rho(t) \in \R $ such that the function 
$$
w^+ := u(t) - 2^{-\frac 1{m-1}} Q_{c^+} (\cdot-  \rho(t))
$$ 
satisfies  for all $t\geq \tilde T$,
\be\label{St1l}
\| w^+(t) \|_{H^1(\R)} + | \rho'(t) - c_\infty(\la) +\la | +|c^+-c_\infty| \leq K\ve^{1/2},
\ee
and 
$$
\lim_{t\to + \infty} \|w^+(t)\|_{H^1(x>\beta t)} =0,
$$
for a fixed $0<\beta <\frac 12 (c_\infty(\la) -\la)$.  Moreover, for $\theta := \frac 1{m-1} -\frac 14$, 
\be\label{Pc2l}
 \frac 1K \limsup_{t\to +\infty }\|w^+(t)\|_{H^1(\R)}^2 \leq  \big( \frac{c^+}{c_\infty} \big)^{2\theta} -1 \leq K\ve.
\ee
\end{thm}
\smallskip
\noindent
Note that this generalized soliton solution behaves, as $t\to +\infty$, as a solitary wave with velocity $\sim c_\infty -\la >0$, but {\bf smaller} than the initial one ($=1-\la$). 

\medskip

Now we consider the case $\tilde \la <\la<1$. Here a completely new behavior is present. The soliton solution is, in this case, a reflected solitary wave.

\begin{thm}[Interaction soliton-potential and reflection, case $\tilde \la<\la<1$]\label{MTL2}~ 

Suppose now $\tilde \la<\la<1$, with $\ve>0$ small enough. Then there exist constants $K, \tilde T, c^+, c_\infty(\la)>0$, with $0 <c_\infty(\la) <\la$; and a smooth function $ \rho(t) \in \R $ such that 
$$
w^+ := u(t) -  Q_{c^+} (\cdot-  \rho(t))
$$
satisfies  for all $t\geq \tilde T$,
\be\label{St1m}
\| w^+(t) \|_{H^1(\R)} + |\rho'(t) - c_\infty(\la) +\la | +|c^+ -c_\infty| \leq K\ve^{1/2},
\ee
and 
$$
\lim_{t\to + \infty} \|w^+(t)\|_{H^1(x>\beta t)} =0,
$$
for a fixed $\beta \in ( -\la, c_\infty(\la)-\la )$. Finally,
\be\label{Pc2m}
 \frac 1K \limsup_{t\to +\infty }\|w^+ (t)\|_{H^1(\R)}^2 \leq  \big( \frac{c_\infty}{c^+} \big)^{2\theta} -1 \leq K\ve.
\ee

\end{thm}

\medskip

Some few remarks are in order.

\begin{rem}
Note that in (\ref{Pc2m}) the final scaling $c^+$ is smaller or equal than $c_\infty(\la)$.\footnote{In Theorem \ref{MTL3} we will prove that it is actually smaller.} This is a big surprise, present in the case of a reflected soliton. In particular, it differs from the results found in the recent literature (compare with the results found in \cite{MMcol1, MMcol2, MMfin, Mu1}.)
\end{rem}

\begin{rem}[More on the literature]
We believe that Theorem \ref{MTL2} is the first completely rigorous result showing the existence and global description of a reflected solitary wave under a slowly varying potential; in this case for gKdV equations. Preliminary, formal results in this direction can be found in \cite{WM1,WM2,WM3,WM4,WM5,DS,H}.
\end{rem}

\begin{rem}[Notation]
With a slight abuse of notation, we have denoted by $w^+(t)$, $\rho(t)$, $c^+$, etc. some different functions or parameters (cf. Theorems \ref{MT}, \ref{MTL1} and \ref{MTL2}.) However, since the range of validity of each definition depends on $\la$, and each region of validity in $\la$ is pairwise disjoint, we have chosen this method, in order to simplify the notation.
\end{rem}

\begin{rem}
Note that the coefficients in front of $Q_{c^+}$ in Theorems \ref{MTL1} and \ref{MTL2} are different since the potential $a(\cdot)$ behaves in a different way depending on $x\to \pm \infty$.
\end{rem}

\begin{rem}[Remaining mass in the case of a soliton reflected]
In the case $\tilde \la <\la<1$, and compared with the case $0<\la<\tilde \la$, the equation for the parameter $0<c_\infty(\la)<1$ is now given by (cf. (\ref{cinf2}))
$$
\la -\la_0 = c_\infty^{\frac 2{m-1}-\frac 12}(\la)(\la-\la_0 c_\infty(\la)),
$$
that is,
$$
 c_\infty^{\la_0} (\frac{\la}{\la_0} -c_\infty)^{1-\la_0} = (\frac{\la }{\la_0} -1)^{1-\la_0}
$$
(compare with (\ref{cinfty}), and see also Lemma \ref{ODE2} for more details.) In addition the final mass is given now by the quantity
$$
M[u](+\infty) =c_\infty^{\frac 2{m-1}-\frac 12}(\la)M[Q] < \la^{\frac 2{m-1}-\frac 12}M[Q],
$$
modulo an error of order at most $\ve$.
\end{rem}

\begin{rem}[Case $\la=\tilde \la$]
The behavior of the solution in the case $\la =\tilde \la $ remains an interesting open problem. It seems that in the case $\la=\tilde \la$ the solution $u(t)$ behaves asymptotically at infinity as an almost bounded state of the form $2^{-1/(m-1)}Q_{\tilde \la} (x-\rho(t))$, for some $\rho'(t)$ small and close to zero. See Lemma \ref{ODE1} and Remark \ref{tildelala} for more details.
\end{rem}

\noindent
{\bf Main ideas in the proof of Theorems \ref{MTL1} and \ref{MTL2}.} Similar to \cite{Mu2}, the proof of this result is based in a detailed description of the behavior of a finite dimensional \emph{dynamical system}, for the case $\la_0<\la<1$, which leads to the different behaviors above mentioned. Indeed, from \cite{Mu2}, one has that the \emph{scaling} $c(t)$ and the \emph{translation} $\rho(t)$ associated to the soliton solution, satisfy, at the first order in $\ve$, the dynamical system (\ref{firstorder0}), now for a given function $f_1(t)<0 $. Since $c'(t) <0$ for all $t\geq -T_\ve$, the scaling is a {\bf decreasing} quantity in time. The key point is then the following: a necessary condition to obtain a \emph{reflected} soliton is that $\rho'(t)<0$ after some point, in other words, $c(t)<\la$. Therefore we need to check the values of $\la$ for which the scaling $c(t)$ satisfies $ c(t)>\la$ for all $t\geq -T_\ve$, or $c(t_0) =\la$ for some $t_0> -T_\ve$. After some computations, it turns out that the sharp parameter deciding between these two regimes is given by $\tilde \la$ in (\ref{tlan}) (see Remark \ref{tildla} for more details.) In addition, we have to prove that $c(t)$ remains far from zero for all time, which is not direct since $c(t)$ is always a decreasing quantity.

\begin{rem}
From the above results we do not discard the existence of small solitary waves traveling \emph{to the left} (since a small soliton moves to the left), at least for the case $m=2$. In the cubic and quartic cases, we believe there are no such soliton solutions. 
\end{rem}

\medskip

Finally, we prove that there is no pure soliton solution at both sides of time.

\begin{thm}[Inelastic character of the interaction soliton-potential]\label{MTL3}~

Suppose $\la_0<\la<1$, with $\la\neq \tilde \la.$ Then one has
$$
\limsup_{t \to +\infty}\|w^+(t)\|_{H^1(\R)} >0,
$$
in particular $c^+>c_\infty(\la)$ for $0<\la<\tilde \la$, and $c^+<c_\infty(\la)$ in the case $\tilde \la<\la<1$.
\end{thm} 

From Remark \ref{NE}, the proof of this result is a consequence of estimates (\ref{St1l})-(\ref{St1m}), and the same argument developed for the proof of Theorem \ref{MTcor} in \cite{Mu2}.

\medskip

The following figure illustrates the behavior stated in Theorems \ref{MTL1}, \ref{MTL2} and \ref{MTL3}.

\bigskip

\begin{pdfpic}
\begin{pspicture}(10cm,10cm)
        \rput(8,8){
                \psset{xunit=1cm,yunit=1.2cm}
          
               
                \parametricplot[plotpoints=200,linecolor=black, linestyle=dashed]{-3}{3}{0  2 t exp 1 2 t exp div   add   ln 1 add sub     t}
                \psline[linewidth=1pt,linecolor=black,linestyle=dashed](-3,-3)(3.5,2.8)

               \psplot[linecolor=black, linewidth=1.5pt]{-4.3}{-1.9}{1.357 4.48 x 3.1  add 5 mul exp 1 4.48 x 3.1 add 5 mul exp div add 0.6667 exp div 4.652 mul 0.1 add 4 div 3 sub}

                 \psplot[linecolor=black, linewidth=1.5pt]{-2.9}{-0.5}{
                1.357 4.48 x 1.7  add 5 mul exp 1 4.48 x 1.7 add 5 mul exp div add 0.6667 exp div 4.652 mul 0.1 add 4 div}  

                \psplot[linecolor=black, linewidth=1.5pt]{-4.3}{-1.9}{1.357 4.48 x 3.1  add 5 mul exp 1 4.48 x 3.1 add 5 mul exp div add 0.6667 exp div 4.652 mul 0.1 add 4 div 3 add}
               \psplot[linecolor=black, linewidth=1.5pt]{5.3}{2.9}{1.357 4.48 x 4.1  sub 5 mul exp 1 4.48 x 4.1 sub 5 mul exp div add 0.6667 exp div 4.652 mul 0.1 add 4 div 3 add}
               
               \psplot[linecolor=black, linewidth=1.5pt]{-7}{-5}{ x 1000 mul sin  30 div  3 add}
               \psplot[linecolor=black, linewidth=1.5pt]{-5.5}{-4}{ x 1000 mul sin  30 div }
                 \uput[0](-7.1,2.7){$H^1$-defect $>0$}
                 \uput[0](-7.1,3.3){\small(Thm. 1.5)}
                 \uput[0](-3.8,-1.7){$Q$}
                 \uput[0](0.8,4.5){$2^{-1/(m-1)}Q_{c^+}$,  $\lambda<c^+<1$}
                 
                 \uput[0](1.3,3.3){\small(Thm. 1.3)}
                 \uput[0](-4.3,4.5){$Q_{c^+}$,  $0<c^+ <\lambda$}
                 \uput[0](-2.5,3.3){\small(Thm. 1.4)}
                 \uput[0](-0.5,-2.9){$t\to -\infty$}
                 \uput[0](2.5,1.5){case $\lambda_0 <\lambda<\tilde \lambda$}
                 \uput[0](-5.5,1.5){case $\tilde \lambda<\lambda<1$}
                 \uput[0](1.5,0){$t\sim 0$}
                 \uput[0](-0.5,3.1){$t\to +\infty$}
                 \uput[0](-5.8,0.5){$O_{H^1}(\varepsilon^{1/2})$}
                 \uput[0](-6.8,-1){$a\equiv 1$}
                 \uput[0](3.8,-1){$a\equiv 2$}
                 \uput[0](-6.9,-2.6){$\uparrow t$}
                 \uput[0](-6.8,-2.8){$\rightarrow x$}
                }
        }
\end{pspicture}
\end{pdfpic}

\bigskip

A second important open question from \cite{Mu2} and this paper is to establish a lower bound 
on the defect $w^+(t)$ as the time goes to infinity, at least in the case $0<\la<1$, $\la\neq \tilde \la$ (the cases $\la=0$ and $\la=\tilde \la$ seem harder.) We expect to treat this problem in a forthcoming paper (see \cite{Mu4}.) For the moment, and based in some formal computations (cf. Proposition \ref{prop:decomp} and Remark \ref{4.4}) we claim that 
\be\label{LBBB}
\liminf_{t\to +\infty} \|w^+(t)\|_{H^1(\R)} \geq K\ve^{p_m}, \quad \hbox{ with } \; p_2=p_4 = 1, \; p_3 =2.
\ee

\medskip

\begin{rem}[The Schr\"odinger case]
The interaction soliton-potential has be also considered in the case of the nonlinear Schr\"odinger equation with a slowly varying potential, or a soliton-defect interaction. See e.g. Gustafson et al. \cite{GFJS, FGJS}, Gang and Sigal \cite{GS}, Gang and Weinstein \cite{GW}, and Holmer, Marzuola and Zworski \cite{HZ,HMZ0, HMZ}, for more details. See also our recent work \cite{Mu1} on soliton dynamics for a modified NLS equation. 
\end{rem}

\medskip

\noindent
{\bf Notation.} In this paper, both $K,\ga>0$ will denote fixed constants, independent of $\ve$, and possibly changing from one line to another. Let us define, for $m=2,3$ and $4$,
\be\label{Muu}
\mu =\mu(\la) := \frac{99}{100}(1-\frac{\la_0}{\la})^{\frac{1-\la_0}{\la_0}} .
\ee
Since $\la_0<\la<1$, this number is always a positive quantity, less than 1. In addition, let us define, for $\ve>0$ small,
\be\label{Te}
 T_\ve := \frac {\ve^{-1 -\frac 1{100}}}{1-\la}>0.
\ee
Third, we consider the unperturbed energy
\be\label{E0}
E_1[u](t) := \frac 12\int_\R u_x^2(t) +\frac \la 2 \int_\R u^2(t) -\frac 1{m+1} \int_\R u^{m+1}(t),
\ee
namely $E_1[u] = E_{a\equiv 1} [u]$. 

\medskip

Finally, we denote by $\mathcal{Y}$ the set of $C^\infty$ functions $f$ such that for all $j\in \N$ there exist $K_j,r_j>0$ such that for all $x\in \R$ we have
\be\label{Y}
|f^{(j)}(x)|\leq K_j (1+|x|)^{r_j} e^{- \frac 12\mu|x|}.
\ee

\medskip

\noindent
{\bf Plan of this work.} 

\noindent
Let us explain the organization of this paper. First, in Section \ref{2} we introduce some basic tools to study the interaction problem, and state several important asymptotic results. In Section \ref{2a} we study a finite-dimensional dynamical system which describes the dynamics in a approximative way. Next, in section \ref{3} we describe the interaction soliton-potential, based in the construction of an approximate solution (see Appendix \ref{A} for that computation). Finally in Section \ref{5} we prove the main results of this article.

\bigskip

\noindent
{\bf Acknowledgments}. I wish to thank Y. Martel and F. Merle for presenting me this problem and for  their continuous encouragement and support, during the elaboration of this work; and the DIM members at Universidad de Chile for their kind hospitality, and where part of this work was written. 

\bigskip

\section{Preliminaries}\label{2}

The purpose of this section is to recall some important properties needed through this paper. For more details or the proof of these results, see Section 2 and 3 in \cite{Mu2}.

\subsection{The Cauchy problem} 

First we recall the following local well-posedness result for the Cauchy problem associated to (\ref{aKdV}). 

Let $u_0\in H^s(\R)$, $s\geq1$, $ \la > 0$.  We consider the following initial value problem 
\be\label{Cp1}
\begin{cases}
u_t + (u_{xx} -\la u + a(\ve x) u^m)_x = 0 \quad \hbox{ in } \R_t \times \R_x \\
u(t=0) =  u_0,
\end{cases}
\ee
where $m=2,3$ or $4$. The equivalent problem for the generalized KdV equations (\ref{gKdV}) has been studied e.g.  in \cite{KPV}. We have the following result.

\begin{prop}[Local and global well-posedness, see \cite{KPV} and Proposition 2.1 in \cite{Mu2}]\label{Cauchy}~

\begin{enumerate}
\item \emph{Local well posedness in $H^s(\R)$}. 

\noindent
Suppose $u_0\in H^s(\R),$ $s\geq 1$. Then there exist a maximal interval of existence $I$ $($with $0\in I)$, and a unique (in a certain sense) solution $u\in C(I, H^s(\R))$ of (\ref{Cp1}). In addition, for any $t\in I$ the energy  $E_a[u](t)$ from (\ref{Ea}) remains constant, and the mass $M[u](t)$ defined in (\ref{Ma}) satisfies (\ref{dMa}).

\medskip

\item \emph{Global existence in $H^1(\R)$, $\la>0$}. 

\noindent
Suppose now $u_0\in H^1(\R)$, and $\la>0$. Then $I$ is of the  form $I=(\tilde t_0 , +\infty)$, for some $-\infty \leq \tilde t_0<0$; and there exists $\ve_0>0$ small such that 
$$
\sup_{t\geq 0} \|u(t)\|_{H^1(\R)}  \leq K.
$$
Finally, suppose $u_0\in L^1(\R)\cap H^1(\R)$. Then (\ref{L1}) is well defined and remains constant for all $t\in I$.

\end{enumerate}

\end{prop}

\medskip

\begin{rem}
In order to prove item (2) in the above result, we introduced in \cite{Mu2} a modified mass, \emph{decreasing} in time. Indeed, consider for all $t\in I$, $m=2,3$ and $4$,
\be\label{hM}
\hat M[u](t) := \frac 12 \int_\R a^{1/m}(\ve x) u^2(t,x) dx.
\ee
Then for any $m=2,3$ and $4$, and for all $t\in I$ we have
\be\label{hM2}
\partial_t \hat M[u](t)  = -\frac 32 \ve \int_\R (a^{1/m})'(\ve x) u_x^2 - \frac \ve 2 \int_\R [ \la (a^{1/m})'  -\ve^2 (a^{1/m})^{(3)}] (\ve x) u^2 . 
\ee
In conclusion, from (\ref{3d1d}) there exists $\ve_0>0$ such that for all $0<\ve\leq \ve_0$ and for all $t\geq 0$, one has 
\be\label{hM3}
\hat M[u](t) \leq \hat M[u](0).
\ee
The global existence follows from the subcritical nature of the nonlinearity $(m<5)$.
\end{rem}

\medskip

We will also need some properties of the corresponding linearized operator of (\ref{aKdV}). All the results here presented are by now well-known, see for example \cite{MMcol1}.

\subsection{Spectral properties of the linear gKdV operator}

In this paragraph we consider some important properties concerning the linearized KdV operator associated to (\ref{aKdV}). Fix $c>0$,  $m=2,3$ or 4, and let
\be\label{defLy}
    \mathcal{L} w(y) := - w_{yy} + c w - m Q_c^{m-1}(y) w, \quad\hbox{ where }\quad  Q_c(y) := c^{\frac 1{m-1}} Q(\sqrt{c} y).
\ee
Here $w=w(y)$. We also denote $\mathcal L_0 := \mathcal L_{c=1}$. 

\medskip

\begin{lem}[Spectral properties of $\mathcal{L}$, see \cite{MMcol2}]\label{surL}~

The operator $\mathcal{L}$ defined (on $L^2(\R)$) by \eqref{defLy}  has domain $H^2(\R)$, it is self-adjoint and satisfies the following properties:
\begin{enumerate}
\item \emph{First eigenvalue}. There exist a unique $\lambda_m>0$ such that  $\mathcal{L}  Q_c^{\frac {m+1}2} =-\lambda_m Q_c^{\frac {m+1}2} $. 
\item The kernel of $\mathcal{L}$ is spawned by $Q'_c$. Moreover,
\be\label{LaQc}
\Lambda Q_c := \partial_{c'} {Q_{c'}}\big|_{ c'=c} = \frac 1c \Big[\frac 1{m-1} Q_c + \frac 12 xQ'_c \Big],
\ee
satisfies $\mathcal{L} (\Lambda Q_c)=- Q_c$. Finally, the continuous spectrum of $\mathcal L$ is given by $\sigma_{cont}(\mathcal L) =[c,+\infty)$.
\item \emph{Inverse}. For all   $h=h(x) $ \emph{ polynomially growing}  function such that $\int_\R h Q_c'=0$, there exists a unique  \emph{ polynomially growing}  function $\hat h $   such that $\int_\R \hat hQ'_c=0$ and $\mathcal{L} \hat h = h$. Moreover,  if $h$ is even (resp. odd), then $\hat h$ is even (resp. odd).

\item \emph{Regularity in the Schwartz space $\mathcal S$}. For $h\in H^2(\mathbb{R})$,  $\mathcal{L} h \in \mathcal{S}$ implies $h\in \mathcal{S}$.

\item\label{6a} \emph{Coercivity}. 

\begin{enumerate}
\item
There exists  $K,\sigma_c>0$ such that for all $w\in H^1(\R)$
$$
 \mathcal B[w,w]   :=   \int_\R (w_y^2+c w^2 - mQ_c^{m-1} w^2)  \geq  \sigma_c \int_\R w^2-K\abs{\int_\R  w Q_c}^2 - K\abs{\int_\R w Q_c'}^2.
$$
In particular, if $\displaystyle{\int_\R  w Q_c = \int_\R w Q_c'=0,}$ then the functional $\mathcal B[w,w]$ is positive definite in $H^1(\R)$. 
\item Now suppose that $\displaystyle{\int_\R  w Q_c = \int_\R w yQ_c=0}$. Then the same conclusion as above holds.
\end{enumerate}
\end{enumerate}
\end{lem}

\bigskip

\subsection{Construction of a soliton-like solution}

Let us recall the following result of existence and uniqueness of a \emph{pure} soliton-like solution for (\ref{aKdV}) for $t\to -\infty$, valid {\bf for any fixed} $0\leq \la<1$.

\begin{prop}[Existence and uniqueness of a pure soliton-like solution, \cite{Mu2}]\label{Tm1}~

Suppose $0\leq  \la <1$ fixed. There exists $\ve_0>0$ small enough such that the following holds for any $0<\ve < \ve_0$. 
There is a solution $u \in C(\R, H^1(\R))$ of (\ref{aKdV}) such that 
\be\label{lim0}
\lim_{t\to -\infty} \|v (t) - Q(\cdot -(1-\la)t) \|_{H^1(\R)} =0,
\ee
and energy $E_a[u](t) = (\la -\la_0)M[Q] .$
Moreover, there exist constants $K,\ga>0$ such that  for all time $t\leq -\frac 1{2}T_\ve$, \footnote{with $T_\ve$ defined in (\ref{Te})},
\be\label{minusTe}
\|u(t) - Q(\cdot -(1-\la)t) \|_{H^1(\R)} \leq  K\ve^{-1} e^{\ve \ga t}.
\ee
In particular, 
\be\label{mTep}
\|u(-T_\ve) - Q(\cdot + (1-\la)T_\ve) \|_{H^1(\R)} \leq K\ve^{-1} e^{- \ga \ve^{-\frac 1{100}}} \leq K \ve^{10},
\ee
provided $0<\ve<\ve_0$ small enough.

Finally, this solution is unique for all $\la>0$, and in the case $\la=0$, $m=3.$
\end{prop}

\begin{rem}
Note that the energy identity above follows directly from (\ref{lim0}), Appendix \ref{IdQ} and the energy conservation law from Proposition \ref{Cauchy}.
\end{rem}

\medskip

The proof of this Proposition is standard and follows the work of Martel \cite{Martel}, where the existence of a unique $N$-soliton solution for gKdV equations was established. Although there exist possible different proofs of this result, the method employed in \cite{Martel} has the advantage of giving an explicit uniform bound in time (cf. (\ref{minusTe})). This bound is indeed consequence of some compactness properties. 

\bigskip

\subsection{Stability and asymptotic stability results for large time} 

In order to prove the stability properties contained in Theorems \ref{MTL1} and \ref{MTL2}, we recall the following result, proved in \cite{Mu2} for the case $0<\la\leq\la_0$, but still valid for any fixed $0<\la<1$ and $c_\infty>0$, satisfying $\la<c_\infty$. 

\medskip

\begin{prop}[Stability and asymptotic stability in $H^1$, see \cite{Mu2}]\label{Tp1}~

Let $m=2,3$ and $4$, and let $0< \la<1$, $c_\infty > \la$. Let $0<\beta< \frac 12(c_\infty -\la)$ be a fixed number. There exists $\ve_0>0$ (depending on $\beta$) such that if $0<\ve <\ve_0$ the following hold. 

\medskip

Suppose that for some time $t_1\geq \frac 12 T_\ve$ and $t_1 \leq X_0 \leq 2t_1$, one has
\be\label{18}
\big\| u(t_1) - Q_{c_\infty} (x - X_0) \big\|_{H^1(\R)} \leq  \ve^{1/2},
\ee
where $u(t)$ is a $H^1$-solution of (\ref{aKdV}). Then $u(t)$ is defined for every $t\geq t_1$ and there exists $K, c^+>0$ and a $C^1$-function $ \rho(t)$ defined in $[t_1,+\infty)$ such that 
\begin{enumerate}
\item \emph{Stability}.
\be\label{S}
 \sup_{t\geq t_1} \big\| u(t) -  2^{-1/(m-1)} Q_{c_\infty} (\cdot - \rho(t)) \big\|_{H^1(\R)} \leq K \ve^{1/2},
\ee
where 
$$
|\rho(t_1)+ X_0 | \leq  K\ve^{1/2},  \quad \hbox{ and for all } t\geq t_1,\quad  |\rho'(t)-c_\infty+\la | \leq K\ve^{1/2}.
$$
\item \emph{Asymptotic stability}. One has
\be\label{AS}
\lim_{t\to +\infty} \big\| u(t) - Q_{c^+} (\cdot - \rho(t)) \big\|_{H^1(x> \beta t)} =0.
\ee
In addition,
\be\label{liminfinity}
\lim_{t\to +\infty} \rho'(t) = c^+-\la , \qquad |c^+ - c_\infty | \leq K\ve^{1/2}. 
\ee
\end{enumerate}
\end{prop}

\begin{rem}
In other words, the above result states that once the soliton has crossed the interaction region, it behaves like a \emph{standard} soliton of a gKdV equation, and stability and asymptotic stability hold. Let us recall that, from \cite{Mu2}, this result is valid for any $m=2,3$ or $4$, and any $0<\la<1$, provided $c_\infty>\la$. In addition, it is still valid for $m=3$ and $\la=0$.
\end{rem}

\begin{rem}
Let us recall that the hypothesis $c_\infty>\la$ is essential; otherwise the soliton should have negative velocity $(=c_\infty -\la)$ and it would return to the interaction region. Indeed, the original proof in \cite{Mu2} falls to be correct since the quantity $(c_\infty(\la)-\la) \hat M[u](t)$ in the Weinstein functional is no longer decreasing. Later, in Lemma \ref{ODE}, we will see that  $c_\infty(\la) $, as introduced in Theorem \ref{MTL1}, satisfies $c_\infty(\la)>\la$ for any $0\leq \la <\tilde \la$ (cf. (\ref{tlan}).)
\end{rem}

\medskip

In order to prove the global stability result of Theorem \ref{MTL2}, we will need a version of the above Proposition for the case of a \emph{reflected soliton}, namely when $\tilde\la <\la<1$. Let us recall that, given $0<\la<1$ and $c>0$ fixed, with $c<\la$, the \emph{small} soliton $Q_c(\cdot -(c-\la)t)$, solution of 
$$
u_t + (u_{xx} -\la u + u^m)_x =0, \quad \hbox{ in } \; \R_t \times \R_x, 
$$
moves towards the left. 

\medskip

\begin{prop}[Stability and asymptotic stability in $H^1(\R)$, reflection case]\label{Tp1r}~

Suppose $m=2,3$ or $4$. Let $0< \la<1$ and $c_\infty>0$ be such that $c_\infty <\la$. Let $-\la<\beta < c_\infty -\la $. There exists $\ve_0>0$ such that if $0<\ve <\ve_0$ the following hold.

\smallskip

\noindent
Suppose that for some time $t_1\geq K T_\ve$ and $t_1 \leq X_0 \leq 2t_1$
\be\label{18b}
\big\| u(t_1) - Q_{c_\infty} (x + X_0) \big\|_{H^1(\R)} \leq  \ve^{1/2}.
\ee
where $u(t)$ is a $H^1$-solution of (\ref{aKdV}). 
Then $u(t)$ is defined for every $t\geq t_1$, and there exists $K>0$ and a $C^1$-function $ \rho(t)$, defined in $[t_1,+\infty)$, such that for all $t\geq t_1$,
\begin{enumerate}
\item \emph{Stability}.
\be\label{Sb}
\| u(t) - Q_{c_\infty} (\cdot - \rho(t)) \|_{H^1(\R)} + |\rho(t_1)+ X_0 | +  | \rho'(t)- c_\infty +\la | \leq K\ve^{1/2}.
\ee
\item \emph{Asymptotic stability}. There exists $c^+>0$ such that
\be\label{ASb}
\lim_{t\to +\infty} \big\| u(t) - Q_{c^+} (\cdot - \rho(t)) \big\|_{H^1(x> \beta t)} =0.
\ee
In addition,
\be\label{liminfinityb}
\lim_{t\to +\infty} \rho'(t) = c^+ - \la , \qquad |c^+ - c_\infty | \leq K\ve^{1/2}. 
\ee
\end{enumerate}
\end{prop}

The proof of this statement requires several new ideas, in particular, the introduction of a modified mass, {\bf almost increasing} in time. It turns out that these requirements are satisfied e.g. by the quantity 
\be\label{Mback}
\mathcal M[u](t) :=\int_\R \frac {u^2(t,x)}{2 a(\ve x)} dx.
\ee
Summarizing, the stability theory requires the introduction of two different, almost monotone masses, depending on $c_\infty$. Indeed, if
$$
\begin{cases}
c_\infty >\la \implies \hbox{ we use } \hat M[u](t), \; \hbox{ (cf. (\ref{hM}))}, \\
c_\infty < \la \implies \hbox{ we use } \mathcal M[u](t).
\end{cases}
$$
\begin{proof}[Proof of Proposition \ref{Tp1r}]
See Appendix \ref{Stab}. 
\end{proof}

\bigskip

Let us finish this section with a result concerning the \emph{dynamical system} associated to the parameters of the soliton solution.

\subsection{Existence of approximate dynamical parameters}

In this paragraph we recall the existence of a unique solution to the dynamical system found in \cite{Mu2}, and involving the evolution of the first order \emph{scaling} and \emph{translation} parameters of the soliton solution, $(C(t), P(t))$, in the interaction region. The behavior of this solution is essential to understand the actual dynamics of the soliton solution inside this region.

Let  us fix some notation. Denote, for $C>0$ and $P\in \R$ given,
\be\label{f1}
f_1(C,P) := p \ C(C-\frac \la{\la_0} ) \frac{a'(\ve P)}{a(\ve P)},  \quad p:=\frac4{m+3},
\ee
with $\la_0$ as in (\ref{l0}). 

\begin{lem}[Existence and basic properties of dynamical parameters, see \cite{Mu2}]\label{ODE0}~

Suppose $m=2,3$ or $4$, $\la_0, a(\cdot), T_\ve$ be as in Theorem \ref{MT}, (\ref{ahyp}) and (\ref{Te}). Let $0\leq \la\leq \la_0$. There exists a unique solution $(C(t), U(t))$, with $C$ bounded positive, monotone increasing, defined for all $t\geq -T_\ve$, of the following system  
\be\label{c00}
\begin{cases}
C'(t)  = \ve f_1(C(t),P(t)), \qquad C(- T_\ve) = 1, \\
P'(t) = C(t) -\la, \qquad P(-T_\ve) =-(1-\la)T_\ve.
\end{cases}
\ee
In addition, one has $1\leq C(t)< C(+\infty) = c_\infty(1+O(\ve^{10}))$, and $(1-\la) t \leq P(t) \leq \frac {101}{100}(c_\infty -\la) t $, with $c_\infty=c_\infty(\la)$ being the unique positive solution of 
\be\label{cinf}
c_\infty^{\la_0}(c_\infty-\frac{\la}{\la_0})^{1-\la_0} = 2^p(1-\frac{\la}{\la_0})^{1-\la_0}, \quad c_\infty(\la)\geq 1.
\ee
In particular, one has $c_\infty(\la =0) =2^{p}>1$ and $c_\infty(\la=\la_0)=1$.
\end{lem}

\begin{rem}
Let us explain the importance of this result. The above lemma \emph{formally} describes the dynamics of the soliton solution by means of some approximate, finite dimensional system of the variables $C(t)$ and $P(t)$. In other words, the dynamics in the case $\ve\sim 0$ can be seen as the projection into a approximate two dimensional manifold represented by $(C(t), P(t))$.
\end{rem}

In the next section, our objective is to extend this result to the full range $\la_0< \la<1.$ In this case, from (\ref{f1}), (\ref{c00}), and the initial condition $C(-T_\ve) =1$, the scaling $C(t)$ is a {\bf decreasing function in time}. In this direction, a first key property to prove is that $C(t)$ remains far from zero independently of $\ve$. Moreover, a new behavior is possible if there exists some time $t_0$ such that $C(t_0)=\la$. In that case, the soliton should be formally reflected by the potential.

\bigskip

\section{Study of a dynamical system revisited}\label{2a}

\medskip

This section is devoted to the study of the approximate dynamical system describing the evolution of the first order \emph{scaling} and \emph{translation} parameters $(C(t), P(t))$, inside the interaction region, in the case $\la_0<\la<1$. This system shares many properties with the nonlinear system considered in \cite{Mu2} for $0\leq \la \leq \la_0$, that is Lemma \ref{ODE0}; however, the large time behavior in the case $\la_0<\la< 1$ may be completely different. In what follows, we prove, among other things, the existence and uniqueness of a suitable solution, and reflection for large enough $\la$. 

\smallskip

\begin{lem}[Existence of dynamical parameters, case $\la_0<\la<1$]\label{ODE}~

Suppose $m=2,3$ or $4$. Let $\la_0, a(\cdot)$, $p$ and $f_1$ be as in (\ref{l0}), (\ref{ahyp}) and (\ref{f1}). Then there exists $\ve_0>0$ small such that, for all $0<\ve<\ve_0$, the following holds. 

\smallskip

\begin{enumerate}
\item \emph{Existence}. 

\noindent
There exists a unique solution $(C(t),P(t))$, with $C(t)$ bounded, positive and monotone decreasing, defined for all $t\geq -T_\ve$, of the following nonlinear system  
\be\label{c}
\begin{cases}
C'(t) = \ve f_1(C(t), P(t)), \qquad C(- T_\ve) = 1, \\
P'(t) = C(t) -\la, \qquad P(-T_\ve) =-(1-\la)T_\ve.
\end{cases}
\ee
In addition for all $t\geq -T_\ve$ one has $ 0<C(t)\leq 1$ and
\be\label{boundC}
C^{\la_0 }(t)  ( \frac \la{\la_0 }- C(t) )^{1-\la_0 }   = ( \frac \la{\la_0 } -1)^{1-\la_0 } \frac{a^p(\ve P(t))}{a^p(-\ve^{-1/100})}.
\ee
Moreover, $\lim_{t\to +\infty} C(t) $ exists and satisfies $\lim_{t\to +\infty} C(t)  > \mu(\la)>0$, for all $\la_0<\la<1$ (cf. (\ref{Muu}).)

\medskip
\item \emph{Asymptotic behavior}. 

\noindent
Let $\la_0<\tilde \la<1$ be the unique number satisfying
\be\label{tl0}
\tilde \la (\frac{1-\la_0}{\tilde\la-\la_0})^{1-\la_0} =2^p.
\ee
Then,
\begin{enumerate}
\item For all $\la_0<\la  \leq  \tilde \la$, one has $\lim_{t\to +\infty} C(t) >\la$ and $\lim_{t\to +\infty} P(t) =+\infty$.

\smallskip

\item For all $\tilde \la<\la<1$, there exists a unique $ t_0\in ( -T_\ve, +\infty)$ such that $C(t_0)=\la$, with  $\lim_{t\to +\infty} C(t) < \la $. Moreover, $\lim_{t\to +\infty} P(t) =-\infty.$ Finally, one has the bound  $-T_\ve <t_0 \leq K(\la) T_\ve ,$ for a positive constant $K(\la)$, independent of $\ve$.
\end{enumerate}
\end{enumerate}
\end{lem}

\medskip

\noindent
Before the proofs, some remark are in order.

\smallskip

\begin{rem}[On the meaning of the parameter $\tilde \la$]\label{tildla}
Let us say some words about where the parameter $\tilde\la$ comes from. Indeed, since this parameter decides whether the soliton is reflected or not, a formal necessary condition is then the existence of $t_0\geq -T_\ve$ such that $C(t_0)=\la$, for $\la>\tilde \la$. Let us suppose this property. Replacing in (\ref{boundC}), we get
$$
\la (\frac{1-\la_0}{\la-\la_0})^{1-\la_0} = \frac{a^p(\ve P(t_0))}{a^p(-\ve^{-1/100})};
$$
(recall that $\la_0<\la<1$.) This is an implicit equation for $P(t_0)$. Since $1< a(\cdot)<2$, we have that if
$$
\la (\frac{1-\la_0}{\la-\la_0})^{1-\la_0} >\frac{2^p}{a^p(-\ve^{-1/100})}
$$
then there is no solution for the above equation. So, since the left hand side above does not depend on $\ve$, in order to \emph{ensure} the existence of a point $t_0$, a necessary condition is that
$$
\la (\frac{1-\la_0}{\la-\la_0})^{1-\la_0} \leq 2^p.
$$
Finally, we define $\tilde \la $ to be the \emph{worst} possible case, such that the equality is reached in the above inequality.
\end{rem}

\medskip

\begin{proof}[\bf Proof of Lemma \ref{ODE}]~

\smallskip

\noindent
{\bf 1.} The local existence of a solution $(C, P)$ of (\ref{c}) is a direct consequence of the Cauchy-Lipschitz-Picard theorem. In addition, $C\equiv 0, \frac \la{\la_0}$ are constant solutions. Since $C(-T_\ve)=1$ and $\la>\la_0$, we have $C$ globally defined, strictly decreasing and satisfying $0<C(t)<\frac \la{\la_0}$ for all $t\geq -T_\ve$.  

\medskip

\noindent
{\bf 2.} Now we use (\ref{c})-(\ref{f1}) to obtain some a priori estimates on the solution $C$. Note that 
$$
\frac{(C(t)-\la)}{C(t)( \frac \la{\la_0 } - C(t))}C'(t) = - \ve p(C(t)-\la) \frac{a'}{a}(\ve P(t)) = - \ve p P'(t) \frac{a'}{a}(\ve P(t)).
$$
In particular,
$$
(1-\la_0 )\partial_t \log ( \frac \la{\la_0 } -C(t) )  + \la_0  \partial_t \log C(t)  = p \partial_t  \log a(\ve P(t)).
$$
By integration on $[-T_\ve, t]$, and by using $C(-T_\ve) =1$, we obtain (\ref{boundC}).

Since $1\leq a \leq 2$ and $C$ is bounded we have $P$ bounded on compact sets and consequently we obtain global existence. Using $C> 0$ and (\ref{boundC}) one proves for $\ve$ small
\be\label{boundCp}
C^{\la_0}( t) \geq \frac{99}{100}(1-\frac{\la_0}{\la})^{1-\la_0} \implies C(t)  \geq \mu(\la).  \qquad \hbox{ (cf. (\ref{Muu})).}
\ee
Moreover, $\lim_{t\to +\infty } C(t)$ {\bf exists and it is always far from zero}, independent of $\ve$, as long as $\la_0<\la<1$. This proves the first part of the Lemma.

\medskip

\noindent
{\bf 3.} Now, given $\la_0<\la<1$, we study the existence of a point $t_0> -T_\ve $ such that $C(t_0)=\la$.  A priori, replacing this condition in (\ref{boundC}), we have
\be\label{cla}
\la  ( \frac 1{\la_0 }-1 )^{1-\la_0 }   = ( \frac \la{\la_0 } -1)^{1-\la_0 } \frac{a^p(\ve P(t_0))}{a^p(-\ve^{-1/100} ) }.
\ee
By choosing $\la := \la_0(1 + \delta)$, with $\delta>0$ a small number, we obtain a contradiction with the above identity. In conclusion, such a $t_0$ does not exist if $\la = \la_0(1+  \delta)$, with $\delta>0$ small. 
Moreover, let $\tilde \la \in (\la_0, 1)$ be the unique solution of (\ref{tl0}). Since the function 
$$
\la \in (\la_0,1) \mapsto f(\la) := \la(\frac{1-\la_0}{\la-\la_0})^{1-\la_0}  \in (0,+\infty)
$$ 
is strictly decreasing\footnote{More precisely, one has $$f'(\la) = -\frac{(1-\la)(1-\la_0)^{1-\la_0}}{(\la-\la_0)^{2-\la_0}}.$$}, we have $f(\la) \geq 2^p $ provided $\la_0<\la \leq \tilde \la$. Therefore, from (\ref{cla}) we have
$$
2^p  \leq f(\la)  =\frac{a^p(\ve P(t_0))}{a^p(-\ve^{-1/100})} < 2^p.
$$
In conclusion, since $f(\la) $ is independent of $\ve$, there is no $t_0 \in \R$ such that $C(t_0)=\la.$ Thus, by continuity we have $C(t) >\la $ for all $t\geq -T_\ve$ and $\lim_{+\infty} C(\cdot) \geq \la$. Moreover, if $\lim_{+\infty}C(\cdot) = \la$, we have from (\ref{boundC}), after passing to the limit,
$$
f(\la) \leq \limsup_{t\to +\infty} \frac{a^p(\ve P(t))}{a^p(-\ve^{-1/100})} <  2^p, \quad \la \leq \tilde \la, 
$$
a contradiction. Therefore, $\lim_{+\infty}C(\cdot) > \la$. Moreover, from the equation for $P$ in (\ref{c}) one has, for all $t\geq 0$,
$$
P(t) = P(-T_\ve) + \int_{-T_\ve}^{0} (C(s)-\la) ds + \int_{0}^{t} (C(s)-\la) ds  \geq  P(-T_\ve) + (C(0) -\la) t;
$$
and thus $\lim_{t\to +\infty} P(t) = +\infty$.

\medskip

\noindent
{\bf 4.} Now, let us prove that for all $\la\in (\tilde\la ,1)$ there exists $t_0\in\R$ such that $C(t_0)=\la$ (and therefore $\lim_{+\infty} C(\cdot) <\la$.) By contradiction, let us suppose $C(t)> \la$ for all $t\geq -T_\ve,$ with  $\tilde c_\infty :=\lim_{+\infty} C(\cdot) \geq \la.$

First, let us suppose $\tilde c_\infty>\la $. Thus $\lim_{+\infty}P(\cdot ) =+\infty $ and from (\ref{boundC}) we have
\be\label{KK}
\tilde c_\infty^{\la_0} ( \frac \la{\la_0 }- \tilde c_\infty )^{1-\la_0 }   = ( \frac \la{\la_0 } -1)^{1-\la_0 } \frac{2^p}{a^p(-\ve^{-1/100})}.
\ee
Since $\tilde c_\infty >\la$ one has
$$
\tilde c_\infty^{\la_0} ( \frac {\la -  \la_0 \tilde c_\infty}{\la -\la_0  })^{1-\la_0 } \leq   \max_{r \in (0,1)} r^{\la_0} ( \frac {\la -  \la_0 r}{\la -\la_0  })^{1-\la_0 }  = f(\la)<2^p,  
$$
a contradiction with (\ref{KK}) for small $\ve$.

Now we suppose $\tilde c_\infty =\la$. Here we have two possibilities: either $P_\infty :=\lim_{t\to+\infty} P(t) =+\infty$, or $P_\infty <+\infty$. For the first case, by following the preceding analysis, we have 
$$
\tilde c_\infty^{\la_0} ( \frac {\la -  \la_0 \tilde c_\infty}{\la -\la_0  })^{1-\la_0 }= f(\la)<2^p,  
$$
a contradiction with (\ref{KK}), for small $\ve$. Otherwise, from the equation of $C'(t)$ in (\ref{c}), one has
$$
\lim_{t\to+\infty} C'(t) = \lim_{t\to+\infty} \ve f_1(C(t),P(t)) =  p\ve \la^2 (1-\frac 1{\la_0}) \frac{a'(\ve P_\infty)}{a(\ve P_\infty)} \neq 0;
$$
for all $m=2,3$ and $4$. This last result contradicts the fact that $\lim_{t\to+\infty} C'(t) =\lim_{+\infty} \frac{C(t)}{t} =0$.

In conclusion, we have that there exists at least one $t_0> -T_\ve$ such that $C(t_0) = \la$. From $C'<0$ we have that such a $t_0$ is unique. 

\medskip

\noindent
{\bf 5.} We finally prove some properties of $P(t)$ in the case $\tilde \la<\la<1$. From (\ref{boundC}), one has
$$
f(\la) = \frac{a^p(\ve P(t_0))}{a^p(-\ve^{-1/100})}.
$$
Since $f(\la)\in (1,2^p)$ for fixed $\la\in(\tilde \la, 1)$, and it is independent of $\ve$, one has, for small $\ve$,
\be\label{Kala}
\abs{\ve P(t_0)} \leq K(\la);
\ee
(the constant $K$ becomes singular as $\la$ approaches $\tilde \la$ or $1$.) Therefore, from (\ref{c}) one has
$$
C'(t_0) =  -\ve p \la^2  (\frac 1{\la_0} -1) \frac{a'(\ve P(t_0))}{a(\ve P(t_0))}  \leq -\kappa(\la) \ve,  \qquad \kappa(\la) >0;
$$
and thus, for $\al>0$ small enough (but independent of $\ve$), since $C''(t) =O_{L^\infty}(\ve^2)$,
\be\label{Cal}
C(t_0 -\frac \al \ve) \geq \la  +  \kappa(\la) \al + O(\al^2) \geq \la + \frac 9{10} \kappa(\la) \al.
\ee
We use this identity to obtain
\bee
P(t_0) & = & -(1-\la) T_\ve + \int_{-T_\ve}^{t_0-\frac \al\ve} (C(s) -\la) ds +   \int^{t_0}_{t_0-\frac \al\ve} (C(s) -\la) ds \\
& \geq &-(1-\la)T_\ve + \frac 9{10} \kappa(\la) \al (t_0 -\frac \al\ve + T_\ve)   - \frac{K\al}\ve,
\eee
and therefore $t_0 \leq K(\la) T_\ve$. 

Finally, note that $P(t)$ is strictly decreasing for all $t > t_0$.  Therefore,  for all $ t\geq t_0+1$ one has  $ C(t_0+1) <\la$ and
$$
P(t) = P(t_0) + \int_{t_0}^{t_0+1} (C(s) -\la) ds  + \int_{t_0+1}^{t} (C(s) -\la) ds \leq P(t_0) +  (C(t_0+1) -\la)(t-t_0-1);
$$
thus $\lim_{t\to +\infty} P(t) = -\infty$.  The proof is complete.

\end{proof}

Some of the properties found in the above Lemma allow to introduce the following definition.

\medskip

\begin{defn}[Escape time]\label{ET}~

Suppose $\la_0 <\la \leq \tilde \la$.  Let us define the {\bf escape time} $\tilde T_\ve > -T_\ve$ such that $P(\tilde T_\ve) := -P(-T_\ve) = (1-\la)T_\ve$. 
Otherwise, if $\tilde \la <\la <1$, let us consider $\tilde T_\ve > t_0$ such that $P(\tilde T_\ve) := P(-T_\ve)= -(1-\la)T_\ve$.
\end{defn}

\medskip

The next result states that in the interval $\la_0<\la<\tilde \la$ the soliton leaves the interaction zone by the right hand side, with a well determined scaling $c_\infty(\la)\in (\la,1)$.  Moreover, the escape time is bounded by $K(\la) T_\ve$, with $K$ becoming unbounded as $\la$ approaches $\tilde \la$. 

\medskip

\begin{lem}[Asymptotic behavior, case $\la_0<\la <\tilde\la$]\label{ODE1}~

Suppose $\la_0<\la < \tilde \la $, $m=2,3$ or $4$. 

\begin{enumerate}
\item There exists  a unique solution $c_\infty =c_\infty(\la)$ of the following algebraic equation
\be\label{cinf1}
c_\infty^{\la_0 } (\frac{\la - \la_0 c_\infty}{\la-\la_0})^{1-\la_0 }  = 2^p, \quad \la < c_\infty<1.
\ee
In addition, $\la\mapsto c_\infty(\la)$ is a strictly decreasing map with  $c_\infty(\la_0) =1$ and $c_\infty(\la) > c_\infty(\tilde \la)=\tilde \la$.

\item  Let $(C(t), P(t))$ be the solution of (\ref{c}). Then $C( \tilde T_\ve) = c_\infty(\la) $, and $\tilde T_\ve \leq K(\la) T_\ve$, with $K(\la) \sim (c_\infty(\la)-\la)^{-1}$.
\end{enumerate}
\end{lem}

\begin{rem}
Note that the condition $c_\infty >\la$ is essential, because there exists another minimal branch of solutions $c_\infty^*(\la)<\la$ increasing in $\la$ with $c_\infty^*(\la_0) =0$ and $c_\infty^*(\tilde \la) =\tilde \la$. 
\end{rem}

\begin{proof}
The proof of existence and uniqueness of a solution $c_\infty(\la)$ of (\ref{cinf1}) is similar to Lemma 4.4 in \cite{Mu2}. We skip the details. 

Let $\tilde c_\infty (\la,\ve) := \lim_{+\infty} C$. From (\ref{boundC}) and $\lim_{+\infty} P = + \infty$ one has 
\be\label{limite}
\tilde c_\infty^{\la_0 } (\frac{\la - \la_0 \tilde c_\infty}{\la-\la_0})^{1-\la_0 }  = \frac{2^p}{a^p(-\ve^{-1/100})}, \quad \la < \tilde c_\infty<1.
\ee
Now let us define for $r\in (0, 1)$
$$
g(r) := r^{\la_0 } (\frac{\la - \la_0 r}{\la-\la_0})^{1-\la_0 }.
$$
Note that $g(r)$ is strictly decreasing in the interval $(\la, 1)$. In addition, from (\ref{cinf1}) and (\ref{limite}) we have
$c_\infty <\tilde c_\infty$. Moreover, from the behavior of $a$ in (\ref{ahyp}) we have $\tilde c_\infty = c_\infty + O(\ve^{10})$, for all $\ve$ small. This implies that 
\be\label{jajaja}
\tilde c_\infty(\la, \ve)-\la > c_\infty(\la)-\la >0,
\ee
uniformly for all $\ve$ small enough. On the other hand, at time $t=\tilde T_\ve$ one has
$$
C(\tilde T_\ve)^{\la_0} (\frac{\la - \la_0 C(\tilde T_\ve)}{\la-\la_0})^{1-\la_0 } = \frac{a^p(\ve^{-1/100})}{a^p(-\ve^{-1/100})}, \qquad 0<C(\tilde T_\ve) <\la,
$$
therefore $C(\tilde T_\ve) = c_\infty(\la) + O(\ve^{10})$. Moreover, 
\bee
(1-\la) T_\ve & = & P(\tilde T_\ve)  = P(-T_\ve) + \int_{-T_\ve}^{\tilde T_\ve} (C(s)-\la)ds  \\
& \geq&  -(1-\la)T_\ve + (\tilde c_\infty(\la,\ve)-\la )(\tilde T_\ve + T_\ve).
\eee
From this inequality and (\ref{jajaja}) we obtain, for all $\la_0<\la <\tilde \la$, the upper bound $\tilde T_\ve \leq K(\la)T_\ve$, with $K(\la) \sim (c_\infty(\la) -\la)^{-1}$. Note that $K(\la)$ becomes singular as $\la \uparrow \tilde \la$.
\end{proof}

\medskip

\begin{rem}\label{tildelala}
Note that $c_\infty(\tilde \la) =\tilde \la$ and therefore in the last inequality above one has, for $\la =\tilde \la$,
$$
(1-\tilde \la) T_\ve \geq -(1-\tilde\la)T_\ve + (\tilde c_\infty(\tilde \la,\ve) - c_\infty(\tilde \la) )(\tilde T_\ve + T_\ve).
$$
Since $\tilde c_\infty(\tilde \la,\ve) - c_\infty(\tilde \la) =O(\ve^{10})$ for $\ve$ small, we cannot obtain any reasonable upper bound of the time $\tilde T_\ve$ in this case. Further developments are probably necessary.
\end{rem}

\medskip

Now we consider the case $\tilde \la <\la<1$. Here we obtain the following striking result: the soliton is finally reflected by the potential. The final scaling is given by a modified parameter $0<c_\infty(\la)<1$, away from zero provided $\la\in (\tilde \la, 1)$.

\medskip

\begin{lem}[Asymptotic behavior, case $\tilde \la<\la<1$]\label{ODE2}~

Suppose $\tilde \la <\la <1$. There exists  a unique solution $c_\infty(\la) $ of the following algebraic equation
\be\label{cinf2}
c_\infty^{\la_0 } (\frac{\la - \la_0 c_\infty}{\la-\la_0})^{1-\la_0 } = 1, \quad 0< c_\infty<\la.
\ee
In addition, the map $\la \mapsto c_\infty(\la) $ is strictly increasing with $ c_\infty(\la) \geq c_\infty(\tilde \la) > \mu(\tilde \la )$, and $\lim_{\la\uparrow 1}c_\infty(\la) =1$. Finally, one has $C(\tilde T_\ve) = c_\infty(\la)$,  and $\tilde T_\ve \leq  K(\la) T_\ve$.
 \end{lem}

\begin{proof}
The proof of existence and uniqueness of a solution $c_\infty(\la)$ of (\ref{cinf2}) is similar to Lemma 4.4 in \cite{Mu2}. We skip the details. 

Let $\tilde c_\infty (\la,\ve) := \lim_{+\infty} C$. From (\ref{boundC}) and $\lim_{+\infty} P = -\infty$ one has 
$$
\tilde c_\infty^{\la_0 } (\frac{\la - \la_0 \tilde c_\infty}{\la-\la_0})^{1-\la_0 } = \frac{1}{a^p(-\ve^{-1/100})}.
$$
with $0< \tilde c_\infty<\la$. From the behavior of $a$ in (\ref{ahyp}) we have $\tilde c_\infty = c_\infty(\la) + O(\ve^{10})$, for all $\ve$ small. This implies that 
$$
\la - \tilde c_\infty(\la, \ve) \geq \frac{99}{100} (\la-c_\infty(\la) )>0,
$$
uniformly for all $\ve$ small enough. On the other hand, at time $t=\tilde T_\ve$ one has
$$
C(\tilde T_\ve)^{\la_0} (\frac{\la - \la_0 C(\tilde T_\ve)}{\la-\la_0})^{1-\la_0 } = \frac{a^p(-\ve(1-\la) T_\ve)}{a^p(-\ve^{-1/100})} =1, \qquad 0<C(\tilde T_\ve) <\la,
$$
therefore by uniqueness $C(\tilde T_\ve) = c_\infty(\la)$. 

Finally, we prove the upper bound on $\tilde T_\ve$. We have
$$
P(-T_\ve) = -(1-\la)T_\ve =  -(1-\la)T_\ve +  \int_{-T_\ve}^{\tilde T_\ve} (C(s) -\la) ds. 
$$
From here we have for $\beta>0$
\bee
0 & = & \int_{ -T_\ve}^{t_0 -\frac \beta\ve } (C(s) -\la) ds +  \int_{ t_0 - \frac \beta\ve}^{t_0 + \frac \beta\ve} (C(s) -\la) ds - \int_{ t_0 + \frac \beta\ve}^{\tilde T_\ve} (\la -C(s)) ds \\
& \leq & (1-\la)( t_0 + \frac \beta\ve + T_\ve) +  \frac{K\beta}{\ve}  - \int_{ t_0 + \frac \beta\ve}^{\tilde T_\ve} (\la -C(s)) ds.
\eee
Similarly to estimate (\ref{Cal}), one has for $\beta>0$ small, but independent of $\ve$,
\be\label{Cal1}
C(t_0 + \frac \beta\ve )  \leq \la - \nu(\la) \beta + O(\beta^2), \quad \nu(\la) >0.
\ee
Inserting this estimate above, and using the estimate on $t_0$, one has
$$
\tilde T_\ve \leq K(\la) T_\ve,
$$
as desired.
\end{proof}

\begin{rem}
In \cite{Mu2}, from a simple study of the dynamical system in the case $0\leq \la\leq \la_0$, we found that the soliton leaves the interaction region at time $t=T_\ve$. However, since the dynamics is \emph{repulsive} in the case $\la_0<\la<1$, the soliton takes more time to exit this region, either by the left hand side or the right one. Fortunately, in the case of an asymptotically flat potential, the \emph{escape time} is of the same order as $T_\ve$. Therefore, in what follows, $\tilde T_\ve$ will denote the corresponding escape time, {\bf for all} $0\leq \la <1$, $\la\neq \tilde \la$. Moreover, we know that $\tilde T_\ve \sim T_\ve$.    
\end{rem}

\bigskip

\section{Description of the interaction soliton-potential}\label{3}

\medskip

This is the main section of this paper. Here we will describe in detail (see also \cite{Mu2} for the case $0\leq \la\leq \la_0$), the dynamics of the soliton-like solution, inside the interaction region, for times $t\in [-T_\ve, \tilde T_\ve]$, and $\la_0<\la<1$, still avoiding the more difficult case $\la=\tilde \la$. In order to obtain this result, we need to construct some modulation parameters $(c(t),\rho(t))$ satisfying, up to order $\ve^{1/2}$, the dynamical system given in Lemma \ref{ODE}. Since we understand the formal behavior of the nonlinear problem for $(C(t), P(t))$, the rigorous description is reduced to the use of an advanced form of  Weinstein functional, as in \cite{MMcol1,Mu1,Mu2, MMfin} (compare with Theorem 4.1 in \cite{Mu2}.)

\medskip

Let us recall that, from (\ref{mTep}), and for all $\ve$ small enough, we have
\be\label{hypINT}
\| u(-T_\ve) - Q(\cdot + (1-\la)T_\ve) \|_{H^1(\R)}\leq K \ve^{10},
\ee
with $u(t)$ the solution constructed in Proposition \ref{Tm1}.

\medskip

\medskip

\begin{prop}[Dynamics in the interaction region, case $0\leq \la<1$, $\la\neq \tilde \la$]\label{T0}~

Suppose $0 \leq \la <1$, with $\la\neq \tilde \la$, cf. (\ref{tl0}). There exists a constant $\ve_0>0$ such that the following holds for any $0<\ve <\ve_0$.

Let $u=u(t)$ be a globally defined $H^1$ solution of (\ref{aKdV}) satisfying (\ref{hypINT}). Then one has

\medskip

\begin{enumerate}
\item \emph{Case} $0\leq \la \leq \la_0.$ \emph{(cf. \cite{Mu2})}

\noindent
There exist a number $K_0 >0 $, a final scaling $c_\infty(\la)\geq 1$ and $\rho_\ve \in \R$ such that
\be\label{INT410}
\|u( T_\ve ) - 2^{-1/(m-1)}Q_{c_\infty}( \cdot - \rho_\ve ) \|_{H^1(\R)} \leq K_0 \ve^{1/2}.
\ee
In addition, $\lim_{\la \uparrow \la_0} c_\infty(\la)=1$. Moreover, we have the bounds
\be\label{INT420}
(1-\la)T_\ve \leq \rho_\ve \leq (c_\infty(\la) -\la) T_\ve,   
\ee
valid for $\ve_0$ sufficiently small.

\smallskip

\item  \emph{Case} $\la_0<\la<\tilde \la$.  

\noindent
There exists $K_0>0 $, a final scaling $\la<c_\infty(\la)<1$ and $\rho_\ve \in \R$ such that
\be\label{INT41a}
\|u( \tilde T_\ve ) - 2^{-1/(m-1)}Q_{c_\infty}( \cdot - \rho_\ve ) \|_{H^1(\R)} \leq K_0 \ve^{1/2}.
\ee
In addition, $\lim_{\la \downarrow \la_0} c_\infty(\la)=1$, $\lim_{\la \uparrow \tilde \la} c_\infty(\la)=\tilde \la$. Moreover, we have the bounds
\be\label{INT42a}
(c_\infty(\la) -\la) T_\ve \leq \rho_\ve \leq (1-\la) T_\ve. 
\ee

\smallskip

\item \emph{Case }$\tilde \la <\la <1$.  

\noindent
Now there exists a constant $K_0>0$, a final scaling $\mu(\la) <c_\infty(\la)<\la$ and $\hat \rho_\ve \in \R$ such that
\be\label{INT41b}
\|u( \tilde T_\ve ) - Q_{c_\infty}( \cdot - \hat \rho_\ve ) \|_{H^1(\R)} \leq K_0 \ve^{1/2}.
\ee
In addition, $\lim_{\la \uparrow 1} c_\infty(\la)=1$. Finally, we have the bounds
\be\label{INT42b}
-K_1(\la) T_\ve \leq \hat \rho_\ve \leq - K_2(\la) T_\ve,   
\ee
valid for $\ve_0$ sufficiently small and some $K_1,K_2>0$.

\end{enumerate}
\end{prop}

\medskip

\begin{rem}
The first part of the above Proposition (namely, the case $0\leq \la\leq \la_0$), was proven in \cite{Mu2}. Now we give a different proof which allows us  to find, at least formally, a lower bound on the defect of the soliton-like solution. 
The proof of the two cases involved in the region $\la_0<\la<1$ is new, and requires the results of Section \ref{2a}, in particular Lemmas \ref{ODE1}  and \ref{ODE2}.  Following \cite{Mu2}, we construct an approximate solution up to certain order of accuracy, given by the power of the nonlinearity involved. This is the objective of the next subsection.
\end{rem}

\begin{rem}
Estimate (\ref{INT42b}) on $\hat \rho_\ve$ shows that the soliton solution is, at time $\tilde T_\ve (\sim T_\ve)$, outside the interaction region; moreover, it is on the left hand side. In other words, this estimate proves that the soliton is {\bf reflected} by the potential.
\end{rem}

\bigskip

\subsection{Construction of an approximate solution describing the interaction}\label{sec:2}

We look for $\tilde u(t,x)$, an approximate solution of (\ref{aKdV}), carrying out a specific structure. In particular, we construct $\tilde u$ as a suitable modulation of the soliton $Q(x-(1-\la)t)$, solution of the following gKdV equation: 
\be\label{orig}
u_t +(u_{xx} -\la u + u^m)_x =0.
\ee 
Let $t\in [-T_\ve, \tilde T_\ve]$, $c=c(t)>0$ and $\rho(t)\in \R$ be bounded functions to be chosen later, and
\be\label{defALPHA}
    y:=x-\rho(t) \quad \hbox{and} \quad     R(t,x): =\frac {Q_{c(t)}(y)}{\tilde a(\ve \rho(t))},
\ee
where $\tilde a (s) := a^{\frac 1{m-1}}(s).$ The parameter $\tilde a$ describes the shape variation of the soliton through the interaction. Concerning the parameters $c(t)$ and $\rho(t)$, we will assume that for all $t\in [-T_\ve, \tilde T_\ve]$,
\be\label{r1}
|c(t)-C(t)| +  |\rho'(t) -P'(t) |\leq \ve^{1/100}.
\ee
with $(C(t),P(t))$ given from Lemmas \ref{ODE0} and \ref{ODE}. Later we will improve these constraints by constructing parameters $(c(t), \rho(t))$ with better estimates. 

\medskip

As in \cite{Mu2}, the form of $\tilde u(t,x)$ will be the sum of the soliton plus a correction term:
\be\label{defv} 
\tilde u(t,x) :=R(t,x)+w(t,x),
\ee
where $w$ is given by
\be\label{defW}
    w(t,x):= \begin{cases} \ve d(t)A_{c} (y) , \quad \hbox{ if $m=2,4$}, \\  \ve d(t)A_{c} (y)  + \ve^2B_c(t,y), \quad \hbox{ if $m=3$}, \end{cases}
\ee
and 
\be\label{dd}
d(t) := \frac{a'}{\tilde a^m}(\ve \rho(t)).
\ee
Here $A_{c}(y)$ and   $B_c(t,y)$ are unknown functions. 

\medskip

\begin{rem}
In \cite{Mu2} we looked for an approximate solution of the form $w(t) = \ve d(t) A_c(y)$, for all $m=2,3$ and $4$. In this opportunity,  we require the inclusion of a second order term $\ve^2 B_c$ in the cubic case, in order to improve the quality of our approximation. In the other cases, namely $m=2$ and $4$, we just need to consider a unique, special choice of $A_c$ to obtain a difference with the dynamics of our solution from a hypothetical, completely pure soliton solution.  
\end{rem}

\medskip

We want to measure the size of the error produced by inserting $\tilde u$ as defined in (\ref{defW}) in the equation (\ref{aKdV}). For this, let 
\be\label{2.2bis}
S[\tilde u](t,x) := \tilde u_t + (\tilde u_{xx} -\la \tilde u +a(\ve x) \tilde u^{m})_x.
\ee
Our first result is the following

\begin{prop}[Improved decomposition of $S{[}\tilde u{]}$, see also \cite{Mu2}]\label{prop:decomp}~

Suppose $(c(t), \rho(t))$ satisfying (\ref{r1}). There exists $\ga>0$ independent of $\ve$ small, and an approximate solution $\tilde u$ of the form (\ref{defv})-(\ref{defW})-(\ref{dd}), such that for all $t\in [-T_\ve, \tilde T_\ve]$,

\medskip

\begin{enumerate}
\item The error term (\ref{2.2bis}) is given by
\bea\label{Decomp}
S[\tilde u](t,x) & = &   (c'(t) - \ve f_1(t) -\ve^2 \delta_{m,3} f_3(t))\partial_c\tilde u  \nonu\\
& & +\  (\rho'(t) -c(t)+ \la -  \ve f_2(t) -\ve^2 \delta_{m,3}f_4(t)) \partial_\rho \tilde u + \tilde S[\tilde u](t,x), 
\eea
with $\partial_\rho \tilde u:= \partial_\rho R -w_y + O(\ve^2 e^{-\ve\ga|\rho(t)|} A_c)$, and $\delta_{m,3}$ being the Kronecker's symbol $(\delta_{3,3} =1, \delta_{2,3} = \delta_{4,3} =0)$. 

\medskip

\item $A_c, B_c$ satisfy 
\be\label{Ac}
\begin{cases}
A_c, \partial_c A_c \in L^\infty(\R), \quad A_c'\in \mathcal Y,  \\
|A_c(y) | \leq K e^{-\ga y} \; \hbox{ as } y\to +\infty, \quad \lim_{-\infty} A_c  \neq 0,\\
\displaystyle{\int_\R Q_c(y) A_c(y) =\int_\R yQ_c(y) A_c(y) =0;}
\end{cases}
\ee
and for $m=3$,
\be\label{Bc}
\begin{cases}
B_c'(t,\cdot) \in L^\infty(\R),\quad  |B_c(t, y) | \leq K e^{-\ga y}e^{-\ve \ga|\rho(t)|} \; \hbox{ as } y\to +\infty, \\
|B_c(t,y)| +|\partial_c B_c(t,y)| \leq K|y| e^{-\ve \ga|\rho(t)|}, \; \hbox{ as } y\to -\infty, \\
\displaystyle{\int_\R Q_c(y) B_c(y) =\int_\R yQ_c(y) B_c(y) =0.}
\end{cases}
\ee

\medskip

\item In addition, $f_1 (t)=f_1(c(t),\rho(t))$ is given by (\ref{f1}),
\be\label{f2}
f_2(t) = f_2(c(t), \rho(t)) := - \frac{\xi_m}{\sqrt{c(t)}} (\la - 3 \la_0c(t)) \frac{a'}{a}(\ve \rho(t)), \quad \xi_m := \frac{(3-m)}{(5-m)^2} \frac{(\int_\R Q)^2}{\int_\R Q^2};
\ee
\be\label{f3}
f_3(t) = f_3(c(t), \rho(t)) :=  \frac{\tilde\xi_3}{\sqrt{c(t)}} (c(t)-\la) \frac{a'^2}{a^2}(\ve \rho(t)), \qquad  \tilde \xi_3 := \frac \la2 \frac{(\int_\R Q)^2}{\int_\R Q^2},
\ee
and $f_4(t)$ satisfies the decomposition
\be\label{f4}
f_4(t) := f_4^1(t)\frac{a'^2}{a^2}(\ve \rho(t)) + f_4^2(t) \frac{a''}{a}(\ve \rho(t)), \quad  |f_4^i (t)| \leq K.
\ee
\item Finally, $\tilde S[\tilde u] (t,\cdot )$ is a polynomially growing function as $y\to -\infty$, and exponentially decaying as $y\to +\infty$. It satisfies\footnote{See Step 7 in Appendix \ref{A} for a precise description.}
\be\label{Stilde}
\| \tilde S[\tilde u] (t,\cdot )\|_{L^2(y\geq -\frac 3\ve)}  + \| \partial_x \tilde S[\tilde u] (t,\cdot )\|_{L^2(y\geq -\frac 3\ve)}  \leq K \ve^{3/2} e^{-\ve\ga|\rho(t)|} + K\ve^3.
\ee
Moreover, one has the improved estimates
\be\label{SIn}
\abs{\int_\R Q_c \tilde S[\tilde u]} +\abs{\int_\R yQ_c \tilde S[\tilde u]} \leq K \ve^2 e^{-\ve\ga|\rho(t)|} + K\ve^3,
\ee
for the quadratic and quartic cases, and 
\be\label{SIn3}
\abs{\int_\R Q_c \tilde S[\tilde u]} +\abs{\int_\R yQ_c \tilde S[\tilde u]} \leq K \ve^3 e^{-\ve\ga|\rho(t)|} + K\ve^4,
\ee
in the cubic case.
\end{enumerate}
\end{prop}

\medskip

\begin{rem}\label{4.4}
A formal \emph{lower bound} in the defect of the soliton solution can be seen as a consequence of the fact that $f_2(t)\neq  0$, and $f_3(t)\neq 0$ for $m=3$. These perturbations of the dynamical system (\ref{ODE}) imply the lower bounds suggested in (\ref{LBBB}). That is the reason because we perform a second order improvement of the solution in the cubic case.
\end{rem}

\begin{proof}
A similar proof is contained in \cite{Mu2}. Now we improve our result by adding the terms $f_2(t), f_3(t)$ and $f_4(t)$ above, which will be of great importance \emph{to quantify} the lower bound on the defect. For the sake of clarity we include the proof in Appendix \ref{A}. 
\end{proof}

Since $\tilde u \not\in L^2(\R)$, we need to perform a correction in our approximate solution, in order to obtain a valid $L^2$ solution.

\subsection{Correction to the solution $\tilde u$}

The next results are contained in \cite{Mu2}. However, we need some new estimates. Consider a cutoff function $\eta \in C^\infty (\R)$ satisfying the following properties:
\be\label{eta}
\begin{cases}
0\leq \eta (s) \leq 1, \quad 0\leq  \eta' (s) \leq 1, \; \hbox{ for any } s\in \R;\\
\eta(s)\equiv 0 \; \hbox{ for } s\leq -1, \quad  \eta(s)\equiv 1 \; \hbox{ for } s\geq 1.
\end{cases}
\ee 
Define 
\be\label{etac}
\eta_\ve (y) := \eta( \ve y +  2 ),
\ee
and for $w=w( t, y)$ the first order correction constructed in Lemma \ref{lem:omega}, {\bf redefine}
\be\label{hatu}
\tilde u(t,x) :=  \eta_\ve (y) \tilde u(t,x) =\eta_\ve (y) (R(t,x) + w(t,x)),
\ee
and similarly for $R(t)$ and $w(t)$. Note that, by definition, 
\be\label{newb}
\tilde u(t, x) = 0 \quad \hbox{ for all } y\leq -\frac 3\ve.
\ee

The following Proposition deals with the error associated to this cut-off function, and the new approximate solution $\tilde u$.

\medskip

\begin{prop}[Final approximate solution for (\ref{aKdV}), \cite{Mu2}]\label{CV}~

There exist constants $\ve_0,\ga, K>0$ such that for all $0<\ve <\ve_0$ the following holds. 

\medskip

\begin{enumerate}

\item Consider the localized function $\tilde u(t) = R(t) +w(t)$ defined in (\ref{etac})-(\ref{hatu}). Then we have

\begin{enumerate}
\item \emph{$L^2$-solution}. For all $t\in [-T_\ve, \tilde T_\ve]$, $w(t, \cdot ) \in H^1(\R)$, with
\be\label{H1}
\|w(t, \cdot ) \|_{H^1(\R)} \leq K \ve^{1/2}e^{-\ga \ve |\rho(t)|}.
\ee
\item \emph{Almost orthogonality}. For all $t\in [-T_\ve, \tilde T_\ve]$ one has
\be\label{AO}
\abs{\int_\R w(t,x)Q_c(y)dx} + \abs{\int_\R yw(t,x)Q_c(y)dx} \leq K \ve^{10}.
\ee
\end{enumerate}

\item \emph{Almost solution}. The error associated to the new function $\tilde u(t)$ satisfies
\bee
S[\tilde u] & = & (c'(t) - \ve f_1(t) -\ve^2 \delta_{m,3} f_3(t))\partial_c\tilde u   \\
& & \qquad + (\rho'(t) -c(t)+ \la -  \ve f_2(t) -\ve^2 \delta_{m,3}f_4(t)) (\partial_\rho  \tilde u +  \ve  \eta_\ve ' \tilde u)+ \tilde S[\tilde u](t),
\eee
with $\|\ve \eta_\ve ' \tilde u \|_{H^1(\R)} \leq K\ve^{3/2} e^{-\ve\ga|\rho(t)|} $, and
\be\label{SH2}
\| \tilde S[\tilde u](t) \|_{H^1(\R)} \leq K \ve^{3/2}e^{-\ga \ve |\rho(t)|}.
\ee
Finally, estimates (\ref{SIn})-(\ref{SIn3}) remains unchanged.
\end{enumerate}

\end{prop}

\medskip

\begin{proof}
The proof of (\ref{H1}) follows from a direct computation. Indeed,
$$
\|w(t) \eta_\ve\|_{H^1(\R)} \leq K \|w(t)\|_{H^1(y\geq -\frac 3\ve)} ,
$$
but from (\ref{Ac})-(\ref{Bc}),
$$
\| \ve d(t) A_c(y) + \ve^2 B_c(t,y) \|_{H^1(y\geq -\frac 3\ve)}  \leq K\ve^{1/2}e^{-\ve\ga|\rho(t)|}.
$$

\medskip

Let us consider now (\ref{AO}). Here we have, using (\ref{Ac}), 
$$
\int_\R w(t,x) \eta_\ve (y) Q_c(y) = \int_\R w(t,x) (\eta (\ve y +2)-1) Q_c(y). 
$$
Note that $\eta (\ve y +2)-1 \equiv 0$ for $y\geq -\frac 1\ve$. Using the exponential decay of $Q_c(y)$, we have
\bee
\abs{\int_\R w(t,x) \eta_\ve (y) Q_c(y)}  & \leq & K \int_{y\leq -\frac 2\ve}  \ve |y| e^{\sqrt{c} y} + K \int_{y\in (-\frac 2\ve, -\frac 1\ve)}  \ve |y| e^{- \frac 12 (\ve y+2)} e^{\sqrt{c} y} \\
& \leq & K e^{-\ga/\ve} \leq K\ve^{10}. 
\eee
The proofs for $yA_c$, $B_c$ and $yB_c$ are similar. We skip the details.

\medskip

For the proof of (\ref{SH2}), we proceed as follows. First of all, a simple computation shows that
$$
S[\eta_\ve \tilde u]  =   \eta_\ve S[\tilde u] + (\eta_\ve)_t \tilde u + 3\ve \eta_\ve' \tilde u_{xx} + 3\ve^2 \eta_\ve^{(2)} \tilde u_{x} + \ve^{3} \eta_\ve^{(3)} \tilde u -\la \ve \eta_\ve' \tilde u + \ve \eta_\ve' a(\ve x) \tilde u^m .
$$
Since $\supp \eta_\ve^{(k)} \subseteq [-\frac 3\ve, -\frac 1\ve]$ for $k=1,2$ and $3$, we have
\bee
& & 3\ve \eta_\ve' \tilde u_{xx} + 3\ve^2 \eta_\ve^{(2)} \tilde u_{x} + \ve^{3} \eta_\ve^{(3)} \tilde u -\la \ve \eta_\ve' \tilde u + \ve \eta_\ve' a(\ve x) \tilde u^m = \\
& & \qquad =  O_{H^1(\R)} (\ve^{3/2} e^{-\ve\ga|\rho(t)|}) + O_{H^1(\R)}(\ve^{10}).
\eee
Similarly, from the definition of $\rho'(t)$ and (\ref{r1})
\bee
(\eta_\ve)_t \tilde u & = & -\rho'(t)\ve \eta_\ve' \tilde u \\
&  = &  O_{H^1(\R)}(\ve^{3/2} e^{-\ve\ga|\rho(t)|}) + O_{H^1(\R)}(\ve^{10}).
\eee
Collecting the above terms, we have
$$
S[\eta_\ve \tilde u]  =   \eta_\ve S[\tilde u]+ O_{H^1(\R)}(\ve^{3/2} e^{-\ve\ga|\rho(t)|}) + O_{H^1(\R)}(\ve^{10}).
$$
Finally, from the decomposition (\ref{Decomp}), one has $S[\tilde u] = \hbox{dynamical system} + \tilde S[\tilde u]$, with
$$
\eta_\ve \tilde S[\tilde u] =  O_{H^1(\R)}(\ve^{3/2} e^{-\ve\ga|\rho(t)|} + \ve^3).
$$
Indeed, we have, from (\ref{Stilde}), (\ref{dd}) and (\ref{Ac}),
$$
\| \eta_\ve \tilde S[\tilde u] \|_{H^1(\R)} \leq K\ve^{3/2} e^{-\ve\ga|\rho(t)|} + K \ve^3.
$$
Finally, one has
\bee
& & \eta_\ve \Big[  (c'(t) - \ve f_1(t) -\ve^2 \delta_{m,3} f_3(t))\partial_c\tilde u    +\ (\rho'(t) -c(t)+ \la -  \ve f_2(t) -\ve^2 \delta_{m,3}f_4(t)) \partial_\rho \tilde u \Big]   \\
& &\qquad  = (c'(t) - \ve f_1(t) -\ve^2 \delta_{m,3} f_3(t))\partial_c(\eta_\ve\tilde u) +  (\rho'(t) -c(t)+ \la -  \ve f_2(t) -\ve^2 \delta_{m,3}f_4(t)) \partial_\rho (\eta_\ve \tilde u) \\
& & \qquad \quad  + \ \ve (\rho'(t) -c(t)+ \la -  \ve f_2(t) -\ve^2 \delta_{m,3}f_4(t)) \eta_\ve ' \tilde u.
\eee
Since $\ve \eta_\ve ' \tilde u =O_{H^1(\R)}(\ve^{3/2} e^{-\ve\ga|\rho(t)|}) $, from this last estimate, we get the final conclusion. 
\end{proof}

\medskip

\begin{rem}
Note that, even under a second order term $(=\ve^2 B_c)$ in our approximate solution $\tilde u$, \emph{we have no chance} of improving the associated error, and we obtain the same result as in \cite{Mu2}, namely $O(\ve^{1/2})$. We believe that this phenomenon is a consequence of the fact that, since $A_c$ is not localized, we have lost most of the accuracy of $\tilde u$, in the standard energy space. Further improvements should consider e.g. a new, more accurate description of the correction term $w(t)$ in $H^1(\R)$.  
\end{rem}

\medskip

\subsection{$H^1$-estimates}

In this subsection we recall some estimates concerning our approximate solution.

\medskip

\begin{lem}[First estimates on $\tilde u$,\cite{Mu2}]~

\begin{enumerate}
\item \emph{Decay away from zero}. Suppose $f=f(y)\in \mathcal Y$, with $y=x-\rho(t)$. Then there exist $K,\ga>0$ constants such that for all $t\in [-T_\ve,  \tilde T_\ve]$
\be\label{Est1}
\norm{a'(\ve x) f(y)}_{H^1(\R)} \leq K e^{-\ga\ve|\rho(t)|}.
\ee
\item \emph{Almost soliton solution}. The following estimates hold for all $t\in [-T_\ve,\tilde  T_\ve]$:
\be\label{Est2}
\norm{\tilde u_t + \rho' \tilde u_x - c'\partial_c \tilde u}_{H^1(\R)}  \leq K \ve e^{-\ga\ve |\rho(t)|},
\ee
\be\label{Est2a}
\tilde u_{xx} -\la \tilde u + a(\ve x) \tilde u^m = \frac 1{\tilde a} (c-\la) Q_c + O_{L^2(\R)}(\ve e^{-\ga\ve |\rho(t)|}),
\ee
and
\be\label{Est20}
\| (\tilde u_{xx} -c\tilde u +  a(\ve x) \tilde u^m)_x \|_{H^1(\R)} \leq K \ve e^{-\ga\ve |\rho(t)|} +K\ve^2.
\ee
\end{enumerate}
\end{lem}

In addition, we have the following result.

\begin{Cl}[Behavior at $t = - T_\ve$, \cite{Mu2}]\label{atpmT}~

Let $(C,P)$ be the \emph{unique} solution of the dynamical systems (\ref{c00}) and (\ref{c}), for any $0\leq \la <1$, $\la \neq \tilde \la$. There exist constants $K,\ve_0>0$ such that for every $0<\ve <\ve_0$ the approximate solution $\tilde u$ constructed in Proposition \ref{CV} satisfies
\be\label{mTe}
\| \tilde u (-T_\ve, C(-T_\ve), P(-T_\ve)) - Q(\cdot + (1-\la)T_\ve ) \|_{H^1(\R)} \leq K \ve^{10}.
\ee
\end{Cl}

In concluding this section, we have constructed and approximate solution $\tilde u(t)$ describing, at least formally, the interaction soliton-potential. In the next section we will show that the \emph{solution} $u$ constructed in Theorem \ref{Tm1} actually behaves like $\tilde u$ inside the interaction region $[-T_\ve, \tilde T_\ve]$.


\bigskip

\subsection{Stability}\label{sec:3}

In this section our objective is to prove that the approximate solution $\tilde u(t)$ describes the dynamics of interaction of the solution $u(t)$, inside the interval $[-T_\ve, \tilde T_\ve]$. Recall that from (\ref{hypINT}) and (\ref{mTe}), one has 
\be\label{hypINTa}
\| u(-T_\ve) - \tilde u(-T_\ve,  C(-T_\ve), P(-T_\ve))) \|_{H^1(\R)}\leq K \ve^{10}.
\ee
In addition, from (\ref{SH2}) one has 
\be\label{INTkl}
\| \tilde S[\tilde u](t)\|_{H^1(\R)} \leq K\ve^{3/2}e^{-\ga\ve|\rho(t)|},    
\ee
for some $K,\ga>0$, and $\la\neq \tilde \la $. 

\medskip

\begin{prop}[Exact solution close to the approximate solution $\tilde u$]\label{prop:I}~

Suppose $\la\in (0,1)$, $\la\neq \tilde \la$. There exists $\ve_0>0$ such that the following holds for any $0<\ve <\ve_0$. There exist $K_0>0$ independent of $\ve$ and unique $C^1$ functions $c, \rho : [-T_\ve, \tilde T_\ve] \to \R$ such that, for all $t\in [-T_\ve, \tilde T_\ve]$,
\be\label{INT41}
\|u(t)-\tilde u(t,c(t), \rho(t)) \|_{H^1(\R)} \leq K_0 \ve^{1/2},
\ee
and
\bea\label{INT42}
& & | \rho'(t) -c(t) +\la -\ve f_2(t) -\ve^2 \delta_{m,3}f_3(t)|  \nonumber \\
& & \qquad +\ \ve^{-1/2} |c'(t) - \ve f_1(t) -\ve^2 \delta_{m,3}f_4(t)| + |c(t) -C(t)|  \leq K_0 \ve^{1/2}.
\eea
Finally, one has
\be\label{crm}
|c(-T_\ve) -1 | + |\rho(-T_\ve) + (1-\la)T_\ve| \leq K\ve^{10},
\ee
with $K>0$ independent of $K_0$.
\end{prop}

\medskip

Before the proof of this result, let us finish the proof of Proposition \ref{T0}.

\medskip

\subsection{Proof of Proposition \ref{T0}}

We are now in position to give a direct proof of Proposition \ref{T0}. Indeed, since (\ref{hypINTa}) is satisfied, we have (\ref{INT41}) for all time $t\in [-T_\ve, \tilde T_\ve]$; in particular, at $t=\tilde T_\ve$ one has
$$
\|u(\tilde T_\ve)  -\tilde u(\tilde T_\ve, c(\tilde T_\ve), \rho(\tilde T_\ve)) \|_{H^1(\R)} \leq K_0\ve^{1/2},
$$
with 
$$
|c(\tilde T_\ve) -C(\tilde T_\ve)| \leq K_0 \ve^{1/2}.
$$
Furthermore, after integration in time of (\ref{INT42})
\be\label{bonda}
| \rho(\tilde T_\ve) -P(\tilde T_\ve)| \leq K_0 \ve^{-1/2 -1/100}.
\ee
Finally, from (\ref{defv}), (\ref{defW}), (\ref{dd}) and Proposition \ref{prop:decomp}, one has
$$
\| \tilde u(\tilde T_\ve, c(\tilde T_\ve), \rho(\tilde T_\ve))  -  2^{-1/(m-1)}Q_{c_\infty(\la)} (\cdot - \rho(\tilde T_\ve))\|_{H^1(\R)} \leq K_0\ve^{1/2}, \quad 0\leq \la<\tilde \la,
$$ 
and
$$
\| \tilde u(\tilde T_\ve, c(\tilde T_\ve), \rho(\tilde T_\ve))  - Q_{c_\infty(\la)} (\cdot - \rho(\tilde T_\ve)) \|_{H^1(\R)} \leq K_0\ve^{1/2}, \quad 0\leq \tilde \la <\la<1.
$$ 
By defining $\rho_\ve := \rho(\tilde T_\ve)$, using (\ref{bonda}) and using the triangle inequality, the conclusion follows, provided Proposition \ref{prop:I} holds.

\bigskip

\begin{rem}
For the sake of clarity in the forthcoming computations, let us denote 
$$
c_1' := c'-\ve f_1-\ve^2 \delta_{m,3}f_3, \quad \hbox{ and } \quad \rho_1' := \rho' -c+\la- \ve f_2-\ve^2 \delta_{m,3} f_4.
$$
\end{rem}

\medskip

\begin{proof}[\bf Proof of Proposition \ref{prop:I}]~

Let $K^*>1$ be a constant to be fixed later. From (\ref{hypINTa}), by continuity in $H^1(\R)$ of the flow, there exists $ -T_\ve<T^*\leq  \tilde T_\ve$ with
\bea
    T^*& := & \sup\big\{T\in [-T_\ve, \tilde T_\ve], \hbox{ such that for all }  t\in [-T_\ve,T],  \hbox{ there exists a smooth}\nonu  \\
    & & \qquad  \hbox{ $r(t)\in \R$, such that } \|u(t) -  \tilde u( \cdot \ ; C(t), r(t))\|_{H^1(\R)}\leq K^* \ve^{1/2 } \big\}. \label{Tstar}
\eea
The objective is to prove that $T^*=  T_\ve$ for $K^*$ large enough and $\al>0$ small. To achieve this, we argue by contradiction, assuming that $T^*< T_\ve $ and reaching a contradiction with the definition of $T^*$ by proving some independent estimates for  $\|u(t)- \tilde u(\cdot \ ; C(t),r(t) )\|_{H^1(\R)}$. 

\medskip

\begin{lem}[Modulation]\label{DEFZ}~

Assume $0<\ve<\ve_0(K^*)$ small enough. There exist $K>0$ and unique $C^1$ functions $c(t), \rho(t)$ such that, for all $t\in [-T_\ve,T^*]$,
\be\label{defz}
z(t)=u(t)-\tilde u(t,c(t),\rho(t)) \quad \hbox{satisfies}\quad
 \int_\R z(t,x) y Q_c(y) dx  = \int_\R z(t,x) Q_c(y)dx=0.
\ee
Moreover, we have,  for all $t\in [-T_\ve,T^*]$,
\bea\label{TRANS3}
& & \|z(-T_\ve)\|_{H^1(\R)} + | c(-T_\ve) -C(-T_\ve) | \leq K \ve^{1/2 }, \nonu \\ 
& &  \|z(t)\|_{H^1(\R)} + | c(t) -C(t) | \leq  K K^* \ve^{1/2 }.
\eea
In addition, $z(t)$ satisfies the following equation
\be\label{Eqz1}
 z_t +  \big\{ z_{xx}  -\la z +  a(\ve x) [ (\tilde u +z)^m - \tilde u^m ] \big\}_x  +  \tilde S[\tilde u]     + c_1'(t) \partial_c \tilde u  + \rho_1'(t) \partial_\rho \tilde u =0.
\ee
Finally, there exists $\ga>0$ independent of $K^*$ such that for every $t\in [-T_\ve, T^*]$,
\be\label{rho1}
 |\rho_1'(t) |  \leq   K\Big[  (m-3 +\ve e^{-\ga\ve|\rho(t)|} ) \Big[\int_\R z^2 e^{- \ga\sqrt{c}|y|} \Big]^{1/2}  +  \int_\R e^{- \ga\sqrt{c}|y|}z^2(t) + \abs{\int_\R yQ_c \tilde S[\tilde u]} \Big],
\ee
and
\be\label{c1}
|c_1'(t) |\leq K \Big[ \int_\R e^{-\ga\sqrt{c} |y|} z^2(t)  +  \ve e^{-\ga\ve|\rho(t)| } \Big[ \int_\R e^{-\ga\sqrt{c} |y|} z^2(t)\Big]^{1/2} + \abs{\int_\R Q_c \tilde S[\tilde u]} \Big].
\ee
\end{lem}

\begin{proof}
The proof of (\ref{defz})-(\ref{TRANS3}) is a standard consequence of the Implicit function theorem, applied for each time $t\in [-T_\ve, T^*]$. Similarly, (\ref{Eqz1}) is a direct computation.

\smallskip

Let us prove (\ref{rho1}) and (\ref{c1}). Let us recall that  $f_2 \equiv 0$ in the cubic case. 
We integrate (\ref{Eqz1}) against $yQ_c$ to obtain
\bee
& &  \partial_t \int_\R yQ_c z  - \int_\R (yQ_c)_t z -  \int_\R (yQ_c)_x \big\{ z_{xx}  -\la z +  a(\ve x) [ (\tilde u +z)^m - \tilde u^m ] \big\}  \nonumber \\
 & & \qquad \qquad  + \int_\R yQ_c \tilde S[\tilde u]   + c_1' \int_\R yQ_c \partial_c \tilde u  + \rho_1'\int_\R yQ_c \partial_\rho \tilde u =0.
\eee
Therefore,
\bee
 \rho_1' \int_\R yQ_c \partial_\rho \tilde u & = & -   \int_\R yQ_c \tilde S[\tilde u] - c_1' \int_\R y Q_c \partial_c \tilde u  + \int_\R (yQ_c)_y \mathcal L z  -\rho_1' \int_\R  (yQ_c)_y z \nonu\\
 & & + c_1' \int_\R y \Lambda Q_c z  -\ve (f_2 + \ve \delta_{m,3} f_4) \int_\R  (yQ_c)_y z + \ve( f_1 +\ve \delta_{m,3} f_3)\int_\R  y\Lambda Q_c z  \nonu \\
 & &  + \int_\R (yQ_c)_y  a(\ve x) [ (\tilde u +z)^m - \tilde u^m -m\tilde u^{m-1}z ]  \nonu \\
 & &  + m \int_\R (yQ_c)_y  [a(\ve x)\tilde u^{m-1} - Q_c^{m-1}]z .
\eee
Note that $\mathcal L \{ (yQ_c)_y\} =-(m-3)Q_c^m -2cQ_c .$ From here and (\ref{defz}) one has
\bee
|\rho_1'| & \leq &   \abs{\int_\R yQ_c \tilde S[\tilde u]} + K(m-3 +  \ve e^{-\ve\ga |\rho(t)|} ) \big[\int_\R e^{-\ga\sqrt{c} |y|} z^2 \big]^{1/2} \\
& & \qquad +  K |c_1'| ( \ve e^{Ð\ve\ga|\rho(t)|}+\|z(t)\|_{L^2(\R)} ) + K\int_\R e^{-\ga\sqrt{c} |y|} z^2.
\eee

We consider now (\ref{c1}). We integrate (\ref{Eqz1}) against $Q_c$ to obtain
\bee
& & \partial_t \int_\R  Q_c z  -  \int_\R  (Q_c)_t z +  \int_\R Q_c \big\{ z_{xx}  -\la z +  a(\ve x) [ (\tilde u +z)^m - \tilde u^m ] \big\}_x   \nonumber \\
 & & \qquad \qquad  +  \int_\R Q_c \tilde S[\tilde u]  + c_1' \int_\R Q_c \partial_c \tilde u  + \rho_1'\int_\R Q_c \partial_\rho \tilde u =0.
\eee
So we have
\bee
 c_1' \int_\R Q_c \partial_c \tilde u & = & -   \int_\R Q_c \tilde S[\tilde u] - \rho_1' \int_\R Q_c \partial_\rho \tilde u  + \int_\R Q_c' \mathcal L z  -\rho_1' \int_\R  Q_c' z + c_1' \int_\R  \Lambda Q_c z \nonu\\
 & &  -\ve (f_2 + \ve f_4) \int_\R  Q_c' z + \ve( f_1 +\ve f_3)\int_\R  \Lambda Q_c z   \nonu\\
 & &  + \int_\R Q_c'  a(\ve x) [ (\tilde u +z)^m - \tilde u^m -m\tilde u^{m-1}z ]   + m \int_\R Q_c'  [a(\ve x)\tilde u^{m-1} - Q_c^{m-1}]z .
\eee
After a similar computation to the recently performed, one gets
\bee
|c_1'| & \leq &  K |\rho_1'| \Big[\int_\R e^{-\ga\sqrt{c} |y|} z^2(t) \Big]^{1/2} +K\ve  |\rho_1'| e^{-\ve\ga|\rho(t)|}  + K \ve e^{-\ve\ga |\rho(t)|} \Big[ \int_\R e^{-\ga\sqrt{c} |y|} z^2(t)\Big]^{1/2} \\
& &\qquad +  \abs{\int_\R Q_c \tilde S[\tilde u]} + K\int_\R e^{-\ga\sqrt{c} |y|} z^2 \\
& \leq &   K \ve e^{-\ve\ga |\rho(t)|}\Big[ \int_\R e^{-\ga\sqrt{c} |y|} z^2(t)\Big]^{1/2}+ K K^* \ve^{1/2}\abs{\int_\R yQ_c \tilde S[\tilde u]}  \\
& &\qquad + \abs{\int_\R Q_c \tilde S[\tilde u]} + K\int_\R e^{-\ga\sqrt{c} |y|} z^2.
\eee
Using (\ref{SIn})-(\ref{SIn3}), we obtain the final result. The proof  is complete.
\end{proof}

\medskip

\noindent
{\bf Virial estimate.}
A better understanding of the estimate on the scaling parameter (\ref{c1}) needs the introduction of a viriel estimate, in the spirit of \cite{Mu2} (see Lemma 6.4). See also \cite{H} for a similar result, in the context of a different gKdV equation.

First of all, we define some auxiliary functions. Let $\phi \in C^\infty(\R)$ be an \emph{even} function satisfying the following properties
\be\label{psip}
\begin{cases}
\phi' \leq 0 \; \hbox{ on } [0, +\infty); \quad  \phi (x) =1 \; \hbox{ on } [0,1], \\
\phi (x) = e^{-x}  \; \hbox{ on } [2, +\infty) \quad\hbox{and}\quad  e^{-x} \leq \phi (x) \leq 3e^{-x}  \; \hbox{ on } [0,+\infty).
\end{cases}
\ee
Now, set $\psi(x) := \int_0^x \phi $. It is clear that $\psi$ an odd function. Moreover, for $|x|\geq 2$,
\be\label{asy}
\psi(+\infty) -\psi (|x|) = e^{-|x|}.
\ee
Finally, for $A>0$, denote 
\be\label{psiA}
\psi_A(x) := A(\psi(+\infty) + \psi(\frac xA))>0; \quad e^{-|x|/A} \leq \psi_A'(x)   \leq 3e^{-|x|/A}. 
\ee
Note that $\lim_{x\to -\infty} \psi(x) =0$. We claim the following

\begin{lem}[Sharp Virial-type estimate]\label{VL}~

There exist $K, A_0, \delta_0>0$  such that for all $t\in [-T_\ve, T^*]$ and for some $\ga =\ga(A_0)>0$,
\be\label{dereta}
 \partial_t \int_\R  z^2(t,x) \psi_{A_0}(y)   \leq   -\delta_0  \int_\R ( z_x^2 + z^2 )(t,x) e^{-\frac 1{A_0} |y|}  + KA_0 K^* \ve^{5/2}e^{-\ve\ga|\rho(t)|}.
\ee
\end{lem}

\begin{proof} Let $t\in [-T_\ve, T^*]$. Replacing the value of $ z_t$ given by (\ref{Eqz1}),  we have
\bea
 \partial_t \int_\R  z^2 \psi_{A_0}(y)  & = &   2\int_\R  z  z_t  \psi_{A_0}(y) -\rho'(t) \int_\R  z^2 \psi_{A_0}'(y) \nonumber \\
&  = &  2\int_\R ( z   \psi_{A_0}(y) )_x (  z_{xx} -\la z+  m a(\ve x) \tilde R^{m-1}  z ) \label{e0} \\
& &    -(c-\la +\ve f_2 +\ve^2 \delta_{m,3}f_4)(t)\int_\R  z^2 \psi_{A_0}' - 2\rho'_1(t)\int_\R  z \partial_\rho \tilde u \psi_{A_0}\label{e1} \\
& &   + 2\int_\R ( z   \psi_{A_0}(y) )_x a(\ve x)  [(\tilde u +  z)^m -\tilde u^m -m\tilde u^{m-1} z ] \label{e2}\\
& &  - 2c_1'(t)\int_\R  z \partial_c\tilde u \psi_{A_0}    - \rho_1'(t) \int_\R  z^2 \psi_{A_0}' \label{e3}\\
& &   + 2m\int_\R z( z   \psi_{A_0}(y) )_x a(\ve x) (\tilde u^{m-1} -\tilde R^{m-1}) -2\int_\R z \psi_{A_0}\tilde S[\tilde u] \label{e4}.
\eea
First of all, note that
\bee
|(\ref{e2})| & \leq & K\abs{ \int_\R  z_x   \psi_{A_0}(y) a(\ve x)  [(\tilde u +  z)^m -\tilde u^m -m\tilde u^{m-1} z ]} \\
& & \qquad +  K\abs{\int_\R     \psi_{A_0}'(y) a(\ve x) z [(\tilde u +  z)^m -\tilde u^m -m\tilde u^{m-1} z ] } \\
& \leq &   KA_0 K^*\ve^{1/2}\int_\R   z^2 (t)e^{-\ga\sqrt{c}|y|} + KK^*\ve^{1/2} \int_\R   z^2(t)e^{-\frac 1{A_0}|y|} \\
& & \qquad +K \abs{\int_\R z^{m+1} ( \psi_{A_0}(y) a(\ve x))_x}\\
& \leq & KK^*A_0 \ve^{1/2}\int_\R   z^2(t)e^{-\frac 1{A_0}|y|} + K A_0 \ve \|z(t)\|_{H^1(\R)}^{m+1} \\
& \leq & KK^*A_0 \ve^{1/2}\int_\R   z^2(t)e^{-\frac 1{A_0}|y|} + K (K^*)^{m+1}A_0 \ve^{(m+3)/2}.
\eee
for $A_0$ large, but independent of $\ve$.
Now, by using (\ref{rho1}) and (\ref{c1})  it is easy to check that for $A_0$ large enough, and some constants $\delta_0, \ve_0$ small, one has
\bee
 |(\ref{e3})|  &  \leq &   |c_1'(t)| \abs{\int_\R  z \partial_c\tilde u \psi_{A_0} }+  K K^* \ve^{1/2} \int_\R  z^2(t) e^{-\frac 1{A_0}|y|} \\
 & \leq &  \frac{\delta_0}{100}\int_\R z^2(t)e^{-\frac 1{A_0}|y|} + KK^* A_0\ve^{5/2} e^{-\ve\ga|\rho(t)|}.
\eee
On the other hand, the terms (\ref{e0}) and (\ref{e1}) goes similarly to the terms $B_1$ and $B_2$ in Appendix B of \cite{MMnon}. Indeed, we have
\bee
(\ref{e0}) + (\ref{e1}) & =& -\int_\R \psi_{A_0}' (  3z_x^2  + c z^2 -  mQ_c^{m-1}  z^2 ) - m\int_\R (Q_c^{m-1})' z^2 \psi_{A_0}  \\
& &   + \int_\R  z^2 \psi_{A_0}^{(3)} - 2\rho_1'(t) \int_\R  z \partial_\rho \tilde u\psi_{A_0} \\
& & +2m\int_\R (z\psi_{A_0})_x z (a\tilde R^{m-1} -Q_c^{m-1}) -\ve(f_2 +\ve\delta_{m,3} f_4)\int_\R z^2 \psi_{A_0}' .
\eee
We finally get, taking $\ve$ small, depending on $A_0$,
$$
(\ref{e0}) + (\ref{e1})\leq  -\frac{\delta_0}{10}\int_\R ( z_x^2 +  z^2)(t)e^{-\frac 1{A_0}|y|}.
$$
Finally, the term (\ref{e4}) can be estimated as follows
\bee
|(\ref{e4})| & \leq &  K \abs{\int_\R z( z   \psi_{A_0}(y) )_x a(\ve x) (\tilde u^{m-1} -\tilde R^{m-1}) } +K\abs{\int_\R z \psi_{A_0}\tilde S[\tilde u]}\\
& \leq &K \abs{\int_\R z^2   \psi'_{A_0}(y) a(\ve x) (\tilde u^{m-1} -\tilde R^{m-1}) }  \\
& & \qquad + K \abs{\int_\R z z_x   \psi_{A_0}(y)  a(\ve x) (\tilde u^{m-1} -\tilde R^{m-1}) } +  K A_0 K^* \ve^{2} e^{-\ve \ga|\rho(t)|} \\
& \leq & K A_0 \ve \int_\R (z^2(t) + z_x^2(t)) e^{-\frac 1{A_0} |y|}  +  K A_0 K^* \ve^{5/2} e^{-\ve \ga|\rho(t)|} + K A_0 K^* \ve^{7/2}.
\eee
Collecting these estimates, we finally get (\ref{dereta}).
\end{proof}

\medskip

A simple but very important conclusion of the last estimate, is the following. One has, from (\ref{c1}) and (\ref{dereta}),
\be\label{intc1}
\int_{-T_\ve}^{t} |c'_1(s)| ds  \leq  KK^* \ve,
\ee
for all $t\in [-T_\ve, T^*]$, by taking $A_0$ large enough, independent of $\ve$ and $K^*$. In other words, we improve the estimate on the integral of $|c_1'(t)|$ (a crude integration of (\ref{c1}) gives $\int_{-T_\ve}^t |c'_1(s) ds|  \leq K\ve^{-\frac 1{100}}$.)

\medskip

\subsection{Energy functional for $z$}\label{EFz}

Consider the functional $\mathcal F$ defined as follows

\be\label{F}
\mathcal F(t) := \frac 12 \int_\R (z_x^2 +c(t) z^2) - \frac{1}{m+1}\int_\R a(\ve x) [(\tilde u+ z)^{m+1} -\tilde u^{m+1} - (m+1)\tilde u^{m}z]. 
\ee

Similary to \cite{Mu2}, and thanks to Lemma \ref{surL}, we have the following coercivity property: there exist $K,\nu_0>0$, independent of $K^*$ and $\ve$ such that for every $t\in [-T_\ve, T^*]$
\be\label{Coer2}
\mathcal F(t) \geq \nu_0 \|z(t)\|_{H^1(\R)}^2  - K (\ve e^{-\ga\ve |\rho(t)|} + \ve^2)\|z(t)\|_{L^2(\R)}^2 - K  \|z(t)\|_{L^2(\R)}^3. 
\ee

The next step is to obtain independent estimates on $\tilde F(T^*)$.  We follow \cite{Mu2}, but now estimate \ref{intc1} is the key element to close the argument.

\begin{lem}[Estimates on $\mathcal F(t)$]\label{Ka}~

The following properties hold for any $t\in [-T_\ve, T^*]$.
\begin{enumerate}
\item First time derivative. 
\bea
\mathcal F'(t) &  = &   -\int_\R z_t \big\{ z_{xx} -c z + a(\ve x) [ (\tilde u+z)^m -\tilde u^m ] \big\} + \frac 12 c'\int_\R z^2 \nonumber \\
& &  -\int_\R a(\ve x) \tilde u_t[ (\tilde u+z)^m -\tilde u^m -m\tilde u^{m-1}z]. \label{Fp}
\eea
\item Integration in time. There exist constants $K,\ga>0$ such that
\bea\label{IntF}
\mathcal F(t) -\mathcal F(-T_\ve)&  \leq &  K(K^*)^4 \ve^{2-\frac 1{100}}  + K(K^*)^3 \ve^{\frac 32-\frac 1{100}} + KK^* \ve \nonumber  \\
 & &  + K\int_{-T_\ve}^t \ve e^{-\ve\ga |\rho(s)|} \|z(s)\|_{H^1(\R)}^2 ds . 
\eea
\end{enumerate}
\end{lem}

\begin{proof}
First of all, (\ref{Fp}) is a simple computation. Let us  consider (\ref{IntF}). Replacing (\ref{Eqz1}) in (\ref{Fp}) we get
\bea
\mathcal F'(t) & = & (c(t)-\la)  \int_\R 
 a(\ve x) [ (\tilde u+z)^m -\tilde u^m ]  z_x \label{Fp1} \\
& & + \rho_1'(t)  \int_\R \partial_\rho \tilde u \big\{ z_{xx} -cz + a(\ve x) [ (\tilde u+z)^m -\tilde u^m ] \big\} \label{Fp2} \\
& & + c_1'(t) \int_\R \partial_c \tilde u \big\{ z_{xx} -cz + a(\ve x) [ (\tilde u+z)^m -\tilde u^m ] \big\} \label{Fp2b} \\
& & +  \int_\R \tilde S[\tilde u] \big\{ z_{xx} -cz + a(\ve x) [ (\tilde u+z)^m -\tilde u^m ] \big\} \label{Fp3} \\
& & + \frac 12 c_1'(t) \int_\R z^2+\frac 12 \ve (f_1 + \ve \delta_{m,3}f_3)(t) \int_\R z^2 \label{Fp3b} \\
& &-\int_\R a(\ve x) \tilde u_t [ (\tilde u+z)^m -\tilde u^m -m\tilde u^{m-1}z]. \label{Fp4}
\eea 
Now we consider the case $m=2$, the other cases being similar (see \cite{Mu2} for more details.) First of all, note that 
$$
\frac 12 \ve f_1(t) \int_\R z^2 \leq K\ve e^{-\ve\ga|\rho(t)|} \|z(t)\|_{L^2(\R)}^2.
$$
Next, after some simplifications, we get
\bee
(\ref{Fp1})&  =&  (c-\la)  \int_\R a(\ve x) [2\tilde u z + z^2 ] z_x \\
& = & -(c-\la) \int_\R [ a(\ve x) \tilde u_x z^2  + \ve a'(\ve x) \tilde u z^2 + \frac 13 \ve a'(\ve x) z^3 ] .
\eee
From this, using (\ref{Est1}),
\be\label{a00}
(\ref{Fp1}) + (c-\la) \int_\R a(\ve x) \tilde u_x z^2 \leq K \ve e^{-\ga\ve |\rho(t)|} \|z(t)\|_{L^2(\R)}^2 + K\ve \|z(t)\|^3_{H^1(\R)}. 
\ee
Now we estimate (\ref{Fp2}). Since $\partial_\rho \tilde u =\partial_\rho R + O(\eta_\ve w_y) + O_{H^1(\R)}(\ve^{3/2} e^{-\ve \ga |\rho(t)|})$ (cf. Proposition \ref{CV}), one has 
\bea\label{a01}
(\ref{Fp2})  & = &  \rho'_1 \int_\R \partial_\rho \tilde u  \big\{ z_{xx} -cz + a(\ve x) [ 2\tilde u z + z^2  ] \big\}  \nonumber  \\
& = & - \rho'_1 \int_\R a(\ve x) \tilde u_x   z^2 + O(\ve e^{-\ve\ga |\rho(t)|} \|z(t)\|^2_{L^2(\R)}).
\eea
Similarly, we have from (\ref{defz})
\bea\label{a02}
(\ref{Fp2b})  & = &  c_1'  \int_\R \partial_c \tilde u  \big\{ z_{xx} -cz + a(\ve x) [ 2\tilde u z + z^2  ] \big\}   \nonu \\
& = &  c_1'  \int_\R a(\ve x) \partial_c \tilde u   z^2 + O(\ve e^{-\ve\ga |t|} \|z(t)\|^2_{L^2(\R)}).
\eea
On the one hand, we have
\bea\label{a03}
& & \abs{  \int_\R \tilde S[\tilde u] \big\{ z_{xx} -cz + a(\ve x) [ 2\tilde u z + z^2 ] \big\}  }  \leq \nonu\\ 
& & \qquad \leq \abs{  \int_\R\partial_x \tilde S[\tilde u] z_{x} } + K (1+K^* \ve^{1/2})\abs{\int_\R \tilde S[\tilde u] z}  + K \abs{\int_\R \tilde S[\tilde u] \tilde u z } \nonu \\
& & \qquad \leq K \ve^{2}(e^{-\ve\ga|\rho(t)|} +\ve) \|z(t)\|_{H^1(\R)} +  K (1+K^* \ve^{1/2})\abs{\int_\R \tilde S[\tilde u] z} \nonu\\
& & \qquad \leq KK^*\ve^{2}e^{-\ve\ga|\rho(t)|}.
\eea
Finally, 
\bea\label{a04}
(\ref{Fp4}) &  = &     - \int_\R a(\ve x) (\tilde u_t + \rho' \tilde u_x  - c' \partial_c \tilde u) z^2 +  \rho'\int_\R a(\ve x) \tilde u_x z^2  -c'\int_\R a(\ve x) \partial_c \tilde u z^2 \nonumber \\ 
& & \qquad +O(\ve e^{-\ve\ga|\rho(t)|} \|z(t)\|_{L^2(\R)}^2). 
\eea
We get then from (\ref{Est2}) and (\ref{a00})-(\ref{a04})
$$
\mathcal F'(t) \leq  \frac 12  c_1' \|z(t)\|_{L^2(\R)}^2 + K \ve e^{-\ga\ve |\rho(t)|} \|z(t)\|^2_{L^2(\R)} + K\ve\|z(t)\|_{H^1(\R)}^3.
$$
Collecting the above estimates and (\ref{rho1}), and using (\ref{intc1}), after an integration, we finally get
$$
\mathcal F(t) -\mathcal F(-T_\ve) \leq  K(K^*)^3 \ve^{\frac 32-\frac 1{100}}  + KK^* \ve + K \int_{-T_\ve}^t  \ve e^{-\ga\ve |\rho(s)|} \|z(s)\|_{H^1(\R)}^2 ds,
$$ 
as desired. The cases $m=3$ and $4$ are similar.
\end{proof}

We are finally in position to show that $T^*<T_\ve$ leads to a contradiction. 

\medskip

\noindent
{\bf End of proof of Proposition \ref{prop:I}} Since from Lemma \ref{DEFZ}, $\mathcal F(-T_\ve) \leq K \ve,$ using (\ref{Coer2}) and Lemma (\ref{IntF}) we get
\bee
\|z(t)\|_{L^2(\R)}^2   \leq   K\Big[ \ve +  (K^*)^4 \ve^{2-\frac 1{100}} + (K^*)^3 \ve^{\frac 32-\frac 1{100}}  + K^* \ve  +  \int_{-T_\ve}^t  \ve e^{-\ga\ve |\rho(s)|} \|z(s)\|_{H^1(\R)}^2 ds \Big].
\eee
Now, by Gronwall's inequality (see e.g. \cite{Mu2} for a detailed proof), there exists a large constant $K>0$, but independent of $K^*$ and $\ve$, such that
\be\label{RMCF}
\|z(t)\|_{H^1(\R)}^2 \leq  K\ve  + K(K^*)^3 \ve^{\frac 32-\frac 1{100}}.
\ee
Indeed, we just need to justify that $\abs{\int_{-T_\ve}^{t} \ve e^{-\ga\ve |\rho(s)|} ds} \leq K$, independent of $\ve$ and $K^*$. It is clear that this estimate holds in the case $0\leq \la<\tilde \la$, since $\rho'(s) \geq  \frac{9}{10}(c(s) -\la )\geq \frac{8}{10}(c_\infty(\la) -\la)>0 $.  The case $\tilde \la<\la<1$ requires more care, since $P'(t_0)=0.$ To overcome this difficulty, we split the proof into three parts, arguing similarly to the proof of Lemma \ref{ODE2}. First, we suppose $t\leq t_0 -\frac \al\ve$, for $\al>0$ small, but independent of $\ve$. It is clear that 
$$
\abs{\int_{-T_\ve}^{t} \ve e^{-\ga\ve |\rho(s)|} ds} \leq \frac K\al,
$$
since $\rho'(t) \sim c(t)-\la \sim \al$ (see (\ref{Cal}).) Let us suppose $t_0 -\frac \al\ve \leq t\leq t_0 +\frac \al\ve$. In this case one has
$$
\abs{\int_{-T_\ve}^{t} \ve e^{-\ga\ve |\rho(s)|} ds} \leq \frac K\al + K\al.
$$
Finally, the remaining case $t\geq t_0 +\frac \al\ve$ is similar to the first case. Since each estimate is independent of $K^*$ and $\ve$, provided $\ve$ small, we get the final conclusion.

\medskip

Let us come back to the main proof. From estimate (\ref{RMCF}), and taking $\ve$ small, and $K^*$ large enough, we obtain that for all $t\in [-T_\ve, T^*]$,
\be\label{KKa}
\|z(t)\|_{H^1(\R)}^2 \leq  \frac 14 (K^*)^2 \ve.
\ee
Therefore, we improve the estimate on $z(t)$ stated in (\ref{TRANS3}). 

\medskip

Next, we prove that 
\be\label{lastest}
\|u(T^*)- \tilde u (\cdot \ ; C(T^*), \rho(T^*)) \|_{H^1(\R)} \leq \frac 12 K^* \ve^{1/2}.
\ee
Indeed, expanding the definition of the energy (\ref{Ea}),
\bee
E_a[\tilde u(\cdot, c,\rho)  + z](t) & = & E_a[\tilde u](t)  -\int_\R z(\tilde u_{xx} -\la \tilde u +a(\ve x) \tilde u^m)\\
& & \qquad  -\frac 1{m+1} \int_\R a(\ve x) [(\tilde u +z)^{m+1} -\tilde u^{m+1} -(m+1) \tilde u^m z].
\eee
Now we use (\ref{Est2a}), the definition of $\tilde u$ given in (\ref{defv}) and the orthogonality condition (\ref{defz}); we get
\bee
E_a[\tilde u(\cdot, c,\rho)  + z](t)  & =  & E_a[R + w](t)+ O(\ve e^{-\ve \ga |\rho(t)|}\|z(t)\|_{H^1(\R)})+ O(\|z(t)\|_{H^1(\R)}^2)\\
& =&  E_a[R + w](t) + O(K^*\ve^{3/2} e^{-\ve \ga |\rho(t)|}) +O((K^*)^2 \ve).
\eee
On the other hand, a simple computation shows that
\bee
 E_a[R+w](t) &  =&    E_a[R](t) -\int_\R w(R_{xx} -\la R +a(\ve \rho) R^m)  -\int_\R w (a(\ve x) -a(\ve \rho)) R^m \\
& & \quad  -\frac 1{m+1} \int_\R a(\ve x) [(R + w)^{m+1} -R^{m+1} -(m+1) R^m w] \\
& = &  E_a[R](t) -\frac 1{\tilde a}(c-\la)\int_\R w Q_c + O(\ve \|w(t)\|_{H^1(\R)}) + O(\|w(t)\|_{H^1(\R)}^2)\\
& = &  E_a[R](t) + O(\ve^2 e^{-\ve\ga |\rho(t)|} ) + O(\ve^{10}).
\eee
Note that in the last line we have used (\ref{H1}) and (\ref{AO}). Finally,
\bee
 E_a[R](t) & = & \frac{1}{\tilde a^2(\ve \rho)} \Big[ \frac 12 \int_\R Q_c'^2 +\frac \la 2 \int_\R Q_c^2  -\frac 1{m+1} \int_\R Q_c^{m+1} \Big] \\
 & & \qquad    -\frac 1{(m+1)\tilde a^2(\ve \rho)} \int_\R \Big[ \frac{a(\ve x)}{a(\ve \rho)}  - 1\Big] Q_c^{m+1}\\
 & =&  \frac{1}{\tilde a^2(\ve \rho)} E_1[Q_c] + O(\ve^2) = \frac{c^{2\theta}}{\tilde a^2(\ve \rho)} (\la-\la_0c)M[Q] + O(\ve^2),
\eee
(for the last identity, see Lemma \ref{IdQ}.)

\medskip

Now we invoke the energy conservation law. We have, for all $t\in [-T_\ve, T^*]$,
$$
E_a[\tilde u(\cdot, c,\rho)  + z](t) = E_a[\tilde u(\cdot, c,\rho)  + z](-T_\ve).
$$
Therefore,
$$
\frac{c^{2\theta}(t)}{\tilde a^2(\ve \rho(t))} (\la-\la_0c(t))M[Q] \Big|_{-T_\ve}^t = O((K^*)^2\ve) + O(K^*\ve^{3/2}) + O(\ve^2).  
$$
We finally get
$$
|c(T^*) - C(T^*) | \leq K \ve^{1/2} + K K^* \ve + K(K^*)^2\ve.
$$
Using this estimate, (\ref{KKa}), and the triangle inequality, we get finally (\ref{lastest}), provided $K^*$ is large enough. This estimate contradicts the definition of $T^*$ given in (\ref{Tstar}), and concludes the proof of Proposition \ref{prop:I}.
\end{proof}

\bigskip

\section{Proof of the Main Theorems}\label{5}

\medskip

In this small section we prove the main results, namely Theorems \ref{MTL1}, \ref{MTL2}, and \ref{MTL3}. It turns out that Theorems \ref{MTL1} and \ref{MTL2} are of similar structure. 
 
\medskip 
 
 \noindent
{\bf Proof of Theorems \ref{MTL1} and \ref{MTL2}.} Let us consider $u(t)$ be the solution of (\ref{aKdV}) satisfying (\ref{Minfty}). Then, from  Proposition \ref{Tm1}, one has (\ref{mTep}). Therefore, Proposition \ref{T0} implies that $u(t)$ satisfies either (\ref{INT41a}) , or (\ref{INT41b}), depending on $\la \in (\la_0,\tilde\la)$ or $\la\in(\tilde\la,1)$, respectively. Finally, invoking Propositions \ref{Tp1} or \ref{Tp1r} respectively, we obtain the final conclusions, namely the asymptotic behavior included in Theorems \ref{MTL1} and \ref{MTL2}. 

Finally, let us prove (\ref{Pc2l}) and (\ref{Pc2m}). It is clear that the proof of (\ref{Pc2l}) is the same as in \cite{Mu2}, since $c_\infty(\la)>\la$. For the proof of (\ref{Pc2m}), we need to be careful. Indeed, from the energy conservation law, one has, for all $t\geq t_1$,
$$
E_a[u](-\infty) = E_a [Q_{c^+}(\cdot -\rho_2(t)) + w^+(t)]
$$
In particular, from the property of asymptotic stability, and Appendix \ref{IdQ} we have as $t\to +\infty$
\be\label{Eplus}
(\la-\la_0)M[Q]    =  (c^+)^{2\theta} (\la -\la_0 c^+ )M[Q] + E^+.
\ee
From this identity $E^+ := \lim_{t\to +\infty} E_a[w^+](t)$ is well defined.  Next, note that from the stability result (\ref{S}) and the Morrey embedding we have that, for any $\la>0$,
\bee
E[w^+](t) & = & \frac 12\int_\R (w^+_x)^2(t) +\frac \la 2 \int_\R (w^+)^2(t) -\frac 1{m+1} \int_\R a(\ve x) (w^+)^{m+1}(t)\\
& \geq &  \frac 12\int_\R (w^+_x)^2(t) +\frac \la 2 \int_\R (w^+)^2(t) -K\ve^{(m-1)/2} \int_\R a(\ve x) (w^+)^{2}(t)\\
& \geq & \nu \|w^+(t)\|_{H^1(\R)}^2
\eee
for some $\nu=\nu(\la)>0$. Passing to the limit, we obtain $\limsup_{t\to +\infty} E[w^+](t) \leq E^+$.

\medskip

On the one hand, note that after an algebraic manipulation the equation for $c_\infty$ in (\ref{cinf2}) can be written in the following form:
$$
c_\infty^{2\theta} (\la_0 c_\infty -\la)M[Q] =(\la_0-\la)M[Q].
$$
On the other hand, note that from (\ref{Eplus}) and the preceding inequality, we have
$$
\nu \limsup_{t\to +\infty }\|w^+(t)\|_{H^1(\R)}^2  \leq  (c^+)^{2\theta} (\la_0 c^+ -\la)M[Q] -  (\la_0-\la)M[Q].
$$ 
Putting together both estimates, we get
$$
 \tilde \nu \limsup_{t\to +\infty }\|w^+(t)\|_{H^1(\R)}^2  \leq   (c^+)^{2\theta+1}-c_\infty^{2\theta+1}   -\frac \la{\la_0} ( (c^+)^{2\theta} -c_\infty^{2\theta} ), 
$$
for some $\tilde \nu>0$. Using a similar argument as in Lemma \ref{C2} we have
$$
 \tilde \nu \limsup_{t\to +\infty }\|w^+(t)\|_{H^1(\R)}^2  \leq \frac 1{\la_0}(\la-c_\infty )(c_\infty^{2\theta}-(c^+)^{2\theta} )   + O(\abs{(c^+)^{2\theta}-c_\infty^{2\theta} }^2).
$$
From this inequality and the bound $\abs{c^+-c_\infty}\leq K\ve $ we get
$$
 \big( \frac{c_\infty}{c^+} \big)^{2\theta} -1 \geq \tilde \nu \limsup_{t\to +\infty }\|w^+(t)\|_{H^1(\R)}^2,
$$
as desired.

\medskip

\noindent
{\bf Proof of Theorem \ref{MTL3}.} Since we have the validity of the stability and asymptotic stability properties, from Remark \ref{NE} we can apply almost the same proof as in \cite{Mu2} to conclude Theorem \ref{MTL3}. Indeed, let us follow the proof of Theorem 1.3 in \cite{Mu2}. It is clear that the proof adapts without modifications in the case $\la_0<\la<\tilde \la$, which is the case where $c_\infty(\la)>\la$. The case $\tilde \la<\la<1$ requires some modifications. First of all, in Proposition 7.2 we use the following Weinstein's functional
$$
E_a[v](t) + (c_\infty(\la)-\la)\hat M[v](t),
$$
with $\hat M[v](t)$ defined in (\ref{hM}). Lemma 7.3 holds with the assumption $-\la<\sigma < \frac{11}{10}(c_\infty(\la)-\la)$. On the other hand, Lemma 7.4 is valid with the assumption $\tilde \sigma >\frac{9}{10}(c_\infty(\la)-\la)$. Finally, in the conclusion of the proof we use that $c_\infty(\la)<\la<1$ to obtain the desired contradiction. The rest of the proof is the same.

\bigskip

\appendix

\section{Proof of Proposition \ref{Tp1r}}\label{Stab}

In this section we sketch the proof of the stability and asymptotic stability result in the case of a reflected soliton. Note that in this case we have $c_\infty(\la) <\la$. For a detailed proof concerning the case $c_\infty(\la)>\la$, see e.g. Theorem 6.1 in \cite{Mu2}.

\medskip

\noindent
{\bf Proof of the Stability result.} Let us recall that the main difference between Propositions \ref{Tp1r} and Theorem 6.1 in \cite{Mu2} is in the modified mass introduced to construct a Weinstein functional. In the former, we have worked with $\hat M[u](t)$ (cf. (\ref{hM})), and now we will use $\mathcal M[u](t)$, defined in (\ref{Mback}).

\medskip

Let us assume that for some $K>0$ fixed, $t_1\geq  \tilde T_\ve,$
\be\label{48bon1}
\|u(t_1)- Q_{c_\infty} (\cdot  - X_0)\|_{H^1(\R)}\leq K \ve^{1/2}.
\ee
From the local and global Cauchy theory exposed in Proposition \ref{Cauchy}, we know that  the solution $u$ is well defined for all $t\geq t_1$.

Let $D_0>2K$ be a large number to be chosen later,  and set 
\bea\label{Tprime}
  T^* & : = &  \sup\Big\{t\geq t_1 \ | \ \forall \ t'\in [t_1, t), \ \exists \ \tilde \rho_2(t')\in\R  \hbox{ smooth, such that }  
  | \tilde\rho_2'(t') - c_\infty +\la| \leq \frac 1{100},  \nonu \\
& & \qquad \qquad   | \tilde\rho_2 (t_1) - X_0 |\leq \frac 1{100},   \hbox{ and } \; \| u(t')-Q_{c_\infty}(\cdot -\tilde\rho_2(t'))\|_{H^1(\R)}\le D_0 \ve^{1/2} \Big\}.
\eea
Observe that $T^*>t_1$ is well-defined since $D_0>2K$, (\ref{48bon1}) and the continuity of $t\mapsto u(t)$ in ${H^1(\R)}$. The objective is to prove $T^*= +\infty$, and thus (\ref{S}). Therefore, for the sake of contradiction, in what follows {\bf we shall suppose} $T^* <+\infty$.

The first step to reach a contradiction is to decompose the solution $u(t)$ in two parts: soliton plus an error term, on the interval $[t_1,T^*]$, using standard modulation theory around the soliton. In particular, we will find a special $\rho_2(t)$ satisfying the hypotheses in (\ref{Tprime}), but with  
\be\label{onehalf}
\sup_{t\in [t_1, T^*]}\|u(t)- Q_{c_\infty}(\cdot -\rho_2(t))\|_{H^1(\R)}\le \frac 12 D_0 \ve^{1/2},
\ee
a contradiction with the definition of $T^*$.

\begin{lem}[Modulated decomposition]\label{3Dr}~

For $\ve>0$ small enough, independent of $T^*$, there exist  $C^1$ functions $\rho_2, c_2$, defined on $[t_1,T^*]$, with $c_2(t)>0$ and  such that the function $z(t)$ given by
\be\label{eta1a}
z(t,x):=u(t,x)-R (t,x),
\ee
where $R(t,x):= Q_{c_2(t) }(x-\rho_2(t))$, satisfies for all $t\in [t_1,T^*],$
\bea 
&&\int_\R  R(t,x) z(t,x)dx =  \int_\R (x-\rho_2(t))R (t,x)z(t,x)dx =0, \label{10a}\\ 
&& \|z(t)\|_{H^1(\R)}+ |c_2(t) - c_\infty |  
  \leq  K D_0\ve^{1/2},\; \hbox{ and�}\label{11a}\\
&& \|z(t_1)\|_{H^1(\R)}+ |\rho_2(t_1)- X_0|+   |c_2(t_1) -c_\infty |   \leq  K \ve^{1/2} \label{12a},
\eea
where $K$  is not depending on $D_0$.  In addition, $z (t)$ now satisfies the following modified gKdV equation
\bea\label{13a}
& &   z_t + \big\{ z_{xx} -\la z+ a(\ve x) [( R + z)^m -  R^m ] + (a(\ve x)-1) Q_{c_2}^m \big\}_x\qquad\qquad\nonumber \\  
& & \qquad\qquad\qquad\qquad + \; c_2'(t)\Lambda Q_{c_2}  + (c_2-\la-\rho_2')(t)Q'_{c_2}   =0.
\eea
Furthermore, for some constant $\ga>0$ independent of $\ve$, we have the improved estimates:
\bea\label{rho2c2}
|\rho_2'(t)+\la -c_2(t)| & \leq & K(m-3)\Big[ \int_\R e^{-\ga |x-\rho_2(t)|}z^2(t,x) dx\Big]^{\frac 12}\nonumber \\
&& \qquad + K\int_\R e^{-\ga |x-\rho_2(t)|}z^2(t,x) dx + K e^{-\ga\ve t};
\eea
and
\be\label{c2rho2}
\frac{|c_2'(t)|}{c_2(t)} \leq K \int_\R e^{-\ga |x-\rho_2(t)|}z^2(t,x) dx  + K e^{-\ga\ve t}\|z(t)\|_{H^1(\R)} + K\ve e^{-\ve\ga t}.
\ee
\end{lem}
 
\begin{rem}
Note that from (\ref{11a}) and taking $\ve$ small enough we have an improved the bound on $\rho_2(t)$. Indeed, for all $t\in [t_1, T^*]$,
$$
|\rho_2' (t) - c_\infty +\la | +  |\rho_2 (t_1) - X_0 |  \leq 2D_0\ve^{1/2}.
$$
Thus, in order to reach a contradiction, we only need to show (\ref{onehalf}).
Note that for any $t\geq t_1$, 
\be\label{boundapriori}
\rho_2(t) \leq \frac 1{10} (c_\infty(\la) -\la) t_1. 
\ee
This inequality implies that the soliton position is far away from the potential interaction region. 
\end{rem}

\begin{proof}[Proof of Lemma \ref{3Dr}]
See \cite{Mu2}.  
\end{proof}

\medskip

\noindent
{\bf Almost conserved quantities and monotonicity}

By using the decomposition proved in Lemma \ref{3Dr}, we have the following mass and energy monotonicity.

\begin{lem}[Monotonicity of mass backwards in time, see Lemma 7.1 in \cite{Mu2}]\label{MMr}~

Suppose $0<\la<1$. Consider the mass $\mathcal M[u](t)$ introduced in (\ref{Mback}).
Then there exists $\ve_0>0$ such that for all $0<\ve <\ve_0$ one has,
\be\label{decr}
\mathcal M[u](t')-\mathcal M[u](t)  \geq -K e^{-\ve \ga t},
\ee
that for all $t, t' \geq t_1$, with $t'\geq t$.
\end{lem}
\begin{rem}
Note that the above identity is valid only in the case $\la>0$, and it is a consequence of (\ref{3d1d}) and the following identity
\bee
\partial_t \int_\R \frac{u^2}{ a(\ve x) } &  = & 2\ve \int_\R  \frac{ a'}{a^2}(\ve x) u_x^2  + \ve  \int_\R u^2 \big[  \la \frac{ a'}{a^2}(\ve x) -\ve^2 (\frac{a'}{a^2})''(\ve x) \big]   -  2\ve  \int_\R \frac{a' }{a}(\ve x) u^{m+1}.
\eee
\end{rem}

\medskip

\begin{lem}[Almost conservation of modified mass and energy]\label{C2}~

Consider $\mathcal M= \mathcal M[R]$ and  $E_a=E_a[R ]$ the modified mass and energy of the soliton $R$ (cf. (\ref{eta1a})). Then for all $t\in [t_1, T^*]$ we have
\bea\label{dE0}
& &  \mathcal M[R](t)  =  \frac{1}{2}c_2^{2\theta}(t) \int_\R Q^2 + O(e^{-\ve\ga t}); \\
& & E_a[R](t)  =  \frac{1}{2} c_2^{2\theta}(t)(\la- \la_0 c_2(t)) \int_\R Q^2 +O(e^{-\ve \ga t}). \label{dE01}
\eea
Furthermore, we have the bound
\bea\label{dE02}
& & \abs{E_a[R](t_1) -E_a[R](t)  +  (c_2(t_1)-\la) (\mathcal M[R](t_1) - \mathcal M[R](t)) }   \qquad  \nonumber \\
& & \qquad \qquad \qquad  \leq K \abs{ \Big[\frac{c_2(t)}{c_2(t_1)}\Big]^{2\theta}-1}^2 +K e^{-\ve\ga t_1}. 
\eea
\end{lem}

\begin{proof}
We start by showing the first identity, namely (\ref{dE0}). First of all, note that from (\ref{hM}),
$$
\mathcal M[R](t) =   \frac 12 \int_\R \frac{1}{a} R^2 = \frac 12c_2^{2\theta}(t) \int_\R Q^2 + \frac 12\int_\R ( \frac 1{a(\ve x)} -1) R^2.
$$
From (\ref{boundapriori}),
\be\label{dec1}
\abs{ \int_\R (\frac 1{a(\ve x)} -1)R^{2}} \leq Ke^{-\ga \ve t }, 
\ee
for some constants $K,\ga>0$. 
Now we consider (\ref{dE01}). Here we have
\bee
E_a[R](t) & =&\frac 12\int_\R R_x^2  +\frac \la 2 \int_\R R^2 - \frac 1{m+1} \int_\R a(\ve x) R^{m+1}  \\
& =&   c_2^{2\theta}(t)  \Big[ c_2(t) (\frac 12 \int_\R Q'^2 - \frac 1{m+1} \int_\R  Q^{m+1} ) + \frac \la 2 \int_\R Q^2 \Big]    +\frac 1{m+1} \int_\R (1- a(\ve x))R^{m+1}. 
\eee
Similarly to a recent computation, we have
$$
\abs{ \int_\R (1-a(\ve x))R^{m+1}} \leq Ke^{-\ga \ve t }, 
$$
for some constants $K,\ga>0$. On the other hand, from Appendix \ref{AidQ} we have that $\frac 12 \int_\R Q'^2 - \frac 1{m+1} \int_\R  Q^{m+1} = - \frac {\la_0}2 \int_\R Q^2$, $\la_0= \frac{5-m}{m+3}$, and thus
$$
E_a[R](t) =   \frac 12c_2^{2\theta}(t)  ( \la -\la_0 c_2(t) ) \int_\R Q^2  + O( e^{-\ga\ve t}).
$$
Adding both identities we have
$$
E_a[R](t) + (c_2(t_1)-\la) \mathcal M[R](t) =   c_2^{2\theta}(t) ( c_2(t_1) -\la_0 c_2(t) ) M[Q]  + O(e^{-\ve \ga t}).
$$
In particular,
\bee
& & E_a[R](t_1) -E_a[R](t) + (c_2(t_1)-\la) ( \mathcal M[R](t_1) - \mathcal M[R](t) ) = \\
& & \; = \la_0 M[Q] \Big[  c_2^{2\theta +1}(t) -  c_2^{2\theta +1}(t_1)   - \frac{c_2(t_1)}{ \la_0} [ c_2^{2\theta}(t) -c_2^{2\theta}(t_1) ] \Big] 
+ O(e^{-\ve \ga t_1}).
\eee
To obtain the last estimate (\ref{dE02}) we perform a Taylor development up to the second order (around $y=y_0$) of the function $g(y):= y^{\frac {2\theta+1}{2\theta}}$; and where $y:= c_2^{2\theta}(t)$ and $y_0 := c_2^{2\theta}(t_1)$. Note that $\frac{2\theta+1}{2\theta} = \frac{1}{\la_0}$ and $y_0^{1/2\theta} = c_2(t_1)$. The conclusion follows at once.
\end{proof}

Now our objective is to estimate the quadratic term involved in (\ref{dE02}). Following \cite{MMT}, we should use a ``mass conservation'' identity. However, since the mass is not conserved, we need to combine (\ref{hM3})-(\ref{Mback}) in order to obtain the desired estimate.

\begin{lem}[Quadratic control on the variation of $c_2(t)$]\label{fin}
\bea\label{dE02b}
& & \abs{E_a[R](t_1) -E_a[R](t)  +  (c_2(t_1)-\la) (\mathcal M[R](t_1) - \mathcal M[R](t)) }   \qquad  \nonumber \\
& & \qquad \qquad \qquad  \leq K\|z(t)\|_{H^1(\R)}^4 + K\|z(t_1)\|_{H^1(\R)}^4 +K e^{-\ve\ga t_1}. 
\eea
\end{lem}
\begin{proof}
From (\ref{hM3}) and (\ref{10a}) we have for all $t\in [t_1, T^*]$,
$$
\hat M [R](t) - \hat M[R](t_1)  \leq    \hat M[z](t_1) - \hat M [z](t) + K e^{-\ve \ga t_1}( \|z(t_1)\|_{L^2(\R)} + \|z(t)\|_{L^2(\R)}),
$$
namely
$$
c_2^{2\theta}(t)- c_2^{2\theta}(t_1)  \leq   K\|z(t)\|_{L^2}^2  +K\|z(t_1)\|_{L^2}^2 + K (1+D_0 \ve^{1/2}) e^{-\ve \ga t_1}.
$$
On the other hand, from (\ref{decr}) one has
$$
c_2^{2\theta}(t)- c_2^{2\theta}(t_1)  \geq  - K e^{-\ve \ga t_1} (1+ D_0 \ve^{1/2})- K\|z(t)\|_{L^2}^2  - K\|z(t_1)\|_{L^2}^2.
$$
Combining both inequalities, we obtain
$$
\abs{\Big[\frac{c_2(t)}{c_2(t_1)} \Big]^{2\theta} -1} \leq  K  \|z(t)\|_{L^2(\R)}^2  +K \|z(t_1)\|_{L^2(\R)}^2  +  K (1+D_0\ve^{1/2}) e^{-\ga\ve t_1}.
$$
Plugin this estimate in (\ref{dE02}) and taking $\ve$ even smaller, we get the conclusion.
\end{proof}

\subsubsection{Energy estimates}

Let us now introduce the second order functional 
\bee
\mathcal F_2(t) & := & \frac 12\int_\R \big\{ z_x^2 + [\la+ (c_2(t_1)-\la ) \frac{1}{a(\ve x)}] z^2 \big\} \\
& & \quad  -\frac 1{m+1} \int_\R a(\ve x) [ (R+z)^{m+1}-R^{m+1}-(m+1)R^m z ].
\eee
This functional, related to the Weinstein functional, have the following properties.

\begin{lem}[Energy expansion]\label{EE3}~

Consider $E_a[u]$ and $\mathcal M[u]$ the energy and mass defined in (\ref{Ea})-(\ref{Mback}). Then we have for all $t\in [t_1,T^*]$,
$$
E_a[u](t) +  (c_2(t_1) -\la)\mathcal M[u](t)  =  E_a[R] + ( c_2(t_1)-\la) \mathcal M[R]  + \mathcal F_2(t)   + O(e^{-\ga \ve  t}\|z(t)\|_{H^1(\R)}).
$$
\end{lem}
\begin{proof}
Using the orthogonality condition (\ref{10a}), we have
\bee
E_a[u](t) &= & E_a[R] - \int_\R z 
  (a(\ve x)-1)  R^m + \frac 12 \int_\R z_x^2  + \frac \la 2 \int_\R z^2 \\
  & & \quad  -\frac 1{m+1} \int_\R a(\ve x) [ (R+z)^{m+1}-R^{m+1}-(m+1)R^m z ].
\eee
Moreover, following (\ref{dec1}), we easily get
$$
\abs{\int_\R z (a(\ve x)-1)  R^m } \leq K e^{-\ga \ve  t}\|z(t)\|_{H^1(\R)}.
$$
Similarly,
\bee
 \mathcal M[u](t)  & =  & \mathcal M[R] +  \mathcal M[z] + \int_\R ( \frac 1{a(\ve x)} -1) R z \\
 & = & \mathcal M[R] +  \mathcal M[z] +O(e^{-\ve\ga t} \|z(t)\|_{H^1(\R)}).
\eee 
Collecting the above estimates, we have
\bee
& & E_a[u](t) +  (c_2(t_1) -\la)\mathcal M[u](t) = \\
& & \quad  E_a[R] + ( c_2(t_1)-\la) \mathcal M[R]  + \frac 1{2}\int_\R \Big\{ z_x^2 + [ \frac {(c_2(t_1)-\la)}{a(\ve x)} + \la ]z^2\Big\} \\
& & \quad -  \int_\R \frac{a(\ve x)}{m+1} [ (R+z)^{m+1}-R^{m+1}-(m+1)R^m z ]+ O(e^{-\ga \ve  t}\|z(t)\|_{H^1(\R)}). 
\eee
This concludes the proof.
\end{proof}

\medskip

\begin{lem}[Modified coercivity for $\mathcal F_2$]\label{Coer3}~

There exists $\ve_0>0$ such that for all $0<\ve<\ve_0$ the following hold. There exist $K,\nu_0>0$, independent of $K^*$ such that for every $t\in [t_1, T^*]$
\be\label{Co3}
\mathcal F_2(t) \geq \nu_0 \|z (t)\|_{H^1(\R)}^2 
  - K \ve e^{-\ga\ve t} \|z (t)\|_{L^2(\R)}^2 +O( \|z(t)\|_{L^2(\R)}^3). 
\ee
\end{lem}

\begin{proof}
First of all, note that 
$$
\mathcal F_2(t)  =  \frac 1{2}\int_\R \Big\{ z_x^2 + [\frac {(c_2(t_1)-\la)}{a(\ve x)} + \la ]z^2 - m Q_{c_2}^{m-1}z^2 \Big\}    + O(\|z(t)\|_{H^1(\R)}^3) + O( e^{-\ga\ve t}\|z(t)\|_{H^1(\R)}^2 ).
$$
Since $\frac {(c_2(t_1)-\la)}{ a(\ve x)} + \la \geq c_2(t_1)$ for all $x\in \R$, we have
$$
\mathcal F_2(t)  =  \frac 1{2}\int_\R (z_x^2 + c_2(t_1) z^2 - m Q_{c_2}^{m-1}z^2) + O(\|z(t)\|_{H^1(\R)}^3) + O( e^{-\ga\ve t}\|z(t)\|_{H^1(\R)}^2 ).
$$
From Lemma \ref{surL} and (\ref{10a})-(\ref{11a}) we finally obtain (\ref{Co3}).
 
\end{proof}

\subsubsection{Conclusion of the proof} Now we prove that our assumption $T^*<+\infty$ leads inevitably to a contradiction. Indeed, from Lemmas \ref{EE3} and \ref{Coer3}, we have for all $t\in [t_1, T^*]$ and for some constant $K>0,$
\bee
\frac 1K \|z(t)\|_{H^1(\R)}^2 & \leq &    E_a[u](t)-E_a[u](t_1) +  (c_2(t_1) -\la)[ \mathcal M[u](t)-\mathcal M[u](t_1)] \\
& &  +  E_a[R](t_1) -E_a[R](t) + ( c_2(t_1)-\la) [ \mathcal M[R](t_1)- \mathcal M[R](t)]\\
& &    + K \mathcal F_2(t_1)  +  K \ve \sup_{t\in [t_1, T^*]} e^{-\ga\ve t} \|z (t)\|_{L^2(\R)} + K\sup_{t\in [t_1, T^*]} \|z(t)\|_{L^2(\R)}^3. 
\eee
From Lemmas \ref{3Dr} and \ref{C2}, Corollary \ref{fin} and the energy conservation we have
\bee
\|z(t)\|_{H^1(\R)}^2 & \leq & K \ve  +  (c_2(t_1) -\la)[ \mathcal M[u](t)- \mathcal M[u](t_1)] \\
& &  \qquad +  K \sup_{t\in [t_1, T^*]} \|z(t)\|_{H^1(\R)}^4 + K e^{-\ve\ga t_1}(1+D_0\ve^{1/2})  + K D_0^3\ve^{3/2}. 
\eee
Finally, from (\ref{hM}) we have $ \mathcal M[u](t)- \mathcal M[u](t_1)\geq -Ke^{-\ga \ve t_1}$. Collecting the preceding estimates we have for $\ve>0$ small and $D_0=D_0(K)$ large enough
$$
\|z(t)\|_{H^1(\R)}^2 \leq \frac 14D_0^2 \ve,
$$
which contradicts the definition of $T^*$. The conclusion is that 
 $$
 \sup_{t\geq t_1} \| u(t) - Q_{c_2(t)}(\cdot -\rho_2(t)) \|_{H^1(\R)} \leq K \ve^{1/2}.
 $$
Using (\ref{11a}), we finally get (\ref{S}). This finishes the proof.

\bigskip

\medskip

\noindent
{\bf Proof of the asymptotic stability result.} In this paragraph we sketch the proof of asymptotic stability property in the case $c_\infty(\la) <\la$, namely $\tilde \la<\la<1$, which is the case of the reflected solitary wave. A detailed proof for the case $0<\la<\la_0$ can be found in \cite{Mu2}, which adapts without modifications to the case $\la_0<\la<\tilde \la$. 

\medskip

Let us consider the remaining case, $\tilde \la<\la<1$. We continue with the notation introduced in the proof of the stability property (\ref{Sb}). From the above mentioned stability result, it is easy to check that the decomposition (\ref{eta1a}) showed in Lemma \ref{3Dr} and all its conclusions hold \textbf{for all time $t\geq t_1$}. 

\medskip

Consider $-\la<\beta <\frac{11}{10}(c_\infty(\la)-\la)$, and let us follow the proof described in \cite{Mu2}. First of all, the Virial estimate (cf. Lemma 6.4 in \cite{Mu2}) holds with no important modifications.

\smallskip
Second, Lemma 6.8, about monotonicity for mass and energy, needs some modifications. Indeed, for $x_0 >0$ we consider, for $t,t_0\geq t_1$, and $\tilde y(x_0):= x- (\rho_2(t_0) + \sigma (t-t_0) + x_0 )$, the modified quantities
\be\label{I0}
I_{x_0,t_0}(t) := \int_\R a^{1/m} (\ve x) u^2(t,x) \phi ( \tilde y(x_0)) dx, \quad \tilde I_{x_0,t_0}(t) :=  \int_\R a^{1/m} (\ve x) u^2(t,x) \phi (\tilde y(-x_0)) dx,
\ee
and
$$
J_{x_0,t_0}(t) :=  \int_\R[ u_x^2  +  a^{1/m}(\ve x)u^2 -\frac {2a(\ve x)}{m+1} u^{m+1}](t,x) \phi (\tilde y(x_0)) dx,
$$
with $\phi(x) := \frac 2\pi \arctan (e^{x/K}).$ Here $\sigma \in (-\la,\frac{11}{10}(c_\infty(\la)-\la)) $ is a fixed quantity, to be chosen later.

\smallskip

First of all, note that  the equivalent of estimate (6.32) is a consequence of the following inequality, valid for $K_0>0$ large and $\ve$ small enough:
\bea\label{dT}
\frac 12\partial_t \int_\R a^{1/m}(\ve x) \phi(\tilde y(x_0)) u^2 & =&  -\frac 32 \int_\R a^{1/m}\phi'  u_x^2   + \frac{m}{m+1} \int_\R a^{1/m +1}(\ve x)\phi' u^{m+1} \nonu\\
& &  +\frac 12 \int_\R u^2 a^{1/m}(\ve x) \big[ -(\sigma + \la)\phi '   + \phi^{(3)} \big] \nonu \\
& &  -\frac 32\ve \int_\R (a^{1/m})' (\ve x) \phi  u_x^2 - \frac \ve 2 \int_\R u^2 [ \la (a^{1/m})'   - \ve^2 (a^{1/m})^{(3)} ](\ve x) \phi \nonu \\
& & + \frac 32\ve \int_\R u^2 \big[ \ve(a^{1/m})^{(2)}(\ve x) \phi'  +  (a^{1/m})' (\ve x)\phi''   \big].
\eea
In this last computation we have six terms. Let us see each one in detail. In what follows we use the decomposition (\ref{eta1a}).
First of all, one has
$$
 \int_\R \phi'  a^{1/m}u_x^2 =  \int_\R \phi'  a^{1/m}(R_x^2 + 2R_x z_x + z_x^2). 
$$
Recall that $R(t)$ is exponentially decreasing in $x-\rho(t)$. On the other hand, $\phi'(\tilde y)$ is exponentially decreasing away from zero. Therefore, one has, for  $K$ large,
$$
\int_\R  a^{1/m}\phi'  u_x^2 =  \int_\R a^{1/m}\phi'   z_x^2 + O (e^{-x_0/K} e^{- (t_0 - t)/K}).
$$
Similarly,
$$
\abs{ \int_\R a^{1/m+1} \phi' u^{m+1}}   \leq K e^{-(t_0-t)/K} e^{-x_0/K} + K \ve^{(m-1)/2} \int_\R a^{1/m} \phi'  z^2.
$$
On the other hand, since $\sigma+ \la>0,$
$$
 \int_\R a^{1/m} u^2 \big[ -(\sigma + \la)\phi'   +\phi^{(3)} \big]  = - \frac 12(\sigma + \la)\int_\R a^{1/m} \phi' z^2 + O(e^{-x_0/K} e^{- (t_0-t)/K}),
$$
and
$$
  -\frac 32\ve \int_\R (a^{1/m})' (\ve x) \phi  u_x^2 - \frac \ve 2 \int_\R u^2 [ \la (a^{1/m})'   - \ve^2 (a^{1/m})^{(3)} ](\ve x) \phi \leq 0,
$$
provided $\ve$ is small.
Finally,
$$
\abs{\frac 32\ve \int_\R u^2 \big[ \ve(a^{1/m})^{(2)}(\ve x) \phi'  +  (a^{1/m})' (\ve x)\phi''   \big]} \leq K \ve e^{-(t_0-t)/K} e^{-x_0/K}  + K\ve \int_\R a^{1/m} z^2 \phi' . 
$$
After these estimates, it is easy to see that
$$
\frac 12\partial_t \int_\R a^{1/m}(\ve x)\phi(y)u^2 \leq K  e^{-(t_0-t)/K} e^{-x_0/K}.
$$
The conclusion follows after integration in time: one has, for all $0<\ve<\ve_0$ and for all $t,t_0\geq t_1$ with $t_0\geq t$,
\be\label{I}
I_{x_0,t_0}(t_0) -I_{x_0,t_0}(t) \leq K e^{-x_0/K} .
\ee
This estimate is an improved version of (6.32) in \cite{Mu2}. On the other hand, to obtain (6.33), we perform a similar computation. Therefore, if $t\geq t_0$, one has
\be\label{tI}
\tilde I_{x_0,t_0}(t) -\tilde I_{x_0,t_0}(t_0) \leq K e^{-x_0/K}.
\ee
Finally if $t_0\geq t$, after a similar computation as performed in \cite{Mu2},
\be\label{J}
J_{x_0,t_0}(t_0) -J_{x_0,t_0}(t) \leq Ke^{-x_0/K}.
\ee
From these estimates, the Virial identity and the decomposition above mentioned, one has (6.35)-(6.39). The rest of the proof is direct, and no deep modifications are needed. The proof is complete.

\bigskip

\section{Proof of Proposition \ref{prop:decomp}}\label{A}

\medskip

This section is an improvement of the Appendix A in \cite{Mu2}. Now we suppose that the parameters $(c(t),\rho(t))$ are not fixed, but satisfy (\ref{r1}).

\medskip

\noindent
{\bf Step 0. Preliminaries.} 

From  (\ref{2.2bis}), we easily have that 
\begin{equation}\label{eq:sion}
S[\tilde u]= \bf I +II +III,
\end{equation}
where (we omit the dependence on $t,x$) 
\be\label{eq:sion1}
{\bf I} := S[R], \quad {\bf II} = {\bf II}(w) := w_t  +  (w_{xx} -\la w + m\ a(\ve x) R^{m-1} w)_x , 
\ee
and for $m=2,3$ or $4$,
\be\label{eq:sion2}
{\bf III} : = \left\{ a(\ve x) [(R + w)^m - R^m - mR^{m-1}w ]\right\}_x.
\ee
Recall that $w$ is given by (\ref{defW}). Since $w$ varies, depending on $m=3$ or $m\neq 3$, we have to consider two different cases in our computations.

In the next results, we expand the terms in \eqref{eq:sion}. Note that $\tilde a = a^{\frac 1{m-1}}$, and
$$
R(t,x) = \frac{Q_{c(t)} (y)}{\tilde a(\ve \rho(t))}, \quad y= x-\rho(t). 
$$

\smallskip

\noindent
{\bf Step 1. Computation of ${\bf I}$.}

\begin{lem}\label{lem:SQ}~

\begin{enumerate}
\item Suppose $m=2$ or $4$. One has
\be\label{eq:SQ}
{\bf I}= F_0^{\bf I}(t,y) + \ve  F_1^{\bf I}(t, y)  + \ve^2  F_c^{\bf I}(t,y),
\ee
where 
\be\label{F0I}
F_0^{\bf I}(t,y) := ( c'(t) -\ve f_1(t) ) \partial_c R(t) + (\rho'(t) - c(t)+ \la -\ve f_2(t))  \partial_\rho R(t),
\ee
$f_1(t)$ and $f_2(t)$ are given by (\ref{f1})-(\ref{f2}), and
\bea\label{F1Q}
F_1^{\bf I}(t; y) & := & f_1(t) \frac{\Lambda Q_c(y)}{ \tilde a(\ve \rho(t))}  -   \frac {\tilde a'}{\tilde a^2}(\ve \rho(t)) (c(t)-\la) Q_c(y) \nonu \\ 
& & \qquad +  \frac{ a' }{\tilde a^m} (\ve \rho(t))(yQ_c^m(y))_y - f_2(t) \frac{Q_c'(y)}{\tilde a(\ve \rho(t))} .
\eea
Finally, for all $t\in [-T_\ve, \tilde T_\ve] $, one has  $\|F_c^{\bf I}(t,\cdot)\|_{H^1(\R)} \leq K(e^{-\ve\ga|\rho(t)|} +\ve) $.

\medskip

\item Suppose now $m=3$. Then one has
\be\label{eq:SQ3}
{\bf I}= F_0^{\bf I}(t,y) + \ve  F_1^{\bf I}(t, y)  + \ve^2  F_2^{\bf I}(t,y) + \ve^3 F_c^{\bf I}(t,y),
\ee
with $F_0^{\bf I}$ given by
\be\label{F0I3}
F_0^{\bf I}(t,y) := ( c'(t) -\ve f_1(t) -\ve^2 f_3(t) ) \partial_c R(t) + (\rho'(t) - c(t)+ \la -\ve^2 f_4(t))  \partial_\rho R(t),
\ee
and $f_1(t), f_3(t)$ and $f_4(t)$ given by (\ref{f1}), (\ref{f3}) and (\ref{f4}) respectively. In addition, $F_1^{\bf I}$ is given by (\ref{F1Q}) (with $f_2\equiv 0$), and
\be\label{F2Q}
F_2^{\bf I}(t,y) := \frac{f_3(t)}{a^{1/2}(\ve \rho(t))} \Lambda Q_c(y)  - \frac{f_4(t)}{a^{1/2}(\ve \rho(t))} Q_c'(y) + \frac{ a'' }{2 a^{3/2}}(\ve \rho(t))(y^2Q_c^3(y))_y .
\ee
Finally,  for all $t\in [-T_\ve, \tilde T_\ve] $, one has $\|F_c^{\bf I}(t,\cdot)\|_{H^1(\R)} \leq K (e^{-\ve\ga|\rho(t)|} +\ve)$.
\end{enumerate}
\end{lem}

\medskip

\begin{proof}[Proof of Lemma \ref{lem:SQ}.]~

We  compute (from now on, and for the sake of simplicity, we avoid the explicit dependence in time $t$ and space $y$ in the computations):
\bee
{\bf I} & = & R_t + (R_{xx}  -\la R+ a(\ve x) R^m )_x\\
& =&  \frac{ c'}{\tilde a}\Lambda Q_c - \frac{\rho'}{\tilde a}  Q_c' - \ve \frac {\tilde a' \rho'}{\tilde a^2} Q_c + \frac 1{\tilde a} Q_c^{(3)} -\frac{\la}{\tilde a}  Q_c' +  \frac{1}{\tilde a^m}( a(\ve x)Q_c^m)_x.
\eee
Note that via a Taylor expansion,
\bee
( a(\ve x)Q_c^m)_x  =  a(\ve \rho) (Q_c^m)_y + \ve a'(\ve \rho) (yQ_c^m)_y + \frac 12 \ve^2 a''(\ve \rho) (y^2 Q_c^m)_y   + O_{H^2(\R)}(\ve^3).
\eee
Therefore, using the equation satisfied by $Q_c$, namely, $Q_c'' -cQ_c +Q_c^m=0$, one has
\bee
{\bf I} &= & \frac{c'}{\tilde a}\Lambda Q_c - \frac{\rho'}{\tilde a}  Q_c' - \frac{\ve}{m-1} \frac {a' \rho'}{\tilde a^m} Q_c + \frac 1{\tilde a} Q_c^{(3)} -\frac{\la}{\tilde a}  Q_c'  \\
& & \qquad + \frac 1{\tilde a} (Q_c^m)' + \frac{\ve a' }{\tilde a^m}(yQ_c^m)_y   + \frac{\ve^2 a'' }{2\tilde a^m}(y^2Q_c^m)_x + O_{H^1(\R)}(\ve^3) \\
& = & \frac 1{\tilde a} (Q_c''- cQ_c + Q_c^m)'  +  \frac{c'}{\tilde a}\Lambda Q_c - (\rho' -c+\la) \frac{Q_c'}{\tilde a}
- \ve \frac {\tilde a'}{\tilde a^2} (\rho' -c+\la)Q_c  \\
& & \qquad - \ve \frac {\tilde a'}{\tilde a^2} ( c-\la)Q_c +  \frac{\ve a' }{\tilde a^m}(yQ_c^m)_y +  \frac{\ve^2 a'' }{2\tilde a^m}(y^2Q_c^m)_y + O_{H^1(\R)}(\ve^3) \\
&  = & (c' -\ve f_1- \delta_{m,3}\ve^2 f_3) \frac{\Lambda Q_c}{\tilde a} - (\rho' -c+\la -\ve f_2 - \delta_{m,3}\ve^2 f_4) ( \frac{Q_c'}{\tilde a}  +  \ve \frac {\tilde a'}{\tilde a^2} Q_c)  \\
& & \qquad +  \ve \big[ \frac{f_1}{\tilde a}\Lambda Q_c - \frac {\tilde a'}{\tilde a^2} (c-\la)Q_c +\frac{ a' }{\tilde a^m}(yQ_c^m)_y  -  \frac{f_2}{\tilde a} Q_c' \big] \\
& & \qquad + \ve^2 F_2^{\bf I}(t,y) + O_{H^1}(\ve^3 e^{-\ve\ga|\rho(t)|} + \ve^4),
\eee
with $F_2^{\bf I}$ given by (\ref{F2Q}), and $\delta_{m,3}$ the Kronecker delta symbol. Moreover  $F_2^{\bf I}(t,y)\in \mathcal Y$ for all $t\in [-T_\ve, \tilde T_\ve]$ and
$$
\| F_2^{\bf I}(t,\cdot) \|_{H^1(\R)} \leq K e^{-\ve\ga|\rho(t)|}.
$$
From the last identity above, we define $F_0^{\bf I}$ and $F_1^{\bf I}$ as above mentioned (cf. (\ref{F0I})-(\ref{F0I3})-(\ref{F1Q}).) Moreover,  depending on the value of $m$, we define $F_c^{\bf I}$ as the rest term of quadratic or cubic order  in $\ve$. Indeed, for $m=3$ we have $F_c^{\bf I}(t,\cdot) = O_{H^1}(\ve^3 e^{-\ve\ga|\rho(t)|} + \ve^4)$, and for $m=2$ or $4$, we have $F_c^{\bf I} =\ve^2 F_2^{\bf I}(t,y) + O_{H^1}(\ve^3 e^{-\ve\ga|\rho(t)|} + \ve^4)$. In both cases, the corresponding estimates, and the decompositions (\ref{eq:SQ})-(\ref{eq:SQ3}) are straightforward. The proof is complete.
\end{proof}

\noindent
{\bf Step 2. Computation of ${\bf II}$.}

\begin{lem}[Decomposition of {\bf II}]\label{lem:dSKdVw}~

Suppose that $(A_c, B_c)$ satisfy (\ref{Ac})-(\ref{Bc}).\footnote{We assume these properties in order to simplify the computations. Later, we will prove that this is indeed the case.} Let $w$ given by (\ref{defW}). The following expansions hold:
\begin{enumerate}
\item \emph{Case $m=2,4$}. We have
\bee
{\bf II} & = &  (c' -\ve f_1)\partial_c w - (\rho' -c +\la-\ve f_2) w_y  - (\mathcal{L} w)_y \\
& &  + \ \ve^2 \big[ \frac 1\ve d' A_{c} +  f_1 d \ \partial_c A_{c}  \big] + \ve^2  F_c^{\bf II}(t; y),
\eee
with $F_c^{\bf II}(t; \cdot) \in \mathcal Y$, uniformly in time. In addition, 
$$
\| F_c^{\bf II}(t; y)\|_{H^1(\R)}\leq K  e^{-\ga\ve|\rho(t)|}.
$$
\item \emph{Case $m=3$.}  Here one has
\bee
{\bf II} & = &  (c' -\ve f_1 -\ve^2 f_3)\partial_c w - (\rho' -c +\la -\ve^2 f_4)  w_y   - (\mathcal{L} w)_y \\
& &  +\ \ve^2 \big[ \frac 1\ve d' A_{c} + f_1 d \  \partial_c A_{c}  + 3 \, d \frac{a'(\ve \rho)}{a(\ve \rho)} ( yQ_c^{2}A_{c})_y \big]  \\
& & + \  \ve^3 \big[ d \ \! f_3  \partial_c A_c +f_1 \partial_c B_c -f_4 (B_c)_y \big]  + \ve^4 f_4 \partial_c B_c + \ve^3  F_c^{\bf II}(t; y),
\eee
with $F_c^{\bf II}(t; \cdot) \in \mathcal Y$, uniformly in time. In addition, 
$$
\| F_c^{\bf II}(t; y)\|_{H^1(\R)}\leq K  e^{-\ga\ve|\rho(t)|}.
$$
\end{enumerate}
\end{lem}

\begin{proof}
Let $D:=D_c(t,y)$, $y= x-\rho(t)$, be a general, smooth function. We compute
$$
{\bf II}(D) := D_t + (D_{xx} -\la D +m \ a(\ve x)  R^{m-1} D)_x.
$$
We have
\bee
{\bf II}(D)& =& c'(t) \partial_c D + D_t  - \rho'(t)D_y  +\big[ D_{yy} - \la D+  \frac{a(\ve x)}{a(\ve \rho)} mQ_c^{m-1} D \big]_x \\
& = &   D_t  -(\mathcal L D)_y    + (c'(t) - \ve  f_1(t) -\ve^2 \delta_{m,3} f_3(t)) \partial_c D   \\
& &  - (\rho'(t) -c(t)+\la -\ve f_2(t) -\ve^2 \delta_{m,3} f_3(t)) D_y \\
& & + m \ve \frac{a'(\ve \rho)}{a(\ve \rho)} ( yQ_c^{m-1} D )_y  +  O((\ve^2 y^2Q_c^{m-2}D )_y) \\
& & + \ve ( f_1(t) + \ve \delta_{m,3}f_3(t) ) \partial_c D - \ve ( f_2(t)  + \ve \delta_{m,3}f_4(t)) D_y.
\eee

\medskip

We apply this identity to the functions $w=\ve d(t)A_c(y)$ (case $m=2,4$) and $w=\ve d(t)A_c(y)+ \ve^2 B_c(t,y)$ (case $m=3$).
We first deal with the cases $m=2$ or $4$. We have
\bee
{\bf II}(w) & = &  \ve d'(t) A_c  - \ve d(t)(\mathcal L A_c)'    + (c'(t) - \ve  f_1(t)) \ve d(t)\partial_c A_c   \\
& &  - (\rho'(t) -c(t)+\la -\ve f_2(t) ) \ve d(t) A_c' \\
& & + \ve^2 d (t) f_1(t) \partial_c A_c  + O_{H^1(\R)} (\ve^2 e^{-\ve \ga |\rho(t)|} ).
\eee
(Recall that $A_c'\in \mathcal Y$.) This proves the first part of Lemma \ref{lem:dSKdVw}.

\medskip

We treat now the cubic case, $m=3$. Here we have $f_2(t) \equiv 0$, $A_c'\in \mathcal Y$ and
\bee
{\bf II}(w)& = &  \ve d'(t) A_c +  \ve^2 (B_c)_t  -(\mathcal L w)_y    + (c'(t) - \ve  f_1(t) -\ve^2 f_3(t)) \partial_c w   \\
& &  - (\rho'(t) -c(t)+\la  -\ve^2 f_4(t)) w_y  + 3 \ve \frac{a'(\ve \rho)}{a(\ve \rho)} ( yQ_c^{2} w )_y \\
& &  +  O((\ve^2 y^2 Q_c w )_y)  + \ve ( f_1(t)  + \ve f_3(t)) \partial_c w - \ve^2 f_4(t) w_y\\
& =& -(\mathcal L w)_y    - (\rho'(t) -c(t)+\la  -\ve^2 f_4(t)) w_y\\
& &  + \ (c'(t) - \ve  f_1(t) -\ve^2 f_3(t)) \partial_c w   +   \ve d'(t) A_c +  \ve^2 (B_c)_t \\
& & +\ \ve^2 ( f_1(t)  + \ve f_3(t)) \partial_c (d(t)A_c + \ve B_c)   - \ve^3 f_4(t) (B_c)_y  \\
& & + \  3 \ve^2 d(t)\frac{a'(\ve \rho)}{a(\ve \rho)} ( yQ_c^{2} A_c )_y + O_{H^1(\R)} (\ve^3 e^{-\ve\ga|\rho(t)|}).
\eee 
Therefore, we have
\bee
{\bf II}(w)& = & -(\mathcal L w)_y    + (c'(t) - \ve  f_1(t) -\ve^2 f_3(t)) \partial_c w  - (\rho'(t) -c(t)+\la  -\ve^2 f_3(t)) w_y\\
& & +  \ \ve^2 \big[ \frac 1\ve d'(t) A_c +   (B_c)_t  + f_1(t)  d(t) \partial_c A_c +  3 d(t)\frac{a'(\ve \rho)}{a(\ve \rho)} ( yQ_c^{2} A_c )_y  \big]\\
& & + \  \ve^3\big[ d(t)f_3(t)  \partial_c A_c +f_1(t) \partial_c B_c -f_4(t) (B_c)_y \big]  \\
&& +\ \ve^4 f_4(t) \partial_c B_c + O_{H^1(\R)} (\ve^3 e^{-\ve\ga|\rho(t)|}).
\eee
This concludes the proof.
\end{proof}

\medskip

\noindent
{\bf Step 3. Nonlinear term.} 

\begin{lem}[Decomposition of {\bf III}]\label{lem:SintIII}~

Suppose that $(A_c, B_c)$ satisfy (\ref{Ac})-(\ref{Bc}). Then we have
\be\label{33}
{\bf III} = \begin{cases} O_{H^1(\R)} ( \ve^2 e^{-\ve\ga|\rho(t)|}), \qquad m=2,4;\\
3\ve^2 a(\ve \rho)d^2(t) (Q_c A_c^2)'  + 3 \ve^4 a(\ve x)(d(t) A_c +\ve B_c )^2  B_c'  +  \ O_{H^1(\R)} (\ve^3 e^{-\ve\ga |\rho(t)|}), \quad m=3.
\end{cases}
\ee
\end{lem}

\begin{proof}
First of all, define ${\bf \tilde{III} } := a(\ve x)[ (R+ w)^m - R^m - m R^{m-1} w]$. We consider separate cases.

Let us suppose $m=2$ or $4$. In these cases, we have $w(t)=d(t) A_c(y)$. Therefore, 
$$
{\bf \tilde{III} }  = 
\begin{cases}
 \ve^2  d^2(t)a(\ve x) A_c^2 \quad \hbox{ if } m=2;\\
\ve^2 a(\ve x)d^2(t) A_c^2[6   Q_c^2 + 4\ve  d(t) Q_c A_c + \ve^2 d^2(t) A_c^2], \quad \hbox{ in the case }  m=4.
\end{cases} 
$$ 
Thus taking space derivative we obtain
\bee
{\bf {III} } &  = &  \ve^{m+1} a'(\ve x)  d^m(t) A_c^m  + O_{H^1(\R)}(\ve^2 e^{-\ve\ga|\rho(t)|}) \\
& =&  O_{H^1(\R)}(\ve^{m+\frac 12} e^{-\ve\ga|\rho(t)|}  + \ve^2 e^{-\ve\ga|\rho(t)|} ) = O_{H^1(\R)}(\ve^2 e^{-\ve\ga|\rho(t)|}).
\eee
Note that $ (A_c^m)'\in \mathcal Y$ because $A_c$ satisfies (\ref{Ac}).

\smallskip

Suppose now $m=3$. We have $w(t,x)= \ve d(t) A_c(y) + \ve^2 B_c(t,y)$, and
${\bf \tilde{III} }  =    a(\ve x) [ 3 Q_c  w^2  + w^3].$
From this identity we get
\bee
{\bf III }  & = & 3\ve^2 a(\ve \rho) d^2(t)(Q_c A_c^2)' +   \ve a'(\ve x) w^3 + 3 a(\ve x)w^2 w_x  + O_{H^1(\R)} (\ve^3 e^{-\ve\ga |\rho(t)|})\\
& =&  3\ve^2 a(\ve \rho)d^2(t) (Q_c A_c^2)' +  \ve^4 a'(\ve x) ( d(t) A_c +\ve B_c)^3  \\
& & \quad +\ 3 \ve^3 a(\ve x)(d(t) A_c +\ve B_c )^2 (d(t) A_c' + \ve B_c') + O_{H^1(\R)} (\ve^3 e^{-\ve\ga |\rho(t)|}).
\eee
The first term above is of second order, so we keep it. The second term in the last identity is in $H^1(\R)$ and it can be estimated as follows:
\bee
& & \ve^4 a'(\ve x) ( d(t) A_c +\ve B_c)^3= \\
& & \qquad =\ve^4 a'(\ve x)(d^3A_c^3 + 3\ve d^2 A_c^2 B_c + 3\ve^2 dA_c B_c^2 + \ve^3 B_c^3) \\
& &  \qquad =  O_{H^1(\R)} (\ve^{7/2} e^{-\ve\ga |\rho(t)|}) + O_{L^\infty(\R)} (\ve^5 a'(\ve x)(|y| + \ve y^2 + \ve^2|y|^3)e^{-\ve\ga |\rho(t)|}) \\
 & & \qquad = O_{H^1(\R)} (\ve^{7/2} e^{-\ve\ga |\rho(t)|})  + O_{H^1(\R)} (\ve^{9/2} (|\rho(t)| + \ve |\rho(t)|^2 + \ve^2|\rho(t)|^3)e^{-\ve\ga |\rho(t)|}).
\eee
Since we assume (\ref{r1}), we have $|\rho(t)| \leq KT_\ve$ inside the interval $[-T_\ve, \tilde T_\ve]$, which gives
$$
 \ve^4 a'(\ve x) ( d(t) A_c +\ve B_c)^3  = O_{H^1(\R)} (\ve^{7/2-3/100}e^{-\ve\ga |\rho(t)|}).
$$
Finally,
\bee
3 \ve^3 a(\ve x)(d(t) A_c +\ve B_c )^2 (d(t) A_c' + \ve B_c')   =   3 \ve^4 a(\ve x)(d(t) A_c +\ve B_c )^2  B_c'  +  O_{H^1(\R)} (\ve^3 e^{-\ve\ga |\rho(t)|}).
\eee
Collecting all these estimates, we finally obtain (\ref{33}).
\end{proof}

\medskip

\noindent
{\bf Step 4. First conclusion.}
Now we collect the estimates from Lemmas \ref{lem:SQ}, \ref{lem:dSKdVw} and \ref{lem:SintIII}. We obtain that, for all $t\in [-T_\ve, \tilde T_\ve]$,
\bea\label{Stt}
S[\tilde u] & = &   (c'(t) - \ve f_1(t) -\ve^2 \delta_{m,3} f_3(t))\partial_c\tilde u  \nonu\\
& & +\ (\rho'(t) -c(t)+ \la -  \ve f_2(t) -\ve^2 \delta_{m,3}f_4(t)) \partial_\rho \tilde u + \tilde S[\tilde u], 
\eea
with $\partial_\rho \tilde u := \partial_\rho R -w_y $,
\bea\label{tStu}
& & \tilde S[\tilde u]   =    \ve [F_1(t,y) -d(t)(\mathcal L A_{c})_y]   +  O(\ve^2 |\rho'(t) -c(t)+ \la -  \ve f_2(t) |e^{-\ve \ga |\rho(t)|} |A_c| )\nonu \\
& & \quad + \ \ve^2 \big[ (\frac {a'}{\tilde a^m})' (\ve \rho)(c-\la) A_{c} +  f_1 d \ \partial_c A_{c}  \big] +   \ve^2O_{H^1(\R)}(e^{-\ve \ga |\rho(t)|} +\ve),  \label{a10}
\eea
for the cases $m=2$ and $4$; and for the cubic case,
\bea\label{tStu3}
 \tilde S[\tilde u] & = &    \ve [F_1(t,y) -d(t)(\mathcal L A_{c})_y]  + \ve^2 [ F_2(t,y) - (\mathcal L B_{c})_y]\nonu\\
& & \quad + \ \ve^3 \big[ d(t)f_3(t)  \partial_c A_c +f_1(t) \partial_c B_c -f_4(t) (B_c)_y \big] \label{a11} \\
& & \quad +\  \ve^4 \big[ f_4(t) \partial_c B_c +  3  a(\ve x)(d(t) A_c +\ve B_c )^2  B_c' \big] +  \ve^3 O_{H^1(\R)} (e^{-\ve\ga |\rho(t)|} +\ve).  \label{a12}
\eea
In addition, $f_1 (t), f_2(t), f_3(t)$ and $f_4(t)$ are given by (\ref{f1}), (\ref{f2}), (\ref{f3}) and (\ref{f4}) respectively,
\be\label{F1}
F_1  := F_1^{\bf I} =  \frac{f_1(t)}{\tilde a(\ve \rho)} \Lambda Q_c +\frac{a'}{\tilde a^m} \big[ (yQ_c^m)_y- \frac 1{m-1} (c-\la)Q_c  \big] -  \frac{f_2(t)}{\tilde a(\ve \rho)} Q_c' ,
\ee
(cf. (\ref{F1Q}).) Moreover, for any $t\in [-T_\ve, \tilde T_\ve]$ one has
\be\label{Or}
\int_\R F_1(t, y) Q_c(y) dy= 0.
\ee
(See \cite{Mu2} for a proof of this identity.) On the other hand, 
$F_2$ is given by  
$$
F_2 := \tilde F_2 + O(|\rho'-c +\la -\ve^3 f_4 | (\frac{a'}{a^{3/2}})'(\ve \rho) A_c),
$$
with
\bea\label{F2}
\tilde F_2 & := & (\frac{a'}{a^{3/2}})'(\ve \rho) (c-\la) A_c + f_1(t) \frac{a'}{a^{3/2}}\partial_c A_c + 3 \frac{a'^2}{a^{5/2}} (yQ_c^2A_c)_y + \frac{f_3(t)}{a^{1/2}} \Lambda Q_c \nonu \\
& & \qquad   +\frac{a''}{2a^{3/2}} (y^2 Q_c^3)_y-\frac{f_4(t)}{a^{1/2}}Q_c' + 3\frac{a'^2}{a^{5/2}} (Q_cA_c^2)_y,
\eea
and $|\tilde F_2(t,y)| \leq K e^{-\ve\ga |\rho(t)|}$.
Finally, one has, with the choice of $f_3(t)$ in (\ref{f3}),
\be\label{Or2}
\int_\R \tilde F_2(t, y) Q_c(y) dy= 0,
\ee
for all time $t\in [-T_\ve, \tilde T_\ve]$. (cf. (\ref{OR}) below for the proof.)

\bigskip

\noindent
{\bf Step 5. Resolution of the first linear problem.}

The next step is the resolution of the linear differential equation involving the first order terms in $\ve$. Indeed, from (\ref{tStu})-(\ref{tStu3}), we want to solve
\be\label{Omegaa}
 d(t)(\mathcal L A_{c})_y (y) = F_1(t, y),  \quad \hbox{ for all  } y\in \R, \; \hbox{ and }  \; t\in [-T_\ve, \tilde T_\ve] \hbox{ fixed;}
\ee
with $d(t)$ given by (\ref{dd}). We start with an important remark.

\begin{rem}[Simplified expression for $F_1$]\label{F1simpli}
Note that from (\ref{f1}), (\ref{f2}) and (\ref{F1}) one has
\bea\label{F1new}
F_1(t; y) &  := & \frac{a'}{\tilde a^m}  \Big[ p  c (c -  \frac \la{\la_0 } ) \Lambda Q_c   - \frac 1{m-1} (c-\la)Q_c +(yQ_c^m)' \Big] - \frac{f_2(t)}{\tilde a } Q_c', \nonumber \\
& = & \frac{a'}{\tilde a^m}  \Big[ p  c^2 \Lambda Q_c   - \frac c{m-1} Q_c +(yQ_c^m)'  - 3\la_0 \xi_m \sqrt{c} Q_c' \Big]  \nonumber \\
& &\qquad + \ \la \frac{a'}{\tilde a^m}  \Big[ - \frac {4c}{5-m} \Lambda Q_c   + \frac 1{m-1} Q_c  + \frac{\xi_m}{\sqrt{c}} Q_c' \Big]  \nonu \\
& = :& d(t)( \tilde F_1(t,y) + \la \hat F_1(t,y)).
\eea
\medskip

Compared with the former term $F_1$ described in \cite{Mu2}, now $F_1$ possesses an additional, odd component given by $-\frac{f_2(t)}{\tilde a} Q_c'$, which is orthogonal to $Q_c$ in $L^2(\R)$. The purpose of this term is to obtain a unique solution $A_c$ satisfying the {\bf additional orthogonality condition} $\int_\R A_c Q_c =0.$ Moreover, since $f_2 \equiv 0$ for the cubic case, it will imply that our solution $A_c$ satisfies in this case, this condition for free. 
\end{rem}
From the above remark, we are reduced to solve the following simple problem,
$$
(\mathcal L A_{c})_y (y) =  \tilde F_1(t,y) + \la \hat F_1(t,y),
$$
with $\tilde F_1$ and $\hat F_1$ defined in (\ref{F1new}), and from (\ref{Or}),
$$
\int_\R (\tilde F_1(t,y) + \la \hat F_1(t,y))Q_c(y) =0.
$$

\medskip

Now we introduce the following function, with the purpose of describing the effect of \emph{potential} on the solution. Let $c>0$ and
\be\label{varfi}
\varphi(x):=-\frac {Q'(x)}{Q(x)}, \qquad  \varphi_c (x) := -\frac{Q_c'}{Q_c} = \sqrt{c} \varphi(\sqrt{c} x).
\ee
Note that $\varphi$ is an odd function, and satisfies (see \cite{MMcol2} for more details)
\be\label{surphi}
\lim_{x\to \pm \infty} \varphi(x)=\pm 1; \quad \varphi^{(k)} \in \mathcal{Y}, \; k\geq 1.
\ee

\medskip

We recall the form of the solution $A_c$ that we are looking for. In addition to the simple structure required in \cite{Mu2}, we seek for a bounded solution satisfying
\be\label{eq:st}
A_{c(t)}(y) = \beta_c(t) (\varphi_c(y) - \sqrt{c(t)}) + \hat A_c(y) + \mu_c(t) Q_c'(y) + \delta_c(t) \Lambda Q_c(y),
\ee
for some $\beta_c(t), \mu_c(t),  \delta_c(t) \in \R$, $\varphi$ defined in (\ref{varfi}), and $ \hat A_c\in \mathcal Y$. The parameters $\mu_c$, $\delta_c$ will be chosen in order to find the \emph{unique} solution $A_c$ satisfying some orthogonality conditions. This last fact is one of the key new ingredients for the proof of our main result.

\medskip

\begin{lem}[Solvability of system (\ref{Omegaa}), improved version]\label{lem:omega}~

Suppose $0\leq \la <1$, $\la\neq \tilde \la$, $(c, \rho)$ given by (\ref{c}), and $f_1(t)$, $f_2(t)$ given by (\ref{f1}) and (\ref{f2}) respectively. There exists a solution $A_{c}= A_{c}(y)$ of
\be\label{A10}
(\mathcal{L}A_{c(t)} )_y(y) = \tilde F_1(t,y) + \la \hat F_1(t,y), 
\ee
satisfying, for every $t\in [-T_\ve, \tilde T_\ve]$,  
\bea\label{Acy}
& & A_{c} (y) :=  \beta_c (\varphi_c(y) -\sqrt{c}) + \hat A_{c}(y) + \mu_c Q_c' +  \delta_c \Lambda Q_c(y), \\
& & \lim_{-\infty} A_c = -2\sqrt{c} \beta_c; \quad  |A_{c}(y)|\leq K e^{-\ga y}, \; \hbox{ as } y\to +\infty,   \label{LI}
\eea
with $\hat A_{c} \in \mathcal Y$.  In addition, we have
\be\label{constants}
\beta_c(t) := \frac 1{2c^{3/2}} \int_\R ( \tilde F_1 +  \la \hat F_1)(t) \neq 0.
\ee
Moreover, $A_{c}$ satisfies
\be\label{Or3}
\int_\R A_{c} Q_c =\int_\R A_{c} y Q_c =0.
\ee
\end{lem}

\begin{proof}
First of all, note that from Remark \ref{F1simpli}, we have used the explicit value of $f_1(t)$ and $f_2(t)$ to obtain the simplified linear problem (\ref{A10}). Next, the existence of a solution $A_c \in L^\infty(\R)$ of the form (\ref{Acy}) for this equation was established in \cite{Mu2}, provided 
$$
\int_\R (\tilde F_1(t , y) +\la F_2(t,y))Q_c =0,
$$
which is indeed the case (cf. (\ref{Or}) and Lemma \ref{surL}). The novelty now is the inclusion of the term proportional to $f_2(t)Q_c'$ in (\ref{F1}), which induces the new term $\delta_c \Lambda Q_c$ in (\ref{Acy}) (Note that from Lemma \ref{surL}  $(\mathcal L \Lambda Q_c)' =-Q_c' $.) Furthermore, the limits in (\ref{LI}) are straightforward from (\ref{surphi}).

\medskip

On the other hand, we choose the terms $\mu_c$ and $\delta_c$ in order to satisfy (\ref{Or3}). Since we do not know explicitly $A_c$, we need another method to compute explicitly $f_2(t)$ (and therefore, $\delta_c(t)$.)  Indeed, multiplying (\ref{A10}) by $\int_{-\infty}^y \Lambda Q_c$ and integrating, one has
\be\label{Beg}
\int_\R (\mathcal L A_c)_y \int_{-\infty}^{y} \Lambda Q_c = \int_\R (\tilde F_1 +\la F_2) \int_{-\infty}^{y} \Lambda Q_c .
\ee
Integrating by parts, we get
$$
(\mathcal L A_c) \int_{-\infty}^{y} \Lambda Q_c \Big|_{-\infty}^{+\infty} +  \int_\R (\mathcal L A_c)_y \int_{-\infty}^{y} \Lambda Q_c = -\int_\R \Lambda Q_c \mathcal L A_c  =\int_\R Q_c A_c =0.
$$
Using (\ref{Ac}), we have $(\mathcal L A_c) \int_{-\infty}^{y} \Lambda Q_c \Big|_{-\infty}^{+\infty} =0$. Therefore, from (\ref{F1new}), 
$$
-f_2 \int_\R Q_c \Lambda Q_c = \frac{a'}{a}\int_\R \Big[ p  c (c -  \frac \la{\la_0 } ) \Lambda Q_c   - \frac 1{m-1} (c-\la)Q_c +(yQ_c^m)' \Big]\int_{-\infty}^y \Lambda Q_c.
$$
A simple computation, using Lemma \ref{IdQ}, gives us
\bee
- \theta  f_2 c^{2\theta -1} \int_\R Q^2  & = & \frac{a'}{2a} \Big[  p  c (c -  \frac \la{\la_0 } ) \int_\R \Lambda Q_c   - \frac 1{m-1} (c-\la) \int_\R  Q_c \Big]\int_\R \Lambda Q_c \\
&  =&  \frac{a'}{2a c} \Big[  p   (c -  \frac \la{\la_0 } ) (\theta -\frac 14)    - \frac 1{m-1} (c-\la) \Big](\theta -\frac 14) c^{2\theta -\frac 12} (\int_\R  Q)^2.
\eee
Using that $p=\frac 4{m+3}$, $\la_0 = \frac{5-m}{m+3}$ and $\theta = \frac1{m-1}-\frac 14$, we finally obtain
\bee
f_2(t) & =& \frac{3-m}{5-m}(\frac{3c}{m+3} -\frac{\la}{5-m})\frac{a'(\ve \rho)}{\sqrt{c}a(\ve \rho)} \frac{(\int_\R Q)^2}{\int_\R Q^2} \\
& = & \frac{3-m}{(5-m)^2}(3\la_0 c -\la)\frac{a'(\ve \rho)}{\sqrt{c}a(\ve \rho)} \frac{(\int_\R Q)^2}{\int_\R Q^2},
\eee
as desired (cf. (\ref{f2}).)
\medskip

Now, let us prove (\ref{constants}). Indeed, from (\ref{A10}), integrating over $\R$ and using (\ref{LI}), we get
\be\label{Acminf}
2\beta_c c \sqrt{c}   = c A_c(-\infty)= \mathcal L A_c(+\infty) -\mathcal L A_c(-\infty)   = \int_\R(\tilde F_1 +\la F_2),
\ee
which gives the value of $\beta_c$.

\medskip

Let us now describe the dependence in $c$ of the solution $A_c$. From (\ref{F1new}) (see also Lemma 4.5 in \cite{Mu2}), one has
$$
\tilde F_1 (t,y) + \la \hat F_1(t,y) = c^{1/(m-1)+1}  \tilde F_1^0(\sqrt{c}y) +\la c^{1/(m-1)} \hat F^0_1(\sqrt{c}y), 
$$
where
\bee
\tilde F^0_1(x) & := & p  \Lambda Q   - \frac 1{m-1} Q +(yQ^m)'  - 3\la_0 \xi_m  Q', \\
\hat F^0_1(x)&  := &  - \frac {4}{5-m} \Lambda Q   + \frac 1{m-1} Q  + \xi_m Q'.
\eee
Moreover, Claim 3 in \cite{Mu2} allows to conclude that $A_c$ satisfies the following decomposition:
\be\label{ScaA}
A_c(y) = c^{1/(m-1) -3/2} [ c \tilde A^0(\sqrt{c}y) +  \la  \hat A^0(\sqrt{c}y)], 
\ee
with $ \tilde A^0,  \hat A^0$ bounded solutions of $(\mathcal L \tilde A^0)' =  \tilde F_1^0$ and $(\mathcal L \hat A^0)' =  \hat F_1^0$, respectively. Moreover, one has $ (\tilde A^0)',  (\hat A^0)' \in \mathcal Y$. Using this decomposition we have
$$
\partial_c A_c  = (\frac 1{m-1} -\frac 32) \frac 1c A_c  + \frac 1{2c} yA_c' + c^{1/(m-1) -3/2}  \tilde A^0(\sqrt{c}y).
$$
From this identity we see that $\partial_c A_c$ has the same behavior as $A_c$: it is bounded, it is not $L^2$-integrable, and satisfies $\lim_{+\infty} \partial_c A_c =0$, $\lim_{-\infty} \partial_c A_c \neq0$. The same result holds for $\partial_c^2 A_c$. 
\end{proof}

\medskip

\begin{rem}[Cubic case]
In the special case $m=3$, the algebra of  functions involved in the linear problem (\ref{A10}) is well understood, and it can be computed explicitly. Indeed, from (\ref{ScaA}) one has 
$$
A_c(y) = \tilde A^0(\sqrt{c} y) + \frac{c}{\la} \hat A^0(\sqrt{c} y),
$$
with
$$
\tilde A^0(s) := \frac 12 (1-Q^2)\int_s^{+\infty} \!\!Q -\frac 1{12}y^2Q' -\frac 23 yQ +Q'\ln Q + \tilde \mu_0 Q',  
$$
and
$$
\hat A^0(s) := -\frac 12 (1-Q^2) \int_s^{+\infty} \!\! Q  + \frac 1{4}y^2Q' + \frac 12 yQ +Q'\ln Q + \hat \mu_0 Q'.  
$$
See Appendix \ref{AidQ} for the main ingredients of the proof of this result.  In particular, we have $\lim_{-\infty}A_c = \frac 12  (1-\frac \la c) \int_\R Q$, which is different from zero provided $c(t) \neq \la$. Finally, the constants $\tilde \mu_0$ and $\hat \mu_0$ are chosen such that 
$$
\int_\R yQ_c A_c(y) =0.
$$
\end{rem}

\medskip

\noindent
{\bf Step 6. Cubic case. resolution of a second linear system.} Since $f_2(t)\equiv 0$ in the case $m=3$ (cf. (\ref{f2})), we need to go beyond in our computations and solve a new linear system, in order to find a formal defect in the solution. From (\ref{tStu}), one has to consider a linear problem  for the unknown function $B_c(t,\cdot)$, with fixed time $t$, and with source term \emph{\bf non localized}. The next result gives the existence of such a second order correction term.

\medskip

\begin{lem}[Existence of a second order correction term]\label{2sys}~

Let $f_3(t), f_4(t)$ be given by (\ref{f3})-(\ref{f4}), and consider $\tilde F_2$ as in (\ref{F2}). For each fixed time $t\in [-T_\ve, \tilde T_\ve]$, there exists a unique solution $B_c(t,\cdot)$ of
\be\label{O2}
(\mathcal L B_c)_y = \tilde F_2(t,y),
\ee
satisfying, for all $t\in [-T_\ve,\tilde  T_\ve]$,
\be\label{O33}
\int_\R Q_c B_c = \int_\R yQ_c B_c=0.
\ee
In addition, one has, for some $\ga>0$ independent of $\ve$,
\be\label{limB}
\begin{cases}
 |B_c(t,y)|+|\partial_cB_c(t,y)| \leq K e^{-\ga y} e^{-\ve\ga |\rho(t)|}, \quad \hbox{ as } \; y\to +\infty, \\
 |B_c(t,y)|+|\partial_cB_c(t,y)| \leq K|y| e^{-\ve\ga |\rho(t)|}, \quad \hbox{ as } \; y\to -\infty.
\end{cases}
\ee
\end{lem}

\begin{proof}
The proof is divided in several steps.

\medskip

\noindent
1. Note that  since $A_c\in L^\infty(\R)$, one has from (\ref{F2}) that  $\tilde F_2(t,\cdot) \in L^\infty (\R)$. From Lemma \ref{surL} (see also \cite{Mu2}), we get solvability in $S'(\R)$ for (\ref{O2}) provided $\tilde F_2$ satisfies the orthogonality condition
\be\label{OR}
\int_\R \tilde F_2 Q_c =0.
\ee
Let us prove this last identity. Indeed, we have\footnote{For the sake of simplicity, we avoid the explicit dependence on time in this computation.}
\bee
\int_\R Q_c \tilde F_2 & = & \int_\R Q_c( f_1 d\partial_c A_c + 3d \frac{a'}{a} (yQ_c^2A_c)_y + \frac{f_3}{\tilde a} \Lambda Q_c  + 3d^2 a (Q_cA_c^2)_y) \\
& =& -f_1 d \int_\R \Lambda Q_c A_c - 3d \frac{a'}{a} \int_\R yQ_c^2Q_c' A_c + \frac{f_3}{\tilde a} \int_\R Q_c \Lambda Q_c   - 3d^2 a \int_\R Q_c Q_c' A_c^2 \\
& =& -\frac{a'^2}{a^{5/2}} \big[ \frac 13 (c-3\la)  \int_\R  yQ_c' A_c +  3 \int_\R yQ_c^2Q_c' A_c +3  \int_\R Q_c Q_c' A_c^2 \big]  + \frac{f_3}{a^{1/2}} \int_\R Q_c \Lambda Q_c.  
\eee
Let us define
$$
\mu_c := \frac 13 (c-3\la)  \int_\R  yQ_c' A_c +  3 \int_\R yQ_c^2Q_c' A_c +3  \int_\R Q_c Q_c' A_c^2.
$$
Our objective is to give a simple expression of this quantity. Indeed, first note that $A_c'\in \mathcal Y$. From the equation $(\mathcal L A_c)' =\tilde F_1 + \la \hat F_1$, one has $ \mathcal L A_c' = \tilde F_1 + \la \hat F_1 + 6Q_cQ_c' A_c$.\footnote{Let us recall that $m=3$.} We multiply this identity by $A_c$ and integrate over $\R$. We get
\be\label{IntALA}
\int_\R A_c \mathcal L A_c'  = \int_\R A_c(  \tilde F_1 + \la \hat F_1) + 6 \int_\R Q_cQ_c' A_c^2.
\ee
On the other hand, after integration by parts, one has
\bee
\int_\R A_c \mathcal L A_c'  & = &  \int_\R A_c' \mathcal L A_c = A_c \mathcal L A_c |_{-\infty}^{+\infty} -\int_\R A_c (  \tilde F_1 + \la \hat F_1) \\
& = & -cA_c^2(-\infty)- \int_\R A_c (  \tilde F_1 + \la \hat F_1).
\eee
From these two identities, we get
$$
3 \int_\R Q_cQ_c' A_c^2 = -\frac 12 cA_{c}^2(-\infty) -\int_\R  A_c (  \tilde F_1 + \la \hat F_1).
$$
We replace this identity above, in the definition of $\mu_c$, to obtain (recall that $A_c$ is orthogonal to $Q_c$ and $Q_c'$)
\bee
\mu_c & =& -\frac 12 cA_{c}^2(-\infty) +  \frac 13 (c-3\la)  \int_\R  yQ_c' A_c +  3 \int_\R yQ_c^2Q_c' A_c    -\int_\R  A_c ( \frac 13  cy Q_c'   +  (yQ_c^3)'   - \la   yQ_c'  )\\
& =& -\frac 12 cA_{c}^2(-\infty)  -\int_\R A_c Q_c^3.
\eee
On the other hand, note that $\mathcal L(-\frac 12 Q_c) = Q_c^3$. We have then
\bee
\mu_c &= & -\frac 12 cA_{c}^2(-\infty)  +\frac 12 \int_\R \mathcal L A_c Q_c = -\frac 12 cA_{c}^2(-\infty)  - \frac 12 \int_\R Q_c\int_y^{+\infty} (\tilde F_1 + \la \hat F_1) \\
& =& -\frac 12 cA_{c}^2(-\infty) - \frac 14 \int_\R Q_c \int_\R (\tilde F_1 + \la \hat F_1),
\eee
(recall that $f_2\equiv 0$.) A simple computation gives 
$$
 \int_\R (\tilde F_1 + \la \hat F_1) = -\frac 12(c-\la)\int_\R Q_c,
$$
since from (\ref{Acminf}) and (\ref{F1new}) one has
$$
A_c(-\infty) = -\frac 1c\int_\R (\tilde F_1 + \la \hat F_1)   = \frac 1{2c}(c-\la)\int_\R Q_c,
$$
we finally get
$$
\mu_c = \frac\la{8c} (c-\la)(\int_\R Q)^2\neq 0.
$$
From the definition of $f_3(t)$ in (\ref{f3}), we get finally (\ref{OR}). In consequence, there exists at {\bf least one solution} $B_c\in S'$ satisfying (\ref{O2}). 

\medskip

\noindent
2. Let us look for a solution $B_c$ with a special behavior. In fact, we will search for a solution with the following structure:
$$
B_c(t,y) = \tilde B_c(t,y) + f_5(t) Q_c'(y) + \frac{f_4}{a^{1/2}} (t)\Lambda Q_c(y),
$$
where $\tilde B_c$ has the following decomposition
\be\label{tBct}
\tilde B_c(t,y) = \al_1(t) \int_y^{+\infty} \!\! A_c +\al_2(t) \int_y^{+\infty}\!\! \partial_c A_c + \al_3(t) +\hat B_c(t,y),  \qquad \hat B_c(t,\cdot )\in \mathcal Y.
\ee
Here $\al_1(t), \al_2(t), \al_3(t)$ are real valued, exponentially decreasing, time-dependent functions, to be found. Note that this function satisfies (\ref{limB}), provided $\al_3(t)\equiv 0$, since $A_c(y), \partial_c A_c(y), \partial_c^2 A_c(y) \to 0$ as $y\to +\infty$, at exponential rate. Moreover, we can choose unique $f_4(t),f_5(t)\in \R$ such that (\ref{O33}) holds, respectively, for all time $t\in [-T_\ve, \tilde T_\ve]$.
 
\smallskip

Let us prove the existence of $\hat B_c$, with the desired properties. By replacing the form (\ref{tBct}) in (\ref{O2}), we get
\bee
 (\mathcal L \hat B_c)_y  & = & -\al_1 \big[ \mathcal L (\int_y^{+\infty} \!\! A_c)\big]_y - \al_2 \big[\mathcal L (\int_y^{+\infty}\!\! \partial_c A_c)\big]_y + 3\al_3 (Q_c^2)'  \\
& & + \ (\frac{a'}{a^{3/2}})' (c-\la) A_c + f_1 \frac{a'}{a^{3/2}}\partial_c A_c \\
& & + \ 3 \frac{a'^2}{a^{5/2}} (yQ_c^2A_c)_y + \frac{f_3}{ a^{1/2}} \Lambda Q_c  +\frac{a''}{2a^{3/2}} (y^2 Q_c^3)_y +  3\frac{a'^2}{a^{5/2}} (Q_cA_c^2)_y.
\eee
On the other hand, one has
$$
\big[\mathcal L ( \int_y^{+\infty} \!\! A_c) \big]_y = -c A_c + A_c''  -3\big[Q_c^2 \int_y^{+\infty} \!\! A_c \big]_y, 
$$
and
$$
\big[\mathcal L ( \int_y^{+\infty} \!\! \partial_c A_c)\big]_y =-c \partial_cA_c + (\partial_c A_c)''  -3\big[Q_c^2 \int_y^{+\infty} \!\! \partial A_c \big]_y.
$$
Therefore, by defining 
$$
\al_1(t) := -(\frac{a'}{a^{3/2}})'(\ve \rho(t)) (1  - \frac{\la}{c(t)}), \quad \hbox{ andÊ} \quad  \al_2(t) := -\frac 1{c(t)} f_1(t) \frac{a'}{a^{3/2}}(\ve \rho(t));
$$
(note that both functions are exponentially decreasing in $|\rho(t)|$), one has that $\hat B_c(t,y)$ must be a solution of
\bee
\mathcal L \hat B_c  & =  &  3 \frac{a'^2}{a^{5/2}} yQ_c^2A_c + \frac{f_3}{2c a^{1/2}} y Q_c  +\frac{a''}{2a^{3/2}} y^2 Q_c^3 +  3\frac{a'^2}{a^{5/2}} Q_cA_c^2 \\
& & +\ \al_1 \big[ A_c' -3Q_c^2\int_y^{+\infty} \!\! A_c   \big] + \al_2 \big[ (\partial_c A_c)'-3Q_c^2 \int_y^{+\infty} \!\! \partial_c A_c  \big] + 3\al_3 Q_c^2. 
\eee
Note that the right hand side above is in $\mathcal Y$ and it is orthogonal to $Q_c'$, since there exists a solution $B_c$ of (\ref{O2}). Therefore, from Lemma \ref{surL}, we have $\hat B_c(t,\cdot )\in \mathcal Y$, with $\|\hat B_c(t,\cdot)\|_{L^\infty(\R)}\leq Ke^{-\ve \ga |\rho(t)|} + K|\al_3(t)|$. Let us adjust the value of $\al_3(t)$. Indeed, first note that 
$$
(\mathcal L(c-Q_c^2))_y =0.   \qquad \hbox{(cf. Lemma \ref{IdQ2} in Appendix \ref{AidQ} below.)}
$$
Therefore, by substracting a suitable ponderation of the term $c- Q_c^2$ in the form of $\tilde B_c$ above (see (\ref{tBct})), we {\bf may suppose} $\al_3(t)\equiv 0$, still having $\hat B_c \in \mathcal Y$. This proves the existence of $B_c$ with the required behavior.

\medskip

\noindent
3. Finally, let us prove that $f_4(t)$ has the form (\ref{f4}). Indeed, note that
\be\label{nulo}
\int_\R B_c Q_c = -\int_\R B_c \mathcal L \Lambda Q_c = -\int_\R \Lambda Q_c \mathcal L B_c ;
\ee
From one has for $y<r$,  $ \mathcal L B_c (y)=\mathcal LB_c(r) - \int_y^{r} \tilde F_2$
therefore
$$
(\ref{nulo}) = -\int_\R \Lambda Q_c (\mathcal LB_c(r) - \int_y^{r} \tilde F_2) = \int_\R  \Lambda Q_c \int_y^{r} \tilde F_2.
$$
From the definition of $\Lambda Q_c$, we get
$$
(\ref{nulo}) = \frac{1}{2c}\int_\R yQ_c \tilde F_2.
$$
Now we use the definition of $F_2$ and the orthogonality conditions on $A_c$ to get
\bee
2c (\ref{nulo}) & = &   
- \frac 13 (c-3\la)\frac{a'^2}{a^{5/2}}  \int_\R y^2 Q_c'  A_c  - 3 \frac{a'^2}{a^{5/2}}  \int_\R (yQ_c)' yQ_c^2A_c    \\
& &  - \frac{a''}{8a^{3/2}} \int_\R y^2 Q_c^4 + \frac{f_4(t)}{2a^{1/2}} \int_\R Q_c^2 -3\frac{a'^2}{a^{5/2}} \int_\R (yQ_c)' Q_cA_c^2 . 
\eee
Now we use the scaling property (\ref{ScaA}) of the function $A_c$, with $m=3$, to obtain a better description of $f_4(t)$: we have
\bee
 2c (\ref{nulo}) &  = &  - \frac 13 (c-3\la)\frac{a'^2}{\sqrt{c}a^{5/2}}  \int_\R y^2 Q'   [  \tilde A^0 +  \frac \la c  \hat A^0]   - 3 \frac{\sqrt{c} a'^2}{a^{5/2}}  \int_\R (yQ)' yQ^2 [  \tilde A^0(y) +  \frac \la c  \hat A^0(y)] \\
& & \quad  - \frac{a'' \sqrt{c}}{8a^{3/2}} \int_\R y^2 Q^4 + \frac{f_4(t)}{2a^{1/2}} \sqrt{c} \int_\R Q^2 -3\frac{\sqrt{c}a'^2}{a^{5/2}} \int_\R (yQ)' Q [  \tilde A^0(y) +  \frac \la c  \hat A^0(y)]^2 . 
\eee
Therefore, one has $f_4(t)$ as in (\ref{f4}), with
$$
f_4^2(t) := \frac 1{8M[Q]} \int_\R y^2 Q^4>0, 
$$
and
\bee
 f_4^1 (t)&  := & \frac 1{3M[Q]} (1-3\frac \la c)\int_\R y^2 Q'   [  \tilde A^0 +  \frac \la c  \hat A^0] +\frac 3{M[Q]}   \int_\R (yQ)' yQ^2 [  \tilde A^0 +  \frac \la c  \hat A^0]    \\
& & \qquad + \frac{3}{M[Q]} \int_\R (yQ)' Q [  \tilde A^0 +  \frac \la c  \hat A^0]^2.
\eee

\medskip

In conclusion, we have the existence of a unique $B_c(t,y)$ satisfying (\ref{O2})-(\ref{O33}). Estimates (\ref{limB}) are direct from (\ref{tBct}).
\end{proof}

\noindent
{\bf Step 7. Final conclusion.} Having solved one linear problem in the cases $m=2$ and $4$, and two linear equations 
in the case $m=3$, from (\ref{Stt}) and (\ref{tStu}) we have
\bee
S[\tilde u](t,x) & = &   (c'(t) - \ve f_1(t) -\ve^2 \delta_{m,3} f_3(t))\partial_c\tilde u  \nonu\\
& & + (\rho'(t) -c(t)+ \la -  \ve f_2(t) -\ve^2 \delta_{m,3}f_4(t)) \partial_\rho \tilde u + \tilde S[\tilde u](t,x), 
\eee
with 
$$
\partial_\rho \tilde u := \partial_\rho R -w_y + O(\ve^2 e^{-\ve \ga |\rho(t)|}|A_c|),
$$
and
$$
\tilde S[\tilde u]  = 
\begin{cases}
(\ref{a10}), \quad \hbox{ for  the cases } m=2,4; \\
(\ref{a11}) + (\ref{a12}), \quad \hbox{ in the cubic case.}
\end{cases}
$$
This proves the first part of Proposition \ref{prop:decomp}.

\medskip

In addition, from Lemmas \ref{lem:omega} and \ref{2sys} we have (\ref{Ac}) and (\ref{Bc}), respectively. This proves the second part of Proposition \ref{prop:decomp}. In addition, from these lemmas,  $f_1(t)$, $f_2(t)$, $f_3(t)$ and $f_4(t)$ are well determined. This proves the third part of Proposition \ref{prop:decomp}.

\medskip

Finally, we prove the last part of Proposition \ref{prop:decomp}. Let us recall that (\ref{a10}) is a bounded, non localized term, and (\ref{a11}) $+$ (\ref{a12}) is a polynomially growing  term. Indeed, from (\ref{dd}), (\ref{f1}), (\ref{f3}), (\ref{f4}), (\ref{Acy}) and (\ref{limB}) we have
$$
|(\ref{a11})| \leq K \ve^3 |y| e^{-\ve\ga |\rho(t)|}, \quad \hbox{ as } y\to -\infty , \qquad |(\ref{a11})|\to 0 \quad \hbox{ as } y\to +\infty,
$$
and
$$
\begin{cases}
|(\ref{a12})| \leq K \ve^4(1+ |y| + \ve^2|y|^2) e^{-\ve\ga |\rho(t)|}, & \hbox{ as } y\to -\infty ,\\
|(\ref{a12})|\to 0, & \hbox{ as } y\to +\infty.
\end{cases}
$$
Moreover, note that
$$
\|(\ref{a10})\|_{H^1(y\geq -\frac 3\ve)} \leq K \ve^{3/2} e^{-\ve\ga |\rho(t)|};
$$
 (cf. \cite{Mu2} for this bound) and
$$
\|(\ref{a11})\|_{H^1(y\geq - \frac 3\ve)} \leq K \ve^{3/2} e^{-\ve\ga |\rho(t)|}, \qquad \|(\ref{a12})\|_{H^1(y\geq - \frac 3\ve)} \leq K \ve^{5/2} e^{-\ve\ga |\rho(t)|}.
$$

\medskip

Finally, (\ref{Stilde}) is direct from this last estimate. On the other hand, from (\ref{a10}) one has (\ref{SIn}), and from (\ref{a11})-(\ref{a12}) we finally obtain (\ref{SIn3}).

The proof of Proposition \ref{prop:decomp} is now complete.

\bigskip

\section{Some identities related to the soliton $Q$}\label{AidQ}

This section has been taken from Appendix C in \cite{MMcol1}.

\begin{lem}[Identities for the soliton $Q$]\label{IdQ}~

Suppose $m>1$ and denote by $Q_c := c^{\frac 1{m-1}} Q(\sqrt{c} x)$ the scaled soliton. Then
\begin{enumerate}
\item \emph{Energy}.
$$
E_1[Q]= \frac 12 (\la - \la_0)\int_\R Q^2 =(\la-\la_0)M[Q],  \quad \hbox{with }\; \la_0 = \frac{5-m}{m+3}.
$$
\item \emph{Integrals}. Recall $\theta = \frac 1{m-1} -\frac 14$. Then
$$
\int_\R Q_c = c^{\theta-\frac 14} \int_\R Q, \quad \int_\R Q_c^{2} = c^{2\theta} \int_\R Q^2, \quad E_1[Q_c] = c^{2\theta}(\la -\la_0 c)M[Q],
$$
and finally
$$
\int_\R Q_c^{m+1} = \frac{2(m+1)c^{2\theta +1}}{m+3} \int_\R Q^2, \quad \int_\R \Lambda Q_c = (\theta -\frac 14) c^{\theta -\frac 14 } \int_\R Q, 
$$
$$
 \int_\R \Lambda Q_c Q_c =\theta c^{2\theta -1} \int_\R Q^2.
$$
\end{enumerate}
\end{lem} 

\medskip

\begin{lem}[Inverse functions, case $m=3$]\label{IdQ2}~

Let $\mathcal L_0$ be the fixed, linearized operator defined in (\ref{defLy}) for $m=3$. Then one has
$$
\mathcal L_0(Q')=0, \quad \mathcal L_0(yQ) = -2yQ^3-2Q', \quad \mathcal L_0(y^2 Q') = -4yQ+4yQ^3 -2Q';
$$
$$
\mathcal L_0(\int_y^{+\infty} Q )=(1-3Q^2)\int_y^{+\infty} Q +Q' , 
$$
$$
 \mathcal L_0(Q^2\int_y^{+\infty} Q) = -3Q^2\int_y^{+\infty} Q + 5Q^2Q',
$$
and
$$
\mathcal L_0(Q^2) =-3Q^2, \quad \mathcal L_0(Q'\ln Q) = -2Q'+\frac 52 Q^2Q'.
$$
\end{lem} 
The proof of these result is a lengthy but direct computation.

\end{document}

\endinput